\documentclass{article}

\usepackage[a4paper, left=2cm, right=2cm, top=2cm, bottom=2cm]{geometry}

\usepackage[english]{babel}
\usepackage[utf8]{inputenc}
\usepackage[T1]{fontenc}

\usepackage{amsmath}
\usepackage{amssymb}
\usepackage{amsthm}
\usepackage{stmaryrd}  

\usepackage{pdfpages}			
\usepackage[inline]{enumitem}	
\usepackage{array}				
\usepackage{multicol}			

\usepackage{wrapfig}			
\usepackage{caption}			
\usepackage{subcaption}			

\usepackage{tikz}
\usetikzlibrary{arrows}
\usetikzlibrary{bbox}			

\usepackage{hyperref}
\hypersetup{
	colorlinks=true,
	linkcolor=blue,
	hypertexnames=false, 
}

\usepackage{thm-restate}

\DeclareFontShape{T1}{cmr}{m}{scit}{<->ssub*cmr/m/scsl}{}

\newcommand{\NN}{\mathbb N}
\newcommand{\ZZ}{\mathbb Z}

\newcommand{\RR}{\mathbb R}

\setlist[itemize,1]{label={\footnotesize\textbullet}}
\setlist[itemize,2]{label=$\triangleright$}
\setlist[itemize,3]{label=$\circ$}

\newtheorem{theoreme}{Theorem}

\numberwithin{theoreme}{section}
\numberwithin{equation}{section}
\newtheorem{proposition}[theoreme]{Proposition}
\newtheorem{lemme}[theoreme]{Lemma}
\newtheorem{corollaire}[theoreme]{Corollary}

\theoremstyle{definition}
\newtheorem{definition}[theoreme]{Definition}
\newtheorem{exemple}[theoreme]{Example}
\theoremstyle{remark}
\newtheorem*{remarque}{Remark}

\newcommand{\ddelta}{\boldsymbol{\delta}}	
\newcommand{\bbeta}{\boldsymbol{\beta}}		
\newcommand{\pprec}{{\boldsymbol{\prec}}}	

\newcommand{\affc}[1]{\boldsymbol{\left(\vphantom{#1}\right.}\!#1\!\boldsymbol{\left.\vphantom{#1}\right)}}
\SetSymbolFont{stmry}{bold}{U}{stmry}{m}{n}			
\SetSymbolFont{largesymbols}{bold}{OMX}{txex}{b}{n}	

\newcommand{\Lcn}{\overline{L_c}}
\newcommand{\Rcn}{\overline{R_c}}

\newcommand{\diffsym}{\mathrel{\Delta}}

\newcommand{\ncad}{\widetilde{\smash{\mathrm{ncad}}}}

\begin{document}


\begingroup
	\def\nom{Abou~Yassin}
	\def\prenom{Jad}
	\def\titre{A generalization in affine type \texorpdfstring{$A$}{A} of \textsc{Coxeter} sortable elements and \textsc{Reading}'s bijection with noncrossing partitions}

	\def\NOM{\expandafter\MakeUppercase{\nom}}

	\hypersetup{
		pdfcreator    = {\LaTeX{}},
		pdfauthor     = {\prenom{}~\NOM{}},
		pdftitle      = {\titre{}},
		pdfsubject    = {},	
		pdfkeywords   = {}	
	}

	\begin{center}
		{\bfseries\Large \titre{}}
		\vspace*{\baselineskip}

		\prenom{}~\textsc{\nom{}}
	\end{center}
\endgroup


\begin{abstract}
	This paper generalizes in the affine symmetric group the notion of \textsc{Coxeter} sortable (or $c$-sortable for short) elements, as well as the classical bijection between $c$-sortable elements and $c$-noncrossing partitions defined by \textsc{Reading} in finite \textsc{Coxeter} groups.
	The generalization to the affine symmetric group of the $c$-sortable elements is achieved by using biclosed sets of reflections. Using recent works from \textsc{Barkley} and \textsc{Speyer}, these biclosed sets admit a sort of «~one-line notation~» called a TITO on $\ZZ$ (translation-invariant total order on $\ZZ$) that coincides with the usual one-line notation in the case of an affine permutation. We characterize the $c$-sortable elements of the affine symmetric group by pattern avoidance on their one-line notation, mirroring the well-known characterizations of $c$-sortable elements in the classical finite types. Based on this criterion, we then define the $c$-sortable biclosed sets, generalizing the $c$-sortable elements, as biclosed sets such that their TITOs on $\ZZ$ avoid certain patterns.
	We also build a bijection from our set of $c$-sortable biclosed sets to the set of $c$-noncrossing partitions using various combinatorial objects and their one-to-one correspondences. First, the TITOs, in bijection with the biclosed sets of the affine symmetric group using results from \textsc{Barkley} and \textsc{Speyer}. Second, the $c$-noncrossing partitions of an annulus, in bijection with the $c$-noncrossing partitions of the affine symmetric group, using results from \textsc{Digne} and \textsc{Reading}. Finally, the cyclic noncrossing arc diagrams, defined by \textsc{Barkley}, for which we exhibit a subset in bijection with both the set of $c$-sortable biclosed sets and the set of $c$-noncrossing partitions.
\end{abstract}


\tableofcontents

\newpage


\section*{Introduction}
	\addcontentsline{toc}{section}{Introduction}

	\paragraph*{Context} The \textsc{Catalan} numbers are famously known for their appearance in various and yet seemingly unrelated locations in mathematics \cite{stanley2015}. Among them, a notable example is the set of noncrossing partitions \cite{kreweras1972}. Associating to each part $\{a_1 < \dots < a_k\}$ of a noncrossing partition the cyclic permutation $(a_1, \dots, a_k)$ of the symmetric group $\mathfrak{S}_n$, we can define the notion of noncrossing partitions of $\mathfrak{S}_n$. It happens that these elements correspond exactly to the interval between the identity and $c = (1, 2, \dots, n-1, n) \in \mathfrak{S}_n$ for an ordering on $\mathfrak{S}_n$ called the absolute order \cite{biane97}.

	The symmetric group is a special case of \textsc{Coxeter} group, for which the absolute order is also defined. Considering $c$ a \textsc{Coxeter} element of a \textsc{Coxeter} group $W$, that is the product of all the generators of $W$ exactly once in a chosen order, we define the $c$-noncrossing partitions of $W$ as the interval between the identity and $c$ for the absolute order. When $W$ is finite, they play a major role in the study of \textsc{Artin} groups \cite{bessis2002dual, brady2002} and in the field of \textsc{Coxeter}-\textsc{Catalan} combinatorics \cite{reiner1997, armstrong2009generalized}.

	Another particular object appearing in \textsc{Coxeter}-\textsc{Catalan} combinatorics is the set of \textsc{Coxeter} sortable elements (or $c$-sortable elements) of a finite \textsc{Coxeter} group $W$. Once a \textsc{Coxeter} element $c$ is chosen, an element of $W$ has a particular reduced expression called its $c$-sorting word. The $c$-sortable elements are the ones for which their $c$-sorting words satisfy some compatibilities with the order in which the generators appear in $c$. They correspond, once a \textsc{Coxeter} element $c$ is chosen, to all the elements of $W$ that have a reduced expression that satisfies some compatibilities with the order in which the generators appear in $c$. These elements were introduced by \textsc{Reading} \cite{reading2005clusters} and were a major tool for obtaining bijective correspondences between various sets enumerated by the \textsc{Coxeter}-\textsc{Catalan} numbers, generalizing the \textsc{Catalan} numbers: $c$-noncrossing partitions, cambrian lattices, generalized associahedrons, cambrian fans and noncrossing arc diagrams (see \cite{reading2005clusters, reading2007cambrianlattice, reading2015noncrossingarcdiagramscanonical}). In type $A$, the case of the symmetric group $\mathfrak{S}_n$, the $c$-sortable permutations have a very explicit combinatorial interpretation as permutations whose one-line notation avoids certain patterns \cite{reading2005clusters}. This description is related to the sylvester congruence on words and binary search trees \cite{hivert2004algebra} and more generally to cambrian trees \cite{Chatel_2017}.

	The \textsc{Coxeter} sortable elements, the $c$-noncrossing partitions as well as \textsc{Reading}'s bijection between these two families of objects which is denoted $nc_c$ are still defined for infinite \textsc{Coxeter} groups. However, unlike the finite case, they form infinite sets and the map $nc_c$ fails to be a bijection: it is only injective \cite{reading2010sortable}. The aim of this paper is to generalize the $c$-sortable elements of the affine symmetric group $\widehat{\mathfrak{S}}_n$ in order to define a natural bijection with the $c$-noncrossing partitions, extending \textsc{Reading}'s map $nc_c$.

	This generalization happens in the set of biclosed sets of the positive root system of type $\widetilde{A}$. The framework of biclosed sets proved to be quite efficient at extending some properties of finite \textsc{Coxeter} groups for infinite \textsc{Coxeter} groups \cite{Dyer_2019, barkley2024affineextendedweakorder, barkley2025affinetamarilattice}, as biclosed sets naturally generalize the notion of inversion sets of elements of $W$. In the affine cases, there is a classification of biclosed sets \cite{barkley2022combinatorial} and specifically in type $\widetilde{A}$ (the affine symmetric group), there is a combinatorial description as translation-invariant total orders on $\ZZ$ (TITOs on $\ZZ$) and cyclic noncrossing arc diagrams \cite{barkley2025extendedweakorderaffine}.

	TITOs on $\ZZ$ are total orders $\prec$ on $\ZZ$ such that for all $x,y \in \ZZ$, $x \prec y$ if and only if $x + n \prec y + n$, for some fixed positive integer $n$. An explicit subset of them is in bijection with biclosed sets of the affine symmetric group \cite{barkley2022combinatorial, barkley2024affineextendedweakorder} in such a way that the TITO corresponding to the inversion set of $w \in \widehat{\mathfrak{S}}_n$ is given by the one-line notation of $w$.

	Noncrossing arc diagrams (on $n$ points) are diagrams encoding the cover reflections of a permutation (two consecutive integers of the one-line notation that are in decreasing order) and are in a one-to-one correspondence with the symmetric group of rank $n$. They play a fundamental role in understanding lattice quotients of the weak order on $\mathfrak{S}_n$ \cite{reading2015noncrossingarcdiagramscanonical} and $c$-sortable elements can be defined as minimal representatives of the cambrian congruence of the weak order \cite{reading2007cambrianlattice}. A generalized version of the noncrossing arc diagrams for the affine symmetric group and the extended weak order is defined by \textsc{Barkley} and shares many of the properties of the noncrossing arc diagrams \cite{barkley2025extendedweakorderaffine}.

	\paragraph*{Outline of the paper} We summarize briefly the outline of this paper before going into more depth in the following paragraphs. The main goal of this paper is to generalize the $c$-sortable elements of the affine symmetric group (a \textsc{Coxeter} group of type $\widetilde{A}$) in such a way that \textsc{Reading}'s map $nc_c$ naturally extends to a bijection with the $c$-noncrossing partitions (Section \ref{sec: biclos c-triable et généralisation de l'application de Reading}). Understanding the $c$-sortable elements of the affine symmetric group is crucial in order to generalize them, so we establish a characterization of $c$-sortable elements of the affine symmetric group based only on the appearance of patterns in the one-line notation of an affine permutation (Section \ref{sec: c triable du groupe symmetrique affine}).

	Taking inspiration from other generalizations from finite to affine \textsc{Coxeter} groups, for example the extended weak order \cite{barkley2024affineextendedweakorder, barkley2025extendedweakorderaffine}, we define the generalization of $c$-sortable elements in the set of biclosed sets of the positive root system of $\widehat{\mathfrak{S}}_n$, or equivalently, in the set of translation-invariant total orders on $\ZZ$ (TITOs on $\ZZ$ for short) \cite{barkley2022combinatorial, barkley2025extendedweakorderaffine}. This is motivated by the fact that biclosed sets form a natural generalization of inversion sets \cite{Dyer_2019}. TITOs on $\ZZ$ also have a notion of «~one-line notation~» for which we adapt and generalize the characterization based on pattern avoidance obtained for the $c$-sortable elements of $\widehat{\mathfrak{S}}_n$.

	In the noncrossing partition side of things, we use planar models as noncrossing partitions of annuli \cite{digne2005presentations, reading2024noncrossingpartitionsannulus}, and take inspiration from the results in the finite type $A$ \cite[6.3]{gobet2018dual}. It happens that \cite{barkley2025extendedweakorderaffine} also provides a third framework: the cyclic noncrossing arc diagrams, generalizing the noncrossing arc diagrams in the finite types \cite{reading2015noncrossingarcdiagramscanonical}. They are not naturally in bijection with neither the $c$-sortable elements or the $c$-noncrossing partitions, but we can easily define a subset of them, the cyclic $c$-noncrossing arc diagrams, to suit our needs. Finally, combining all these bijections together, we can build the generalization of \text{Reading}'s map $nc_c$, denoted $nc_c^{\widetilde{A}}$ (Section \ref{sec: generalization of reading's bijection}). Figure \ref{fig: diagramme des bijections} shows the main maps and objects that appear in this paper.

	\begingroup 
	\begin{figure}[ht!]
		\centering
		\newcommand{\biclos}{\begin{tikzpicture}[scale=0.8, racine/.style={inner sep=0.7pt, draw, circle, fill}, baseline=(current bounding box.center), rounded corners=0pt]
				\foreach \i in {0,1,2}{
					\coordinate (a\i) at ({-30 + 120*\i}:1);
					\pgfmathsetmacro{\j}{div(mod(\i+1,3),1)}
					\coordinate (a\i\j) at ({150 + 120*\i}:{1/2});
					\draw[draw=none] (0,0) -- node[pos=1.25] {$\alpha_\i$} (a\i);
					\draw[very thin] (a\i) -- (a\i\j);
				}

				\draw[draw=none, fill=gray, opacity=0.3] (a0) -- (a12) -- (a20) -- cycle;
				\node[racine] at (a0) {};
				\node[racine] at (a12) {};
				\node[racine] at (a20) {};
				\node[racine] at (-30:{1/4}) {};

				\draw (a0) -- (a1) -- (a2) -- cycle;
		\end{tikzpicture}}

		\newcommand{\ncadtikz}{\begin{tikzpicture}[scale=0.8, bezier bounding box, baseline=(current bounding box.center)]
				\def\n{3}
				\pgfmathsetmacro{\nm}{\n-1}

				\def\Rc{0}
				\def\Lc{1,2}

				\newcommand{\ang}[1]{{-(##1)*360/\n + 90}}

				\foreach \i in {0,...,\nm}{
					\coordinate (\i) at (\ang{\i}:1);
				}

				\foreach \i in \Rc{
					\draw (0,0) -- (\i);
				}
				\foreach \i in \Lc{
					\draw (\ang{\i}:2) -- (\i);
				}

				\node[fill, circle, inner sep=1pt] at (0,0) {};

				\draw[thick] (0) to[out=-45, in=0, looseness=2] (2) to[out=80, in=100, looseness=4] (1);

				\foreach \i in {0,...,\nm}{
					\node[circle, fill=white, inner sep=0pt] at (\i) {\i};
				}
		\end{tikzpicture}}

		\newcommand{\ncgeom}{\begin{tikzpicture}[bezier bounding box, scale=0.6]
				\def\n{3}
				\def\Lc{1,2}
				\def\Rc{0}

				\def\r{0.7}
				\def\R{2}

				\newcommand{\ang}[1]{{-(##1)*360/\n + 90}}
				\newcommand{\dis}[1]{{##1*\R + (1-##1)*\r}}

				\draw (0,0) circle (\dis{0});
				\draw (0,0) circle (\dis{1});

				\foreach \i in \Lc{
					\coordinate (\i) at (\ang{\i}:\dis{1});
					\draw[draw=none] (0,0) -- node[pos=1.2] {\i} (\i);
				}
				\foreach \i in \Rc{
					\coordinate (\i) at (\ang{\i}:\dis{0});
					\draw[draw=none] (0,0) -- node[pos=0.5] {\i} (\i);
				}

				\pgfmathparse{\n-1}	
				\foreach \i in {0,...,\pgfmathresult}{
					\node[draw, circle, fill, inner sep=1pt] at (\i) {};
				}

				\draw[fill, fill opacity=0.2]
				(0)	to[out=0,in=-50, looseness=3] (2)
				to[out=100, in=80, looseness=2.8] (1)
				to[out=90, in=120, out looseness=3.2, in looseness=2] ({\ang{2}}:\dis{0.3})
				to[out=-60, in=-5, looseness=3.2] (0);
		\end{tikzpicture}}

		\begin{tikzpicture}[scale=3]
			\node (sortc) at (-1,1) {$\mathrm{Sort}_c(\widehat{\mathfrak{S}}_n)$};
			\node (NCWc) at (1,1) {$\mathrm{NC}(\widehat{\mathfrak{S}}_n,c)$};
			\node (cBicW) at (-1,0) {$c$-$\mathrm{Bic}(\widehat{\mathfrak{S}}_n)$};
			\node (cBic) at (-1,-0.5) {$c$-$\mathrm{Bic}(\Phi^+)$};
			\node (cTITO) at (0,-0.5) {$c$-$\mathrm{TITO}$};
			\node (cNCAD) at (1,0) {$c$-$\widetilde{\mathrm{NCAD}}$};
			\node (geom) at (2.5,1) {$\widetilde{\mathrm{NC}}_c^{A,\circ}$};

			\draw[right hook->] (sortc) -- node[pos=0.5, above] {$nc_c$} (NCWc);
			\draw[right hook->] (sortc) -- node[pos=0.5, left] {$N$} (cBicW);
			\draw[right hook->] (sortc) -- node[pos=0.75, sloped, above] {\tiny one-line notation} (cTITO);

			\draw[left hook->, transform canvas={shift={(0.05,-0.05)}}] (cTITO)	-- node[pos=0.5, below, sloped] {$\ncad_c$} (cNCAD);
			\draw[densely dashed, ->, transform canvas={shift={(-0.05,0.05)}}] (cNCAD)	-- node[pos=0.5, above, sloped] {$\mathrm{tito}_c$} (cTITO);

			\draw[->] (cBicW) 	-- node[pos=0.5, sloped, above=-0.6ex] {$\sim$} node[pos=0.5, right] {$\bbeta$} (cBic);
			\draw[->] (cBicW) 	-- node[pos=0.5, sloped, above=-0.6ex] {$\sim$} node[pos=0.5, below left] {$\pprec$} (cTITO);
			\draw[->] (cNCAD)	-- node[pos=0.5, sloped, above=-0.6ex] {$\sim$} node[pos=0.5, right] {$\mathrm{NC}_c$} (NCWc);
			\draw[->] (NCWc)	-- node[pos=0.5, sloped, above=-0.6ex] {$\sim$} (geom);
			\draw[->, thick] (cBicW)	-- node[pos=0.5, sloped, above=-0.6ex] {$\sim$} node[pos=0.5, below, sloped] {$nc_c^{\widetilde{A}}$} (NCWc);

			\node[draw, rectangle, rounded corners=2pt] at ([shift={(-0.2,0.2)}]sortc) {$w = s_0s_1s_2s_1$};
			\node[draw, rectangle, rounded corners=2pt] at ([shift={(-0.6,0.2)}]cBicW) {$B = \{ \affc{0,2}, \affc{2,1}_1 \}$};
			\node[draw, rectangle, rounded corners=2pt] at ([shift={(-0.75,-0.15)}]cBic) {\biclos};
			\node[draw, rectangle, rounded corners=2pt] at ([shift={(0,-0.2)}]cTITO) {$\prec = [-3,4,2]$};
			\node[draw, rectangle, rounded corners=2pt] at ([shift={(0.6,-0.5)}]cNCAD) {$\mathcal{A} = $\ncadtikz};
			\node[draw, rectangle, rounded corners=2pt] at ([shift={(0,0.2)}]NCWc) {$x = s_0s_1s_2s_0$};
			\node[draw, rectangle, rounded corners=2pt] at ([shift={(-0.4,-0.55)}]geom) {\ncgeom};
		\end{tikzpicture}
		\caption{Main objects and maps used in this paper alongside a moving example with $n = 3$ and $c = s_0s_1s_2$.}
		\label{fig: diagramme des bijections}
	\end{figure}
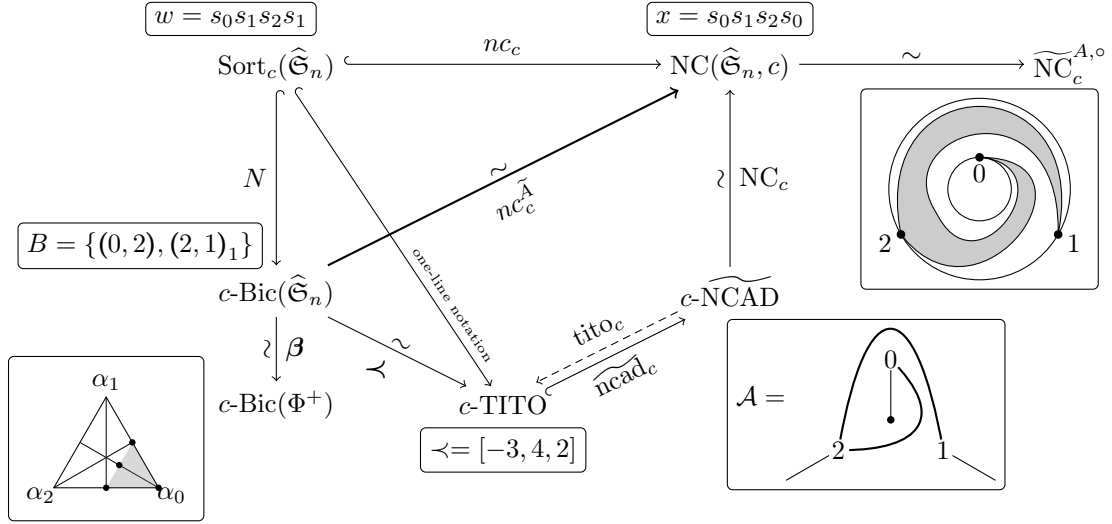
	\endgroup

	\paragraph*{} To be more precise, in Section \ref{sec: background}, we define the main objects used in this paper: the affine symmetric group, \textsc{Coxeter} sortable elements, noncrossing partitions, biclosed sets, cyclic noncrossing arc diagrams, translation-invariant total orders on $\ZZ$, as well as some known bijections between some of these objects and a few results we use throughout this paper.

	\paragraph*{} In Section \ref{sec: c triable du groupe symmetrique affine}, we give a criterion to characterize the $c$-sortable elements of the affine symmetric group $\widehat{\mathfrak{S}}_n$ in terms of pattern avoidance in their one-line notation. In general, if $W$ is a finitely generated \textsc{Coxeter} group, and $c \in W$ a \textsc{Coxeter} element, the definition of a $c$-sortable element $w \in W$ involves a specific reduced expression of $w$. This is often quite difficult to compute, even in small examples. In the finite classical types, there are criteria based on pattern avoidance in the one-line notation of the elements of $W$ seen as a subset of permutations \cite[§4]{reading2005clusters}. Theorem \ref{thm: critère par motif pour les c-triables de Sn tilde} is similar in spirit to these known criteria but for the affine type $\widetilde{A}$. See Paragraph \ref{par: elements c-triables} of Section \ref{sec: background} for the definition of some notions (for example the sets $L_c$ and $R_c$).

	\begin{restatable}[Pattern avoidance criterion for $c$-sortable affine permutations]{theoremeA}{motifaffine}
		\label{thm: critère par motif pour les c-triables de Sn tilde}
		Let $n \in \NN^*$ and $c$ be a \textsc{Coxeter} element of the affine symmetric group $\widehat{\mathfrak{S}}_n$, characterized by a partition $L_c \sqcup R_c$ of $\llbracket 1,n \rrbracket$. Let $w \in \widehat{\mathfrak{S}}_n$. Then $w$ is $c$-sortable if and only if the one-line notation of $w$ avoids the two following patterns:
		\begin{enumerate}[label*=(\roman*)]
			\item $kij$ with $i,j,k \in \ZZ$, $i < j < k$ and $j \in \Lcn$,
			\item $jki$ with $i,j,k \in \ZZ$, $i < j < k$ and $j \in \Rcn$.
		\end{enumerate}
	\end{restatable}

	This criterion is very similar to the one in type $A$: let $c \in \mathfrak{S}_n$ be a \textsc{Coxeter} element, defined by a partition $L_c \sqcup R_c = \llbracket 2,n-1 \rrbracket$ with no empty parts. Then a permutation $\sigma \in \mathfrak{S}_n$ is $c$-sortable if and only if the one-line notation of $\sigma$ avoids the patterns $kij$ with $i < j < k$ and $j \in L_c$ as well as the patterns $jki$ with $i < j < k$ and $j \in R_c$ \cite{reading2005clusters}. One key difference is that in the affine type, the one-line notation is infinite and hence there is an infinite number of patterns to check. We will see in Section \ref{ssec: critère motif c-triable} that, in fact, one only needs a finite number of checks to distinguish between $c$-sortable and non $c$-sortable elements of $\widehat{\mathfrak{S}}_n$, and using Theorem \ref{thm: critère par motif pour les c-triables de Sn tilde} is quite efficient in practice\footnotemark.
	\footnotetext{Compared to Sage's \href{https://doc.sagemath.org/html/en/reference/categories/sage/categories/coxeter_groups.html\#sage.categories.coxeter_groups.CoxeterGroups.ElementMethods.is_coxeter_sortable}{built-in function} for checking whether an element is $c$-sortable, this criterion results in calculations that are on average about 500 times faster.}

	The methodology used to prove Theorem \ref{thm: critère par motif pour les c-triables de Sn tilde} is described in \cite[§4]{reading2005clusters} and generalized for infinite \textsc{Coxeter} groups in \cite[§4]{reading2010sortable}: it is a general tool for finding such a criterion in various \textsc{Coxeter} groups that have a permutation interpretation. This method is based on a characterization of $c$-sortable elements as $c$-aligned elements \cite{reading2005clusters, reading2010sortable}. An element $w \in W$ is said to be $c$-aligned if the intersection of its inversion set with noncommutative generalized rank two parabolic subgroups is compatible with an ordering of the roots of $\Phi^+(W)$ given by an explicit skew-symmetric bilinear form named $\omega_c$. The form $\omega_c$ is defined as the skew-symmetrization of the \textsc{Euler} form $E_c$ that appears in quiver representation theory \cite{derksen2005quiver}. This ordering is given by the sign of $\omega_c(\beta_1, \beta_2)$ for every two roots $\beta_1$ and $\beta_2$ in $\Phi^+(W)$. The goal of Section \ref{sssec: forme omega_c} is to compute the sign of the form $\omega_c$ in the case where $W$ is of type $\widetilde{A}$, see Corollary \ref{cor: sign of omega_c}. In order to apply \textsc{Reading}'s method, we only need to compute this sign for $\beta_1$ and $\beta_2$ the canonical generators of noncommutative generalized rank two parabolic subgroups of $W$. Section \ref{sssec: sgpgr2nc} is dedicated to the study of these subgroups in type $\widetilde{A}$, and especially their canonical generators. Using restrictions on the inversion set of $w \in \widehat{\mathfrak{S}}_n$ given by the definition of $c$-alignment (Section \ref{sssec: c-aligné Sn tilde}), we deduce restrictions on the patterns that can appear in the one-line notation of $w$ in order to prove Theorem \ref{thm: critère par motif pour les c-triables de Sn tilde}.

	\paragraph*{}In Section \ref{sec: biclos c-triable et généralisation de l'application de Reading}, we generalize the notion of $c$-sortable elements of the affine symmetric group in order to obtain a larger family of objects. The aim is to generalize \textsc{Reading}'s map $nc_c$ into a bijection to the set of $c$-noncrossing partitions of the affine symmetric group. This generalization takes the form of \emph{$c$-sortable biclosed sets} of $W$. In the following definition, see Paragraph \ref{par: tito} of Section \ref{sec: background} for the definition of TITOs and their shape.

	\begin{restatable*}[$c$-sortable TITO on $\ZZ$, $c$-sortable biclosed set]{definition}{biclosctriable} \label{def: biclos c-triable}
		Let $c$ be a \textsc{Coxeter} element of $W$ and $\prec$ be a TITO on $\ZZ$. We say that $\prec$ is \emph{$c$-sortable} if it satisfies the two following conditions.
		\begin{enumerate}
			\item It is either of shape $[\llbracket 0,n-1 \rrbracket]$ or $[L]\underline{[M]}[R]$ where $L \sqcup M \sqcup R = \llbracket 0,n-1 \rrbracket$ and $L \subset L_c$, $R \subset R_c$, $M \cap L_c \neq \emptyset$ and $M \cap R_c \neq \emptyset$.
			\item The patterns $k \prec i \prec j$ with $j \in \Lcn$ and $j \prec k \prec i$ with $j \in \Rcn$ are avoided in the one-line notation of $\prec$.
		\end{enumerate}
		We denote by $c$-$\mathrm{TITO}$ the set of $c$-sortable TITOs.	A biclosed set $B$ of $R$ is said to be \emph{$c$-sortable} if the TITO $\pprec(B)$ is $c$-sortable and we denote $c$-$\mathrm{Bic}(W)$ the set of $c$-sortable biclosed sets of $R$.
	\end{restatable*}

	We state a few motivations behind this choice of generalization. The $c$-sortable biclosed sets naturally extend the $c$-sortable elements of $W$ using the injection $N : W \hookrightarrow \mathrm{Bic}(W)$ associating to each element $w$ of $W$ its inversion set $N(w)$. In fact, among the finite biclosed sets, only those of the form $N(w)$ with $w$ a $c$-sortable element of $W$ are $c$-sortable biclosed sets. Moreover, there is a notion of «~one-line notation~» for biclosed sets of $W$, using a bijection between biclosed sets and translation-invariant total orders on $\ZZ$ (TITOs on $\ZZ$) \cite{barkley2022combinatorial}. Using this bijection, we define the $c$-sortable biclosed sets in Section \ref{ssec: biclos c-triables} using patterns avoidance in their one-line notation, taking inspiration from Theorem \ref{thm: critère par motif pour les c-triables de Sn tilde} and keeping important properties of $c$-sortable elements such as the injectivity of the map $w \in \mathrm{Sort}_c(W) \mapsto \mathrm{Cov}(w)$, which holds in any \textsc{Coxeter} group. Moreover, biclosed sets of the positive root system are a powerful tool for studying $c$-sortable elements, thanks to the correspondence of the latter with $c$-aligned elements \cite{reading2010sortable}.

	Defining the generalized bijection $nc_c^{\widetilde{A}}$ requires additional tools, because one cannot simply take a $c$-sorting word for a biclosed set. Using once again results from \cite{barkley2025extendedweakorderaffine}, we take a detour through the cyclic noncrossing arc diagrams. They are in bijection under a map denoted $\ncad$ with the widely generated TITOs on $\ZZ$, a subset of TITOs which includes all our wanted candidates for $c$-sortable biclosed sets (or $c$-sortable TITOs on $\ZZ$ through the bijection between these two sets). Restricting this bijection to $c$-sortable TITOs on $\ZZ$, we obtain the map $\ncad_c$ and its image is the set of what we call the cyclic $c$-noncrossing arc diagrams. We study them in detail in Section \ref{ssec: c-noncrossing arc diagrams}. They, in turn, are in bijection with the $c$-noncrossing partitions under a map named $\mathrm{NC}_c$ which we define in Section \ref{sssec: bijection between NCAD and NC}. Figure \ref{fig: diagramme des bijections} shows the main injections and bijections used in this paper. The map $nc_c^{\widetilde{A}}$ has a bold arrow, and is defined as $nc_c^{\widetilde{A}} = \mathrm{NC}_c \circ \ncad_c \circ \pprec$.

	\paragraph*{}In Section \ref{sec: generalization of reading's bijection}, we prove that the map $nc_c^{\widetilde{A}}$ is a bijection between the set of $c$-sortable biclosed sets and the set of $c$-noncrossing partitions of the affine symmetric group and that it extends \textsc{Reading}'s map $nc_c$.

	\begin{restatable}[Generalized \textsc{Reading} map]{theoremeA}{readinggeneralise}
		\label{thm: generalisation de l'application de Reading}
		Let $W$ be a \textsc{Coxeter} group of type $\widetilde{A}_n$, and $c \in W$ be a \textsc{Coxeter} element of $W$. The map $nc_c^{\widetilde{A}}$ is a bijective map from $c$-$\mathrm{Bic}(W)$ to $\mathrm{NC}(W,c)$. Moreover, for all $w \in \mathrm{Sort}_c(W)$, we have $nc_c^{\widetilde{A}}(N(w)) = nc_c(w)$.
	\end{restatable}

	Most of Section \ref{sec: generalization of reading's bijection} is dedicated to prove that the map $\ncad_c$ is bijective, and we also explicitly compute its inverse map (which we call $\mathrm{tito}_c$) in Section \ref{sssec: défintion de TITOc}. It is defined by selecting arcs of a cyclic $c$-noncrossing arc diagram in a specific order, resembling the computation of the inverse map of \textsc{Reading}'s bijection $nc_c$ in type $A$ \cite[6.3]{gobet2018dual}. The main difficulty of this procedure is to find a « good » system of representatives of arcs modulo $n$, called \emph{numbering}, to chose from. This is the goal of Section \ref{ssec: renumérotation} in which we use a sequence called $\mathcal{C}_a(\mathcal{A})$, where $\mathcal{A}$ is a cyclic $c$-noncrossing arc diagram, in order to determine a particular choice of numbering. We study the properties of the sequence in Section \ref{sssec: structure des c-NCAD + suite C}.

	The chain of bijections $nc_c^{\widetilde{A}} = \mathrm{NC}_c \circ \ncad_c \circ \pprec$ does not hide the explicit nature of the computation of the map $nc_c^{\widetilde{A}}$: see for example Definition \ref{def: reading généralisée} and Example \ref{ex: direct map}.

	\paragraph*{Acknowledgements} The present work is part of the author's PhD thesis at Université de Tours, Institut Denis Poisson, supervised by Thomas \textsc{Gobet} and Cédric \textsc{Lecouvey}. The author would like to thank both Thomas \textsc{Gobet} and Cédric \textsc{Lecouvey} for their guidance throughout every stage of the writing process of this paper and beyond. Additionally, the author thanks the ANR project CORTIPOM (ANR-21-CE40-0019) which gave the author the possibility of attending to many conferences.

\section{Background} \label{sec: background}

		\paragraph{Finite and affine symmetric groups}

		We denote by $\mathfrak{S}_n$ the \emph{symmetric group} on $\llbracket 1,n \rrbracket$ and $\widehat{\mathfrak{S}}_n$ the \emph{affine symmetric group} of order $n$. It is the subgroup of the symmetric group over $\ZZ$ consisting of all the permutations such that:
		\begin{multicols}{2}
			\begin{enumerate}[label*=(\alph*)]
			\item for all $k \in \ZZ$, $\sigma(k+n) = \sigma(k) + n$,
			\item the equality $\sum_{k=1}^n \sigma(k) = \sum_{k=1}^n k$ holds.
		\end{enumerate}
		\end{multicols}

		\noindent An \emph{orbit} of an affine permutation $\widehat{\sigma} \in \widehat{\mathfrak{S}}_n$ is a set of the form $\{ \widehat{\sigma}^k(a) \;|\; k \in \ZZ \}$ for some integer $a \in \ZZ$.

			\subparagraph{Notation} The \emph{one-line notation} of a permutation $\sigma$ is the sequence $\sigma(1) \dots \sigma(n)$ and for an affine permutation $\widehat{\sigma}$, it is $\cdots \widehat{\sigma}(-1) \widehat{\sigma}(0) \widehat{\sigma}(1) \cdots$. For simplicity, we often use the \emph{window notation} of an affine permutation $\widehat{\sigma}$: $[\widehat{\sigma}(1), \dots, \widehat{\sigma}(n)]$, which suffices to encode $\widehat{\sigma}$ using point (a) above.

			The natural injection $\iota : \mathfrak{S}_n \hookrightarrow \widehat{\mathfrak{S}}_n$ is given by $\iota(\sigma) = [\sigma(1), \dots, \sigma(n)]$.

			\subparagraph{Cover reflections} A \emph{cover reflection} of a permutation $\sigma \in \mathfrak{S}_n$ is a transposition $(i,j)$ such that $i < j$ and $ji$ is a factor of the one-line notation of $\sigma$. They are a special case of \emph{inversions} where we only require $j$ to precede $i$. We define the same notion for an affine permutation, in this case a cover reflection is an affine transposition $\affc{i,j}$. We denote $\mathrm{Cov}(\sigma)$ the set of all cover reflections of an (affine) permutation $\sigma$.

			\subparagraph{Affine transpositions} The \emph{affine transposition} $\affc{i,j}$, for $i \not\equiv j \bmod n$, is the affine permutation that exchanges, for all $p \in \ZZ$, $i + pn$ and $j + pn$ and fixes all the other integers. Since any representative of the class of $(i,j)$ in $\ZZ^2 / (n,n) \ZZ$ gives the same affine transposition, there is not a unique description of them. This is why we introduce the notation $\affc{i,j}_p$ for $1 \leq i,j \leq n$ and $p \in \ZZ$, defined as $\affc{i,j}_p = \affc{i, j+pn}$. To have uniqueness in the notation, we either enforce $i < j$ or $p \geq \ddelta_{i > j}$.

			\subparagraph{Affine cycles} Cycles in the affine symmetric group behave in a more complicated way than in the finite symmetric group. We refer to \cite[§2]{digne2005presentations} for the notations and nomenclature we use in this paper. A \emph{cycle} will be denoted $\affc{a_1, \dots, a_k}_{[p]}$ where $p \in \ZZ$ is its \emph{shift}. If $p = 0$, we may omit the index and say that it is a \emph{finite cycle}. A pseudo-cycle is a product of cycles such that the sum of their shifts is zero, but any partial sum of their shifts is nonzero.

			\subparagraph{\textsc{Coxeter} systems} With the simple transpositions $S = \{ (i,i+1) \;|\; i \in \llbracket 1,n-1 \rrbracket \}$, $(\mathfrak{S}_n, S)$ is a \textsc{Coxeter} system of type $A$, and with the simple affine transpositions $\widehat{S} = \{ \affc{i,i+1} \;|\; i \in \llbracket 1,n \rrbracket \}$, $(\widehat{\mathfrak{S}}_n, \widehat{S})$ is a \textsc{Coxeter} system of type $\widetilde{A}$. For background on the general theory of \textsc{Coxeter} groups, see \cite{humphreys_1990}.

		\paragraph{Root systems \cite{humphreys_1990}} We use the geometric representation of a \textsc{Coxeter} group $W$, which is a real vector space $V$ of basis $(\alpha_s)_{s \in S}$ on which $W$ acts faithfully.

			\subparagraph{Definition} In this paper, we consider the \emph{root system} of a \textsc{Coxeter} group $W$ to be the set of unit vectors $\Phi = \{ w(\alpha_s) \;|\; w \in W, s \in S \} \subset V$. The geometry on $V$ is given by the symmetric bilinear form:
			\begin{equation*}
				\forall s,s' \in S,\ B(\alpha_s, \alpha_{s'}) = - \cos \left( \dfrac{\pi}{m_{s,s'}} \right)
			\end{equation*}
			where $(m_{s,s'})_{s,s' \in S}$ is the \textsc{Coxeter} matrix. The \emph{positive root system} $\Phi^+$ corresponds to the set of all roots of $\Phi$ such that their coefficients in the basis $(\alpha_s)_{s \in S}$ are all nonnegative.

			\subparagraph{Roots in type $\widetilde{A}$} Let $n \geq 2$ and consider $S = \{ s_0, \dots, s_{n-1} \}$ such that $m(s_i, s_j) = 3$ if $|i-j| = 1 \bmod n$ and $2$ otherwise. If $n = 1$, consider $S = \{s_0, s_1\}$ such that $m(s_0, s_1) = \infty$.

			For simplicity, we denote the basis of $V$ by $(\alpha_i)_{0 \leq i \leq n-1}$ instead of $(\alpha_{s_i})_{0 \leq i \leq n-1}$. The positive root system of type $\widetilde{A}_n$ is
			\begin{equation*}
				\Phi^+ = \bigsqcup_{p \in \NN} \bigsqcup_{0 \leq i < j \leq n}  \{ \alpha_i + \dots + \alpha_{j-1} + p\delta,\; - (\alpha_i + \dots + \alpha_{j-1}) + (p+1)\delta \}
			\end{equation*}
			where $\delta$ is the vector $\delta = \alpha_0 + \dots + \alpha_{n-1}$. This vector plays an important role in the affine root system, it spans the radical of $B$ and the vector space $V/\RR\delta$ alongside the bilinear form induced by $B$ is a geometric representation of a \textsc{Coxeter} group of type $A_n$.

			\subparagraph{Bijection with reflections} The positive root system of $W$ is in bijection with $R$, the set of reflections of $W$. In type $\widetilde{A}_n$, this bijection is given by, using the convention for affine permutations $(i,j)_p$ with $0 \leq i,j \leq n-1$ and $p \geq \ddelta_{i > j}$:
			\begin{equation*}
				\begin{array}{*4c}
					\bbeta :	& R				& \rightarrow	& \Phi^+ \\
								& t = (i,j)_p	& \mapsto		&
						\beta_t = \left\lbrace \begin{array}{ll}
							\alpha_i + \dots + \alpha_{j-1} + p\delta & \text{ if } i < j \\
							- (\alpha_i + \dots + \alpha_{j-1}) + (p-1) \delta & \text{ if } j < i.
						\end{array} \right.
				\end{array}
			\end{equation*}
			In particular, $\beta_{s_i} = \alpha_i$ for all $i \in \llbracket 0, n-1 \rrbracket$.

		\paragraph{Biclosed sets}

		Biclosed sets are defined as subsets of the root system $\Phi^+$.

			\subparagraph{Definition} We denote by $R = \bigcup_{w \in W} wSw^{-1}$ the set of reflections of $W$. Let $X \subset R$. We say that $X$ is:
			\begin{itemize}
				\item \emph{closed} if for all $r,s,t \in R$ such that $\beta_t \in \mathrm{Span}_+(\beta_r, \beta_s)$ and $r,s \in X$, we have $t \in X$.
				\item \emph{coclosed} if $R \setminus X$ is closed.
				\item \emph{biclosed} if $X$ is closed and coclosed.
			\end{itemize}

			Usually, for example in \cite{barkley2022combinatorial}, biclosed sets are defined for subsets of the positive root system $\Phi^+$. The biclosed sets of $\Phi^+$ are precicely the subsets $B \subset \Phi^+$ such that $\bbeta^{-1}(B)$ is a biclosed set of $W$.

			We denote by $\mathrm{Bic}(W)$ (resp. $\mathrm{Bic}(\Phi^+)$) the set of biclosed sets of $R$ (resp. of $\Phi^+$).

			\subparagraph{Inversion sets} Let $w \in W$. Its (left) inversion set $N(w) = \{ t \in R \;|\; l(tw) < l(w) \}$ is a biclosed subset. In fact, the finite biclosed subsets of $R$ are exactly the inversion sets of elements of $W$ \cite{Dyer_2019}. Note that this definition coincides with the inversions of a permutation of $\mathfrak{S}_n$ or $\widehat{\mathfrak{S}}_n$.

			When $W$ is infinite, all the infinite biclosed sets fail to be inversion sets because no element of $W$ has an infinite number of inversions. Some infinite biclosed sets can be expressed as inversions sets of infinite reduced words \cite{wang2019infinite}, but we do not enter in such considerations in this paper.

			Due to the relation, for any $s \in S$ and $w \in W$, $N(ws) = N(w) \diffsym \{wsw^{-1}\}$ (where $\diffsym$ is the symmetric difference operator), we can define the notion of cover reflections entirely on the inversion set of $w$, and generalize this definition to biclosed sets. We say that $t \in R$ is a \emph{cover reflection} of a biclosed set $B \subset R$ if $B \setminus \{t\}$ is a biclosed set.

		\paragraph{\textsc{Coxeter} sortable elements} \label{par: elements c-triables}

		A \emph{\textsc{Coxeter} element} of a \textsc{Coxeter} system $(W,S)$ is an element $c \in W$ which can be expressed as the product of all the elements of $S$ exactly once in some order.

			\subparagraph{Definition \cite{reading2005clusters}}
			Let $c = s_1 \dots s_n$ where $\{s_i\}_{i \in \llbracket 1,n \rrbracket} = S$, and fix this reduced expression of $c$. Define the half-infinite word $c^\infty$ as an infinite repetition of the chosen reduced expression of $c$:
			\begin{equation*}
				c^\infty = s_1s_2 \dots s_n \  s_1s_2 \dots s_n \  s_1s_2 \dots s_n \  \cdots
			\end{equation*}
			Let $u = u_1 \dots u_k$ be a finite subword of $c^\infty$. It defines a sequence $(p_i(u))_{i \in \NN} \in \{0,1\}^\NN$ such that $p_i(u) = 1$ if and only if a letter of $u$ came from the $i$-th letter of $c^\infty$.
			Let $w \in W$. Among all its reduced expressions $u \in S^*$, we consider the one that appears the first as a subword in $c^\infty$ for the lexicographic order on the sequences $(p_i(u))_{i \in \NN}$, denoted $w_c$ and called $w$'s \emph{$c$-sorting word}. We say that $w$ is \emph{$c$-sortable} if for all $i \in \NN$, $p_{i+n}(w_c) \leq p_i(w_c)$.

			The set of $c$-sortable elements of $W$ is denoted $\mathrm{Sort}_c(W)$.

			\subparagraph{\textsc{Coxeter} elements in type \texorpdfstring{$A$}{A} and \texorpdfstring{$\widetilde{A}$}{Ã}} \textsc{Coxeter} elements can be characterized in type $A$ and $\widetilde{A}$. A permutation $\sigma \in \mathfrak{S}_n$ is a \textsc{Coxeter} element if and only if $\sigma = (1, a_1, \dots, a_{k-1}, n, b_{n-k-1}, \dots, b_1)$ where $1 \leq k \leq n - 1$, $a_1 < \dots < a_{k-1}$, $b_1 < \dots < b_{n-k-1}$ and $\{1,n,a_1, \dots, a_{k-1}, b_1, \dots, b_{n-k-1}\} = \llbracket 1,n \rrbracket$. In this case, we denote $L_\sigma = \{ a_1, \dots, a_{k-1} \}$ and $R_\sigma = \{ b_1, \dots, b_{n-k-1} \}$. The following map is a bijection:
			\begin{equation*}
				\begin{array}{*3c}
					\{ \text{\textsc{Coxeter} elements of } \mathfrak{S}_n \} & \xrightarrow{\sim} & \{ \text{partitions of } \llbracket 2,n-1 \rrbracket \text{ in two parts} \} \\
					c & \mapsto & (L_c, R_c).
				\end{array}
			\end{equation*}

			A similar result holds for \textsc{Coxeter} elements of the affine symmetric group {\cite[§2.10]{digne2005presentations}}. An affine permutation $c$ is a \textsc{Coxeter} element if and only if it is a pseudo-cycle of the form
			\begin{equation*}
				\affc{a_1, \dots, a_k}_{[1]} \affc{b_{n-k}, \dots, b_1}_{[-1]}
			\end{equation*}
			with $1 < k < n$, $\{ a_1, \dots, a_k, b_1, \dots, b_{n-k} \} = \llbracket 1,n \rrbracket$, $a_1 < \dots < a_k$ and $b_1 < \dots < b_{n-k}$.

			In this case, we will denote $L_c$ (resp. $R_c$) the set $\{a_1, \dots, a_k\}$ (resp. $\{b_1, \dots, b_{n-k}\}$). Furthermore, we will write $\Lcn = L_c + n\ZZ$ and $\Rcn = R_c + n\ZZ$. The following map is a bijection:
			\begin{equation*}
				\begin{array}{*3c}
					\{ \text{\textsc{Coxeter} elements of } \widehat{\mathfrak{S}}_n \} & \xrightarrow{\sim} & \{ \text{partitions of } \llbracket 1,n \rrbracket \text{ in two nonempty parts} \} \\
					c & \mapsto & (L_c, R_c).
				\end{array}
			\end{equation*}

			Note that the sets $\Lcn$ and $\Rcn$ have the following characterization: let $i \in \ZZ$. Then $i \in \Lcn$ if and only if $c(i) > i$.

			\subparagraph{\textsc{Coxeter} sortable elements in type $A$}

			There is a known characterization of $c$-sortable elements of the symmetric group using pattern avoidance on their one-line notation \cite[4.12]{reading2005clusters}. We recall this characterization here as the goal of Section \ref{sec: c triable du groupe symmetrique affine} is to find a similar characterization for affine permutations.

			Let $c \in \mathfrak{S}_n$ be a \textsc{Coxeter} element. Then a permutation $\sigma \in \mathfrak{S}_n$ is $c$-sortable if and only if its one-line notation avoids the patterns
			\begin{multicols}{2}
				\begin{itemize}
				\item $kij$ with $i < j < k$ and $j \in L_c$,
				\item $jki$ with $i < j < k$ and $j \in R_c$.
			\end{itemize}
			\end{multicols}
			These patterns are subwords of the one-line notation of $w$, but we can also assume that $k$ and $i$ are always consecutive.

		\paragraph{Noncrossing partitions of a \textsc{Coxeter} group}

			\subparagraph{Absolute order} In a \textsc{Coxeter} group $W$, since $S \subset R$, the reflections also generate $W$. The \emph{absolute length} of $w \in W$, denoted $l_R(w)$ is the length of a reduced expression of $w$ on the alphabet $R$. If $l$ denotes the usual length (on the alphabet $S$), we have $l_R(w) \leq l(w)$.

			Using reflections as generators, we define the \emph{absolute order} similarly to the weak order: $x \leq_R y$ if and only if every reduced expression of $x$ on the alphabet $R$ (called an \emph{$R$-reduced expression} of $x$) is a prefix of an $R$-reduced expression of $y$. Using the absolute length function, this is equivalent to $l_R(x) + l_R(x^{-1}y) = l_R(y)$.

			\subparagraph{Noncrossing partitions} If $c \in W$ is a \textsc{Coxeter} element, an element $w \in W$ is said to be a \emph{$c$-noncrossing partition} if $w \leq_R c$. We denote the set of $c$-noncrossing partitions of $W$ by $\mathrm{NC}(W,c)$ instead of $[1,c]_{\leq_R}$.

			\subparagraph{In finite \textsc{Coxeter} groups} When $W$ is a finite \textsc{Coxeter} group, the $c$-sortable elements and the $c$-noncrossing partitions are in bijection \cite{reading2005clusters}. This bijection, denoted $nc_c$ and called in this paper \emph{\textsc{Reading}'s bijection}, is defined as such: let $w \in W$ be a $c$-sortable element and consider $w_c = w_1 \dots w_r$ its $c$-reduced expression. We have
			\begin{equation*}
				N(w) = \{ w_1,\ w_1w_2w_1,\ w_1w_2w_3w_2w_1,\ \dots,\ w_1 \dots w_{r-1} w_r w_{r-1} \dots w_1 \}
			\end{equation*}
			and we totally order $N(w)$ such that $w_1 \dots w_{i-1} w_i w_{i-1} \dots w_1 \lhd w_1 \dots w_{j-1} w_j w_{j-1} \dots w_1$ if and only if $i < j$. Let $t_1 \lhd \dots \lhd t_k$ be the cover reflections of $w$ ordered by the restriction of $\lhd$ to $\mathrm{Cov}(w)$. Then $nc_c(w) = t_1 \dots t_k$ is a $c$-noncrossing partition.

			The inverse map of $nc_c$ is not explicit in general, but can be described on a type by type basis \cite{gobet2018dual}.

			\subparagraph{In infinite \textsc{Coxeter} groups} When $W$ is infinite, the map $nc_c$ is still defined but is only an injection \cite{reading2010sortable}. In Section \ref{sec: biclos c-triable et généralisation de l'application de Reading} we generalize the notion of $c$-sortable elements of the affine symmetric group in such a way that \textsc{Reading}'s map $nc_c$ generalizes into a bijection.

		\paragraph{Geometric interpretations in type \texorpdfstring{$\widetilde{A}$}{Ã}} \label{par: geometric interpretation of noncrossing partitions}

		In the case where $W$ is a \textsc{Coxeter} group of type $A$, there is a well known interpretation of the $c$-noncrossing partitions as noncrossing sets of polygons \cite{biane97} and this geometric point of view is used to describe the inverse map in type $A$ \cite{gobet2018dual}.

		When $W$ is of type $\widetilde{A}$, there is also a geometric description of the $c$-noncrossing partitions, as a set of noncrossing curved polygons in an annulus \cite{reading2024noncrossingpartitionsannulus,digne2005presentations}. In this paper, we do not consider dangling annular blocks: our $c$-noncrossing partitions will correspond to the set $\widetilde{\mathrm{NC}}_c^{A, \circ}$ of \cite{reading2024noncrossingpartitionsannulus}.

		\begin{figure}[ht!]
			\centering
			\begin{tikzpicture}[bezier bounding box, scale=0.8]
				\def\n{10}
				\def\Lc{0,2,4,6,8,9}
				\def\Rc{1,3,5,7}

				\def\r{1}
				\def\R{2}

				\newcommand{\ang}[1]{{-(#1-0.5)*360/\n + 90}}
				\newcommand{\dis}[1]{{#1*\R + (1-#1)*\r}}

				\draw (0,0) circle (\dis{0});
				\draw (0,0) circle (\dis{1});

				\foreach \i in \Lc{
					\coordinate (\i) at (\ang{\i}:\dis{1});
					\draw[draw=none] (0,0) -- node[pos=1.2] {\i} (\i);
				}
				\foreach \i in \Rc{
					\coordinate (\i) at (\ang{\i}:\dis{0});
					\draw[draw=none] (0,0) -- node[pos=0.7] {\i} (\i);
				}

				\pgfmathparse{\n-1}	
				\foreach \i in {0,...,\pgfmathresult}{
					\node[draw, circle, fill, inner sep=1pt] at (\i) {};
				}

				\draw[relative, fill, fill opacity=0.2]
				(8) to[out=100,in=80, looseness=1.8] (4) to[out=60, in=120, looseness=1.2] (8)
				(3) to[out=-100, in=-90, looseness=1.7] (\ang{8}:\dis{0.6}) to[out=-90, in=-80, looseness=2] (3);

				\begin{scope}
					\draw[relative, clip]
					(5) to[out=40, in=135] (7) to[out=80, in=105, looseness=1.4] (1) to[out=-110, in=-70, looseness=3] (5);
					\begin{scope}[rotate=45]
						\foreach \x in {-1,-0.95,...,1.1}{
							\draw ({\R*\x},{-\R}) -- ({\R*\x}, \R);
						}
					\end{scope}
				\end{scope}

				\draw[thick, relative]
				(9) to[out=40, in=140, looseness=1] (2);
			\end{tikzpicture}
			\caption{A $c$-noncrossing partition of an annulus for $n = 10$}
			\label{fig: exemple partition non croisée anneau}
		\end{figure}
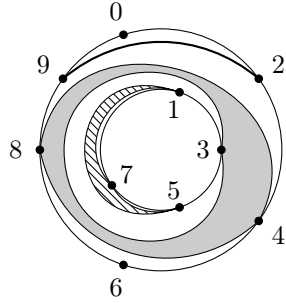

			\subparagraph{Notation} In this paper, we use a slightly different terminology compared to \cite{reading2024noncrossingpartitionsannulus}: we call the embedded blocks \emph{curved polygons} for similarities with the case of the finite type $A$, and we omit the adjective «~nondangling~» for annular blocks as we only consider such annular blocks.

			\subparagraph{\texorpdfstring{$c$}{c}-marking of an annulus} Let $c$ be a \textsc{Coxeter} element of the affine symmetric group. A \emph{$c$-marking} of an annulus $A$ is done by placing $n$ numbered points from $0$ to $n-1$ (or any interval of $\ZZ$ of length $n$) on the border of the annulus. These points appear in increasing order starting by $0$ and reading in the clockwise direction, with the elements of $\Lcn$ on the outer border, and the elements of $\Rcn$ on the inner border.

			\subparagraph{\texorpdfstring{$c$}{c}-noncrossing partition of an annulus} Let $A$ be a $c$-marked annulus. A \emph{$c$-noncrossing partition} of $A$ is a set of disjoint curved polygons such that each marked point of $A$ belongs to a curved polygon. A $c$\nobreakdash-noncrossing partition is defined up to isotopy, meaning the precise shape of each curved polygon is not relevant, only their relative position.

			A bijection between $c$-noncrossing partitions of an annulus and $c$-noncrossing partitions of $W$ is done explicitly in \cite[§3.5]{reading2024noncrossingpartitionsannulus}. We call $\mathrm{perm}(A)$ the $c$-noncrossing affine permutation associated to a $c$-noncrossing partition of an annulus $A$.

			For example, the marked annulus of Figure \ref{fig: exemple partition non croisée anneau} is a $c$-marked annulus for $c = \affc{0,2,4,6,8,9}_{[1]} \affc{7,5,3,1}_{[-1]}$. Moreover, the five curved polygons present in this annulus form a $c$-noncrossing partition (recall that a numbered point is a curved polygon), and its corresponding $c$-noncrossing partition of $W$ is the product of the five pseudo-cycles $\sigma = \affc{6} \affc{10} \affc{-1,2} \affc{-5,-3,1} (\affc{4,8}_{[1]} \affc{3}_{[-1]})$.

			\subparagraph{Combinatorial description}

			For a curved polygon $P$ of a $c$-noncrossing partition of an annulus, we define $\mathrm{part}(P)$ to be the set of all orbits of the action of the affine permutation $\mathrm{perm}(P)$ on $\ZZ$. They correspond to all the different «~copies~» modulo $n$ of the vertices of a curved polygon.

			Let $A$ be a $c$-marked annulus and $P$, $Q$ two curved polygons of $A$. Then $P$ and $Q$ intersect if and only if there exists $U \in \mathrm{part}(P)$, $V \in \mathrm{part}(Q)$ such that $U \neq V$ and there exist $i,j,k,l \in \ZZ$ such that $i < j < k < l$ and one of the following cases is true
			\begin{itemize}
				\item $(i,k) \in U^2, \, (j,l) \in V^2$ and $(j,k) \in (\Lcn)^2 \sqcup (\Rcn)^2$,
				\item $(i,l) \in U^2, \, (j,k) \in V^2$ and $(j,k) \in \Lcn \times \Rcn \sqcup \Rcn \times \Lcn$.
			\end{itemize}

		\paragraph{Noncrossing partitions of \texorpdfstring{$\widehat{\mathfrak{S}}_n$}{the affine symmetric group}} \label{par: partitions non croisées Sn tilde}

		Using the notion of \emph{elementary divisors of $c$} from \cite[2.19]{digne2005presentations}, we obtain some properties of $c$-noncrossing partitions of $\widehat{\mathfrak{S}}_n$.

			\subparagraph{Elementary divisors} Let $L = \{ l_1 < \dots < l_l \} \subset \Lcn$ and $R = \{ r_1 < \dots < r_r  \} \subset \Rcn$ such that $l,r \geq 0$, $l_l - l_1 < n$ and $r_r - r_1 < n$. Then there exists a unique finite cycle of orbit $L \cup R$ that is a $c$-noncrossing partition, it is $\affc{l_1, \dots, l_l, r_r, \dots, r_1}$. Moreover, if $l,r \geq 1$, there also exists a unique pseudo-cycle of orbit $(L \cup R) + n\ZZ$ that is a $c$-noncrossing partition, it is $\affc{l_1, \dots, l_l}_{[1]} \affc{r_r, \dots, r_1}_{[-1]}$. We call them \emph{elementary divisors of $c$}.

			\subparagraph{Characterization}

			Let $x \in \widehat{\mathfrak{S}}_n$ and $c_1, \dots, c_r \in \widehat{\mathfrak{S}}_n$ be cycles such that $x = c_1 \dots c_r$ is the decomposition of $x$ in a product of disjoint cycles. Then $x$ is a $c$-noncrossing partition of $\widehat{\mathfrak{S}}_n$ if and only if the three following conditions are satisfied:
			\begin{enumerate}[label*=(\roman*)]
				\item There is either zero or two integers $i \in \llbracket 1,r \rrbracket$ such that $c_i$ has a nonzero shift. In the case where there are two such cycles, their product is a pseudo-cyclic elementary divisor of $c$.
				\item If $i \in \llbracket 1,r \rrbracket$ is such that $c_i$ if a finite cycle, $c_i$ is a cyclic elementary divisor of $c$.
				\item For all $i,j \in \llbracket 1,r \rrbracket$, for all $U$ a $c_i$-orbit and $V$ a $c_j$-orbit such that $U \neq V$, the following holds for all $a,b,c,d \in \ZZ$ such that $a < b < c < d$:
				\begin{equation*}
					\left\lbrace
					\begin{array}{l}
						(a,c) \in U^2, \, (b,d) \in V^2 \Rightarrow (b,c) \in (\Lcn \times \Rcn) \sqcup (\Rcn \times \Lcn) \\
						(a,d) \in U^2, \, (b,c) \in V^2 \Rightarrow (b,c) \in (\Lcn)^2 \sqcup (\Rcn)^2.
					\end{array}
					\right.
				\end{equation*}
			\end{enumerate}

			\subparagraph{\texorpdfstring{$R$}{R}-reduced expression} Let $x$ be a $c$-noncrossing partition of $\widehat{\mathfrak{S}}_n$. Define $P(x)$ to be the reflection subgroup of $\widehat{\mathfrak{S}}_n$ generated by all the reflection $t \in R$ such that $t \leq_T x$. We have the following results:
			\begin{enumerate}[label*=(\roman*)]
				\item For any $R$-reduced expression $t_1 \dots t_r$ of $x$ with $t_1, \dots, t_r \in R$, we have $P(x) = \langle t_1, \dots, t_r \rangle$.
				\item The subgroup $P(x)$ uniquely determines $x$.
			\end{enumerate}
			Although this result is already known for an irreducible affine \textsc{Coxeter} group \cite{paolini_2020}, in the case of the affine symmetric group, the proof is elementary in comparison and implicitly present in \cite{digne2005presentations}.

			This proposition has the very useful consequence which we will use in Section \ref{sssec: generalized reading's map} to prove that our generalization of \textsc{Reading}'s bijection $nc_c$ indeed coincides with $nc_c$ on the set of $c$-sortable elements of $W$: If $x = t_1 \dots t_r$ is a $R$-reduced expression of $x \in \mathrm{NC}_c(\widehat{\mathfrak{S}}_n)$, then for all $\sigma \in \mathfrak{S}_r$, the element $x_\sigma = t_{\sigma(1)} \dots t_{\sigma(r)}$ is either equal to $x$ or is not a $c$-noncrossing partition.

		\paragraph{Translation-invariant total orders on \texorpdfstring{$\ZZ$}{Z}} \label{par: tito}

		In the affine symmetric group, there is a description of the biclosed sets of $R$ as a subset of the set of total orders on $\ZZ$: the real \emph{translation-invariant total orders} (TITOs) on $\ZZ$. We refer to \cite{barkley2022combinatorial}, \cite{barkley2025extendedweakorderaffine} and \cite{barkley2025affinetamarilattice} for the results about these objects and we recall a few important results as well as some new definitions around them.

			\subparagraph{Definition} A TITO on $\ZZ$ is a total order $\prec$ on $\ZZ$ such that for all $x,y \in \ZZ$, $x \prec y$ if and only if $x + n \prec y + n$. We say that a TITO $\prec$ on $\ZZ$ is \emph{real} if there is no cover relation $x + n \prec x$ for $x \in \ZZ$. In this paper, we only consider real TITOs on $\ZZ$, and we refer to them simply as TITOs (without the adjective «~real~»).

			A \emph{block} is an infinite interval $I$ of $\prec$ such that for all $x \in I$, $x$ is at a finite distance of all other elements of $I$ and at infinite distance of all elements of $\ZZ \setminus I$.

			\subparagraph{Notation} Since TITOs are total orders, we can represent a TITO $\prec$ on a single infinite line. Using their periodic nature, only a finite number of terms are needed to completely describe the TITO. Hence, we often use their \emph{window notation}, similarly to the ones for affine permutations but with potentially many windows, each corresponding to a different block. If $x$ belongs to a block and $x+n \prec x$, we underline the window and say it is a \emph{waning} block, otherwise it is a \emph{waxing} block. For example, this is a $4$-TITO on $\ZZ$ composed of a waning block followed by a waxing block:
			\begin{equation*}
				\cdots \prec -2 \prec 5 \prec -6 \prec 1 \prec \dots \prec 3 \prec -4 \prec 7 \prec 0 \prec \cdots \ = \ \underline{[5,-6]} [-4,7].
			\end{equation*}

			\subparagraph{Shape of a TITO}
			The \emph{shape} of a TITO $\prec$ is a window notation of $\prec$ in which the integers of each window of each block are replaced by the set of their residue modulo $n$. For convenience in future statements, we allow blocks $[\emptyset]$ or $\underline{[\emptyset]}$ to also appear. For example, the shape of the previous TITO is $\underline{[\{ 1,2 \}]} [\{0,3\}]$.

			A shape can be viewed as a signed ordered partition of $\llbracket 0,n-1 \rrbracket$. It does not fully describe a TITO but it is a useful tool for classifying them. In fact, two TITOs have the same shape if and only if their inversion sets (see below) are commensurable \cite[4.9]{barkley2022combinatorial}.

			\subparagraph{Inversions and cover reflections}
			Let $\prec$ be a TITO on $\ZZ$, an \emph{inversion} of $\prec$ is an affine reflection $\affc{i,j}$ with $i < j$ distinct modulo $n$ such that $j \prec i$. We note $N(\prec)$ the set of inversions of $\prec$. Furthermore, $\affc{i,j}$ is a \emph{cover reflection} if additionally, $j \prec i$ is a cover relation.

			The bijection between the set of TITOs on $\ZZ$ and the biclosed sets of $R$ is given by the map $N$, and in this paper we denote its inverse map by $\pprec$.

			In \cite{barkley2025extendedweakorderaffine}, the cover reflections are called \emph{lower walls}. In this paper, we shall use the terminology of cover reflections because of the following result: let $B \subset R$ be a biclosed set, then $\mathrm{Cov}(B) = \mathrm{Cov}(\pprec(B))$.

			\subparagraph{Descending chains}
			A \emph{descending chain} of a TITO on $\ZZ$ $\prec$ is a sequence $(a_1, \dots, a_k)$ of strictly decreasing integers, pairwise distinct modulo $n$ except possibly $a_1$ and $a_k$, such that for all $i \in \llbracket 1, k-1 \rrbracket$, $\affc{a_i, a_{i+1}}$ is a cover reflection of $\prec$. We say that a descending chain $(a_1, \dots, a_k)$ is \emph{maximal} if $a_k \equiv a_1 \bmod n$ or if there are no other descending chain properly containing $a_1, \dots, a_k$.

			We consider two descending chains of a TITO $\prec$ to be equal if they contain the same elements modulo $n$ and have the same length. Note that this second condition is crucial: if $(a_1, \dots, a_k)$ is a descending chain such that $a_1 \equiv a_k \bmod n$, then $(a_2, \dots, a_k)$ truly is a different descending chain despite having the same elements modulo $n$. Meanwhile, $(a_2, \dots, a_k, a_2 + (a_k - a_1))$ is considered the same descending chain.

			An \emph{orbit} of a descending chain $d$ of $\prec$ is a set $\{ a_1, \dots, a_k \}$ such that $d = (a_1, \dots, a_k)$. A TITO has a finite number of maximal descending chains, and their orbits form a partition of $\ZZ$.

		\paragraph{Cyclic noncrossing arc diagrams} \label{par: cyclic noncrossing arc diagrams}

		Noncrossing arc diagrams appear as a useful tool for studying the weak order in the symmetric group \cite{reading2015noncrossingarcdiagramscanonical}. In the affine symmetric group, these arc diagrams can be generalized into \emph{cyclic noncrossing arc diagrams} and play a similar role for the extended weak order \cite{barkley2025extendedweakorderaffine}.

			\subparagraph{Definition}
			If $n \in \NN^*$, a \emph{cyclic noncrossing arc diagram (on $n$ points)} is a noncrossing arc diagram on countably infinite points labeled by $\ZZ$ that is invariant by translation of $n$. We refer to \cite[§4.3]{barkley2025extendedweakorderaffine} for more information on these diagrams. We may use the term «~\emph{cyclic arc}~» to refer to the class of all the translated copies of an arc. We insist on the fact that two arcs are not considered «~crossing~» if the final point of one is the initial point of the other (\textit{i.e.} we can create chains of arcs).

			\subparagraph{Diagram on a cycle} One could also «~roll~» the line of points indexed by $\ZZ$ into a circle of $n$ points, each one labeled by a class of residues modulo $n$, so that each curve between two points represents a cyclic arc instead of an individual arc. Figure \ref{fig: diagramme d'arcs non croisée deux représentations} shows an example of the two different representations of the same cyclic noncrossing arc diagram.

			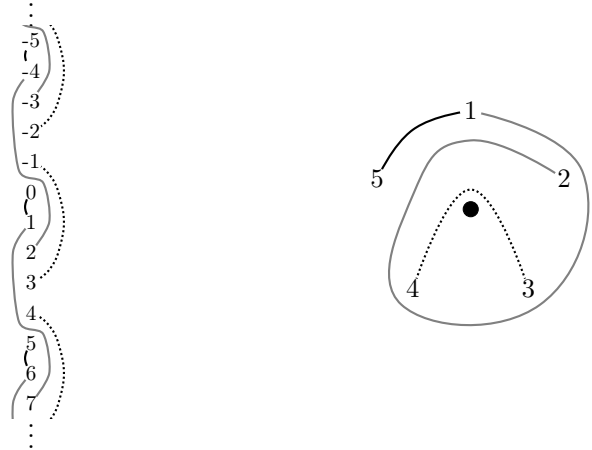
\begin{figure}[ht]
				\centering
				\begin{tikzpicture}[scale=0.4, baseline=(current bounding box.center)]
					\def\n{5}

					\node at (0, -7.8) {$\vdots$};
					\node at (0, 6.1) {$\vdots$};

					\begin{scope}
						\clip (-2,-7.5) rectangle (2,5.5);

						\foreach \i in {-8,...,8}{
							\coordinate (\i) at (0,-\i);
						}

						\draw[thick]
						plot [ smooth, tension=0.8 ] coordinates {(0) (-0.2,-0.5) (1)}
						plot [ smooth, tension=0.8 ] coordinates {(-5) (-0.2,4.5) (-4)}
						plot [ smooth, tension=0.8 ] coordinates {(5) (-0.2,-5.5) (6)};

						\draw[gray, thick]
						plot [ smooth, tension=0.5 ] coordinates {(1) (-0.6,-2) (-0.4,-4.3) (0.4, -4.7) (0.6, -6) (7)}
						plot [ smooth, tension=0.5 ] coordinates {(6) (-0.6,-7) (-0.4,-9.3) (0.4, -9.7) (0.6, -11) (0,-12)}
						plot [ smooth, tension=0.5 ] coordinates {(-4) (-0.6,3) (-0.4,0.7) (0.4, 0.3) (0.6, -1) (2)}
						plot [ smooth, tension=0.5 ] coordinates {(0,9) (-0.6,8) (-0.4,5.7) (0.4, 5.3) (0.6, 4) (-3)};

						\draw[densely dotted, thick]
						plot [ smooth, tension=1 ] coordinates {(4) (0.9,-5) (0.9,-7) (8)}
						plot [ smooth, tension=1 ] coordinates {(-1) (0.9,0) (0.9,-2) (3)}
						plot [ smooth, tension=1 ] coordinates {(-6) (0.9,5) (0.9,3) (-2)};

						\foreach \i in {-8,...,8}{
							\node[circle, fill=white, inner sep=0pt, scale=0.8] at (\i) {\i};
						}
					\end{scope}
				\end{tikzpicture} \hspace*{0.2\textwidth}
				\begin{tikzpicture}[scale=1.3, baseline=(current bounding box.center)]
					\def\n{5}

					\newcommand{\ang}[1]{{-(#1-1)*360/\n + 90}}

					\node[draw, circle, fill, inner sep=2pt, black] at (0,0) {};
					\foreach \i in {1,...,\n}{
						\coordinate (\i) at (\ang{\i}:1);
					}

					\draw[thick]	plot [ smooth, tension=0.8 ] coordinates {(5) (\ang{5.5}:1) (1)};
					\draw[thick, gray]	plot [ smooth, tension=0.8 ] coordinates {(1) (\ang{2}:1.2) (\ang{3}:1.2) (\ang{4}:1.2) (\ang{5}:0.6) (\ang{1}:0.7) (2)};
					\draw[thick, densely dotted]	plot [ smooth, tension=0.8 ] coordinates {(4) (\ang{1}:0.2) (3)};

					\foreach \i in {1,...,\n}{
						\node[circle, fill=white, inner sep=0pt] at (\i) {$\i$};
					}

				\end{tikzpicture}
				\caption{Two representations of the same cyclic noncrossing arc diagram on $5$ points}
				\label{fig: diagramme d'arcs non croisée deux représentations}
			\end{figure}

			\subparagraph{Combinatorial description} We also recall from \cite{barkley2025extendedweakorderaffine} that an arc $\alpha$ can be entirely defined by a tuple $(p,q,L,R)$ where $p$ is the initial point of $\alpha$, $q$ the final point of $\alpha$, $L$ (resp. $R$) the set of points between $p$ and $q$ that are on the left (resp. right) of $\alpha$. Using the «~cyclic~» representation, replace «~left~» and «~right~» with «~outwards the center~» and «~inwards the center~» respectively.

			Using this description, two distinct arcs $\alpha = (p,q,L,R)$ and $\alpha' = (p', q', L', R')$ intersect if and only if
			\begin{equation} \label{eq: critère de croisement d'arcs}
				((L \cup \{p,q\} ) \cap R') \cup (L \cap \{p',q'\}) \neq \emptyset \quad\text{ and }\quad (R \cap (L' \cup \{p',q'\} ) ) \cup (\{p,q\} \cap L') \neq \emptyset.
			\end{equation}
			Indeed, $\alpha$ and $\alpha'$ cross only when at some point, $\alpha$ is on the left of $\alpha'$ and at another point, $\alpha'$ is on the left of $\alpha$.

			\subparagraph{Bijection with TITOs}

			In the finite case, the noncrossing arc diagrams are in bijection with the symmetric group \cite{reading2015noncrossingarcdiagramscanonical}. In the cyclic case, we do not have such a bijection as there are not enough affine permutations. Generalizing to the set of all TITOs using their one-line notation is too much.

			If we only consider \emph{widely generated} TITOs on $\ZZ$, which correspond to certain elements of the extended weak order poset on the TITOs, then we obtain a bijection with the cyclic noncrossing arc diagrams \mbox{\cite[4.12]{barkley2025extendedweakorderaffine}}. 
			The widely generated TITOs have a very easy characterization: they correspond to the TITOs with no two consecutive waxing blocks.

			Moreover, this bijection, which will be denoted in the rest of this paper by $\ncad$, is explicit. Let $\prec$ be a TITO on $\ZZ$. For all $p,q \in \ZZ$ with $p < q$ such that $\affc{p,q} \in \mathrm{Cov}(\prec)$, let $L$ (resp. $R$) be the set of all integers $i$ strictly between $p$ and $q$ such that $i \prec q$ (resp. $p \prec i$). Define $\alpha_{(p,q)}$ to be the arc $(p,q,L,R)$. Then:
			\begin{equation*}
				\ncad(\prec) = \bigcup_{
				\substack{p < q \text{ s.t. } \\ q \prec p \text{ and } \\ \left|[q,p]_\prec\right| = 2}
				} \alpha_{(p,q)} + n\ZZ.
			\end{equation*}

\section{\textsc{Coxeter} sortable elements in the affine symmetric group} \label{sec: c triable du groupe symmetrique affine}

	In this section, we use the correspondence between $c$-aligned elements and $c$-sortable elements \cite[§4]{reading2010sortable} in order to derive a characterization of $c$-sortable elements of the affine symmetric group by pattern avoidance on their one-line notation. The $c$-aligned elements are the elements whose inversion sets, viewed as roots of the positive root system of the \textsc{Coxeter} group, satisfy a certain alignment condition with respect to an orientation of the positive root system given by the choice of $c$.

	\motifaffine*

	\subsection{Orientation of a \textsc{Coxeter} group}

		In this subsection, we will consider $W$ a general \textsc{Coxeter} group with a finite set of generators $S$, and $c$ a \textsc{Coxeter} element of $W$.

		\subsubsection{\textsc{Euler} form and \texorpdfstring{$\omega_c$}{omega\_c} form} \label{sssec: forme omega_c}
			We will define an order on $R$, the set of reflections of $W$, given by the sign of a skew-symmetric form $\omega_c$ defined on the geometric representation $V$. The form $\omega_c$ is the skew-symmetrization of the \textsc{Euler} form $E_c$ that appears in quiver representation theory \cite{derksen2005quiver}.

			Let $A = (a_{s,s'})_{s,s' \in S}$ be the generalized \textsc{Cartan} matrix of $W$ defined by $a_{s,s'} = -2\cos(\pi/m_{s,s'})$. Recall that the set $\{\alpha_s \;|\; s \in S\}$ defines a basis of $V$.
			\begin{definition}[\textsc{Euler} form]
				Let $c$ be a \textsc{Coxeter} element of $W$ and fix a reduced expression $c = s_1 \dots s_n$. The \textsc{Euler} form (associated with the chosen reduced word for $c$) on $V$ is the bilinear form on $V$ defined by
				\begin{equation*}
					\forall s,s' \in S,\quad E_c(\alpha_s, \alpha_{s'}) = \left\lbrace \begin{array}{ll}
						a_{s,s'} & \text{if } s \text{ appears after } s' \text{ in the chosen reduced word for } c, \\
						1 & \text{if } s = s', \\
						0 & \text{if } s \text{ appears before } s' \text{ in the chosen reduced word for } c. \\
					\end{array} \right.
				\end{equation*}
			\end{definition}

			Note that this definition only depends on the element $c$ and not on the choice of a reduced expression of it. Indeed, two reduced expressions of $c$ differ by commutation of commuting generators, and by definition the \textsc{Euler} form on two simple roots associated with commuting generators is zero.

			Reading introduced the bilinear form $\omega_c$ obtained as the skew-symmetrization of the \textsc{Euler} form \mbox{\cite[§3]{reading2010sortable}}. 

			\begin{definition}[Form $\omega_c$]
				Let $c$ be a \textsc{Coxeter} element of $W$ and $E_c$ the \textsc{Euler} form on $V$. We define a skew-symmetric bilinear form $\omega_c$ on $V$ by
				\begin{equation*}
					\forall \beta, \beta' \in V,\quad \omega_c(\beta,\beta') = E_c(\beta,\beta') - E_c(\beta', \beta).
				\end{equation*}
			\end{definition}

			In particular, on the basis of simple roots $(\alpha_s)_{s \in S}$, we have
			\begin{equation*}
				\forall s,s' \in S,\quad \omega_c(\alpha_s, \alpha_{s'}) = \left\lbrace \begin{array}{ll}
					a_{s,s'} & \text{if } s \text{ appears after } s' \text{ in a reduced word for } c, \\
					0 & \text{if } s = s', \\
					-a_{s,s'} & \text{if } s \text{ appears before } s' \text{ in a reduced word for } c. \\
				\end{array} \right.
			\end{equation*}

			In the rest of this paper, we will be mostly interested by the sign of the $\omega_c$ form: $\omega_c(\alpha_s, \alpha_s')$ is positive (resp. negative) if and only if $s$ appears before (resp. after) $s'$ in a reduced word for $c$ and $s$ and $s'$ do not commute.

			In \cite[§3]{reading2010sortable}, the \textsc{Euler} form (and hence the $\omega_c$ form) is defined more generally by using a \emph{symmetrizable generalized \textsc{Cartan} matrix} instead of symmetric ones. However, the sign of the $\omega_c$ form does not depend on the choice of such a generalized \textsc{Cartan} matrix, so we will stay in the case of a symmetric one for easier computations.

			\begin{exemple}
				Let $W$ be a \textsc{Coxeter} group of type $\widetilde{C}_3$, with the following \textsc{Coxeter} graph.
				
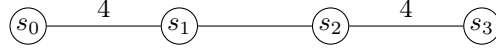
\begin{center}
					\begin{tikzpicture}[scale=2]
						\foreach \i in {0,1,2,3}
						{
							\node[draw, circle, inner sep=1] (\i) at (\i,0) {$s_\i$};
						}
						\draw (0) -- node[above] {4} (1) -- (2) -- node[above] {4} (3);
					\end{tikzpicture}
				\end{center}
				Let $c = s_0 s_2 s_1 s_3$. Then we have $\omega_c(\alpha_{s_2}, \alpha_{s_1}) = a_{s_2, s_1} = -1$ and $\omega_c(\alpha_{s_2}, \alpha_{s_0}) = -a_{s_2, s_0} = 0$. If we chose instead the reduced word $c = s_2 s_0 s_1 s_3$, then we also have $\omega_c(\alpha_{s_2}, \alpha_{s_1}) = a_{s_2, s_1} = -1$ and $\omega_c(\alpha_{s_2}, \alpha_{s_0}) = a_{s_2, s_0} = 0$.
			\end{exemple}

			Let $W$ be a \textsc{Coxeter} group of type $\widetilde{A}_{n-1}$, with generators $S = \{ s_0, \dots, s_{n-1} \}$. Recall it is isomorphic to the affine symmetric group by identifying each generator $s_i$ with the affine transposition $\affc{i,i+1}_0$ if $i > 0$ and $s_0$ to $\affc{n,1}_1$.

			In order to simplify notation, for all $i \in \llbracket 0, n-1 \rrbracket$ we shall simply write $\alpha_i$ instead of $\alpha_{s_i}$. Let $c \in W$ be a \textsc{Coxeter} element given by one of its reduced word on the generators. Then the $\omega_c$ form on $V$ is
			\begin{equation*}
				\forall i,j \in \llbracket 0,n-1 \rrbracket,\quad \omega_c(\alpha_i, \alpha_j) = \left\lbrace \begin{array}{ll}
					1 & \text{if } s_i \text{ appears before } s_j \text{ in } c \text{ and } j-i \equiv \pm1 \bmod n, \\
					-1 & \text{if } s_i \text{ appears after } s_j \text{ in } c \text{ and } j-i \equiv \pm1 \bmod n, \\
					0 & \text{otherwise.}
				\end{array} \right.
			\end{equation*}

			Once again, this definition does not depend on the reduced word chosen for $c$. The $\omega_c$ form on the simple roots indicates whether or not two simple roots of consecutive indices modulo $n$ appear in the «~natural~» order or not in any reduced word for $c$.

			We will now focus for the rest of this subsection on the computation of the form $\omega_c$ on pairs of positive roots. We start by rewording the value on the simple roots in terms of the sets $L_c$ and $R_c$ defined by $c$.

			\begin{lemme} \label{lem: omega c ai-1 ai}
				Let $i \in \llbracket 0, n-1 \rrbracket$. We have
				\begin{equation*}
					\omega_c(\alpha_{i-1}, \alpha_{i}) = \left\lbrace \begin{array}{ll}
						1 & \text{if } i \in L_c \\
						-1 & \text{if } i \in R_c
					\end{array} \right. = (-1)^{\ddelta_{i \in R_c}}
				\end{equation*}
				where $s_{-1} = s_{n-1}$ and $\ddelta$ is the \textsc{Kronecker} symbol.
			\end{lemme}

			\begin{proof}
				If $\omega_c(\alpha_{s_{i-1}}, \alpha_{s_i}) = 1$, we have $s_{i-1}$ preceding $s_i$ in $c$, which means we have $c = u s_{i-1} v s_i w$ a reduced word where $u$, $v$ and $w$ are also reduced words on disjoint alphabets. Since all the generators except $s_i$ and $s_{i-1}$ fix $i$, we have $s_iw(i) = i+1$. Since $u s_{i-1} v$ does not contain $s_i$, we have $u s_{i-1} v(i+1) \geq i+1$. Therefore, $c(i) \geq i+1 > i$: this means that $i \in L_c$ (see Paragraph \ref{par: elements c-triables} of Section \ref{sec: background}).

				Similarly, if $\omega_c(\alpha_{s_{i-1}}, \alpha_{s_i}) = -1$, we have $c = u' s_i v' s_{i-1} w'$ a reduced expression for $c$ and $c(i) \leq i-1 < i$ which means that $i \in R_c$.
			\end{proof}

			Recall that if $t = \affc{i,j}_p$ is a reflection of $W$ represented as an affine transposition (with $i,j \in \llbracket 0,n-1\rrbracket$ and $p \geq \ddelta_{i > j}$), then its corresponding positive root is $\beta_t = \beta_{\affc{i,j}_0} + p\delta$ if $i < j$ and $\beta_t = -\beta_{\affc{j,i}_0} + p\delta$ otherwise, with $\beta_{\affc{a,b}_0}$ being equal to the sum of all the simple roots $\alpha_{s_k}$ for $k \in \llbracket a,b-1 \llbracket$, and $\delta$ being equal to the sum of all simple roots.

			This means we can compute the form $\omega_c$ on any two roots using its bilinearity.

			\begin{exemple}
				Let $n = 6$ and $c = s_3s_0s_1s_2s_5s_4 = \affc{1,2,4}_{[1]} \affc{5,3,0}_{[-1]}$. Let $t = \affc{3,1}_2$ and $r = \affc{2,5}_1$. Recall that $\delta = \sum_{k = 0}^5 \alpha_k$. We have
				\begin{equation*}
					\begin{array}{rcl}
						\omega_c(\beta_t, \beta_r) & = & \omega_c(-\alpha_1 -\alpha_2 + 2\delta, \alpha_2 + \alpha_3 + \alpha_4 + \delta) \\
						& = & \omega_c(\alpha_0, \alpha_1) + \omega_c(\alpha_1, \alpha_2) - 2\omega_c(\alpha_2, \alpha_3) - 2\omega_c(\alpha_4, \alpha_5) \\
						& = & 1 + 1 + 2 + 2 = 6.
					\end{array}
				\end{equation*}
			\end{exemple}

			As we see, the computations can become a bit long even for simple examples, especially with larger values of $n$. A general formula might be derived, but we will only focus on the case where $t$ and $r$ do not commute as this will suffice for our needs later on.

			We will start with finite reflections, i.e. reflections of the form $\affc{i,j}$ with $i,j \in \llbracket 0,n-1 \rrbracket$ distinct.

			\begin{lemme} \label{lem: omega_c(ij0 jk0)}
				Let $i,j,k \in \llbracket 0,n \rrbracket$ be distinct modulo $n$. Write $a$ the element of $\{i,j,k\}$ that is neither maximal or minimal. We have
				\begin{equation*}
					\omega_c(\beta_{\affc{i,j}}, \beta_{\affc{j,k}}) = (-1)^{\ddelta_{a \in R_c} + \ddelta_{k < i}}.
				\end{equation*}
			\end{lemme}

			\begin{proof}
				Let us calculate $\omega_c(\beta_{\affc{i,j}}, \beta_{\affc{j,k}})$ for all six possible orders of the integers $i$, $j$ and $k$. Since the $\omega_c$ form is skew-symmetric, we have
				\begin{equation*}
					\omega_c(\beta_{\affc{i,j}}, \beta_{\affc{j,k}}) = - \omega_c(\beta_{\affc{k,j}}, \beta_{\affc{j,i}}).
				\end{equation*}
				Therefore, the cases $k < i$ can all be obtained by reversing the sign of the cases $i < k$. We can assume that $i < k$ and we are left with only three cases.
				\begin{itemize}
					\item If $i < j < k$, then $a = j$. Using Lemma \ref{lem: omega_c(ij0 jk0)}, we have
					\begin{equation*}
						\omega_c(\beta_{\affc{i,j}}, \beta_{\affc{j,k}})
							= \omega_c(\alpha_i + \dots + \alpha_{j-1}, \alpha_j + \dots + \alpha_{k - 1})
							= \omega_c(\alpha_{j-1}, \alpha_j)
							= (-1)^{\ddelta_{a \in R_c}}.
					\end{equation*}

					\item If $j < i < k$. We have $a = i$. Since $\omega_c$ is skew symmetric, we have:
					\begin{equation*}
						\omega_c(\beta_{\affc{i,j}}, \beta_{\affc{j,k}})
							= \omega_c(\beta_{\affc{i,j}}, \beta_{\affc{j,i}} + \beta_{\affc{i,k}})
							= \omega_c(\beta_{\affc{j,i}}, \beta_{\affc{i,k}})
							= (-1)^{\ddelta_{a \in R_c}}.
					\end{equation*}

					\item If $i < k < j$, we have $a = k$ and
					\begin{equation*}
						\omega_c(\beta_{\affc{i,j}}, \beta_{\affc{j,k}})
							= \omega_c(\beta_{\affc{i,k}} + \beta_{\affc{k,j}}, \beta_{\affc{j,k}})
							= \omega_c(\beta_{\affc{i,k}}, \beta_{\affc{k,j}})
							= (-1)^{\ddelta_{a \in R_c}}.
					\end{equation*}
				\end{itemize}
				Hence the result.
			\end{proof}

			Now we need to compute the form $\omega_c$ between a finite reflection and $\delta$.

			\begin{lemme} \label{lem: omega_c(ij0 delta)}
				Let $i,j \in \llbracket 0, n-1 \rrbracket$ be distinct. We have
				\begin{equation*}
					\omega_c(\beta_{\affc{i,j}}, \delta) = 2(-1)^{\ddelta_{i > j}}(\ddelta_{i \in R_c} - \ddelta_{j \in R_c}).
				\end{equation*}
			\end{lemme}

			\begin{proof}
				If $i < j$, using the fact $\omega_c$ is bilinear and skew-symmetric we obtain using Lemma \ref{lem: omega_c(ij0 jk0)}
				\begin{equation*}
					\begin{split}
						\omega_c(\beta_{\affc{i,j}}, \delta)
						& = \omega_c(\beta_{\affc{i,j}_0}, \beta_{\affc{0,i}_0}) + \omega_c(\beta_{\affc{i,j}_0}, \beta_{\affc{i,j}_0}) + \omega_c(\beta_{\affc{i,j}_0}, \beta_{\affc{j,n}_0}) \\
						& = - (-1)^{\ddelta_{i \in R_c}} + 0 + (-1)^{\ddelta_{j \in R_c}} = 2 \left( \ddelta_{i \in R_c} - \ddelta_{j \in R_c} \right).
					\end{split}
				\end{equation*}
				And if $j < i$, we have $\omega_c(\beta_{\affc{i,j}}, \delta) = \omega_c(\beta_{\affc{j,i}_0}, \delta)$ which is the opposite of the result of the previous calculation, hence the result.
			\end{proof}

			We can now give a formula for the value of the form $\omega_c$ applied on two positive roots associated with affine reflections that share an index.

			\begin{proposition} \label{prop: omegac ijp jkq}
				Let $i,j,k \in \llbracket 0, n-1 \rrbracket$ be distinct and $p,q \in \NN$ such that $p \geq \ddelta_{i > j}$ and $q \geq \ddelta_{j > k}$. Denote by $\sigma$ the permutation of $\mathfrak{S}(\{i,j,k\})$ such that $\sigma(i) < \sigma(j) < \sigma(k)$ (in particular, the integer $a$ from Lemma \ref{lem: omega_c(ij0 jk0)} is $\sigma(j)$). We have
				\begin{equation*}
					\omega_c(\beta_{\affc{i,j}_p}, \beta_{\affc{j,k}_q}) = \varepsilon(\sigma) (-1)^{\ddelta_{\sigma(j) \in R_c}} -2p \left( \ddelta_{j \in R_c} - \ddelta_{k \in R_c} \right) + 2q \left( \ddelta_{i \in R_c} - \ddelta_{j \in R_c} \right).
				\end{equation*}
			\end{proposition}

			\begin{proof}
				Using Lemmas \ref{lem: omega_c(ij0 jk0)} and \ref{lem: omega_c(ij0 delta)}, we calculate
				\begin{equation*}
					\begin{split}
						\omega_c(\beta_{\affc{i,j}_p}, \beta_{\affc{j,k}_q})
						& = \omega_c\left( (-1)^{\ddelta_{i > j}} \beta_{\affc{i,j}} + p\delta, (-1)^{\ddelta_{j > k}} \beta_{\affc{j,k}} + q\delta \right) \\
						& = (-1)^{\ddelta_{i > j} + \ddelta_{j > k}} \omega_c(\beta_{\affc{i,j}}, \beta_{\affc{j,k}})
						  + (-1)^{\ddelta_{i > j}} q \omega_c(\beta_{\affc{i,j}}, \delta)
						  + (-1)^{\ddelta_{j > k}} p \omega_c(\delta, \beta_{\affc{j,k}})
						  + pq \omega_c(\delta, \delta) \\
						& = (-1)^{\ddelta_{i > j} + \ddelta_{j > k} + \ddelta_{k < i}} (-1)^{\sigma(j) \in R_c}
						  + 2q(\ddelta_{i \in R_c} - \ddelta_{j \in R_c})
						  - 2p(\ddelta_{j \in R_c} - \ddelta_{k \in R_c})
						  + 0 \\
						& = \varepsilon(\sigma) (-1)^{\ddelta_{\sigma(j) \in R_c}} -2p \left( \ddelta_{j \in R_c} - \ddelta_{k \in R_c} \right) + 2q \left( \ddelta_{i \in R_c} - \ddelta_{j \in R_c} \right).
					\end{split}
				\end{equation*}
			\end{proof}

			\begin{exemple} \label{ex: calcul de omega_c pour un grand n}
				Let $n = 14$ and $c = \affc{1,4,6,7,8,13,14}_{[1]} \affc{12,11,10,9,5,3,2}_{[-1]}$. We have
				\begin{equation*}
					\left\lbrace
					\begin{array}{*5{@{\;}l}}
					\omega_c(\beta_{\affc{3,11}_{12}}, \beta_{\affc{11,6}_{31}}) & = & (-1) (-1)^0 - 2 \times 12 (1 - 0) + 2 \times 31 (1 - 1) & = & - 25 \\
					\omega_c(\beta_{\affc{2,1}_{16}}, \beta_{\affc{1,14}_{0}}) & = & (-1) (-1)^1 - 2 \times 16 (0 - 0) + 2 \times 0 (1 - 0) & = & 1 \\
					\omega_c(\beta_{\affc{14,5}_{18}}, \beta_{\affc{5,13}_{21}}) & = & 1 (-1)^0 - 2 \times 18 (1 - 0) + 2 \times 21 (0 - 1) & = & -77.
					\end{array}
					\right.
				\end{equation*}
			\end{exemple}

			What is important to us in the rest of this section is the sign of the form $\omega_c$ (positive, negative or zero), not the precise value it takes. Notice that the formula of Proposition \ref{prop: omegac ijp jkq} shows that the value of $\omega_c(\beta_{\affc{i,j}_p}, \beta_{\affc{j,k}_q})$ is always an odd number, and therefore its sign cannot be zero.

			\begin{corollaire}[sign of $\omega_c$] \label{cor: sign of omega_c}
				Let $i,j,k \in \llbracket 0, n-1 \rrbracket$ be distinct and $p,q \in \NN$ such that $p \geq \ddelta_{i > j}$ and $q \geq \ddelta_{j > k}$. The sign of $\omega_c(\beta_{\affc{i,j}_p}, \beta_{\affc{j,k}_q})$ is given in Table \ref{tbl: sign of omega_c} and only depends on the relative order of $i$, $j$ and $k$ and whether they belong to $L_c$ or $R_c$.
				\begin{table}[ht!]
				\begin{equation*}
					\def\arraystretch{1.2}\begin{array}{|c|@{\hspace*{1pt}}|*6{c|}}
					\hline
					(i,j,k) & i < j < k & i < k < j & k < i < j & k < j < i & j < k < i & j < i < k \\
					\hline &&&&&& \\[\dimexpr 1pt -\arraystretch\normalbaselineskip]
					\hline L_c \times L_c \times L_c & + & - & + & - & + & - \\
					\hline L_c \times L_c \times R_c & + & + & + & + & + & + \\
					\hline L_c \times R_c \times L_c & - & - & - & - & - & - \\
					\hline L_c \times R_c \times R_c & - & - & - & - & - & - \\
					\hline R_c \times L_c \times L_c & + & + & + & + & + & + \\
					\hline R_c \times L_c \times R_c & + & + & + & + & + & + \\
					\hline R_c \times R_c \times L_c & - & - & - & - & - & - \\
					\hline R_c \times R_c \times R_c & - & + & - & + & - & + \\
					\hline
					\end{array}
				\end{equation*}
				\vspace*{-\baselineskip}
				\caption{Sign of $\omega_c(\beta_{\affc{i,j}_p}, \beta_{\affc{j,k}_q})$ for $p \geq \ddelta_{i > j}$ and $q \geq \ddelta_{j > k}$.}
				\label{tbl: sign of omega_c}
				\end{table}

			\end{corollaire}

			Notice that, except in the case where all three of $i$, $j$ and $k$ belong to the same set $L_c$ or $R_c$, the sign of $\omega_c(\beta_{\affc{i,j}_p}, \beta_{\affc{j,k}_q})$ only depends on the occurrence of $j$ in $L_c$ or $R_c$ (positive if it is in $L_c$, negative if it is in $R_c$).

			Using the fact that $\omega_c$ is skew-symmetric and Lemma \ref{lem: omega_c(ij0 delta)} we can easily compute the value on pairs of reflections sharing both coefficients when written as affine transpositions. We will use the following case repeatedly, so it is worth making it explicit.

			\begin{proposition} \label{prop: sign of omega_c on A1 tilda parabolic generators}
				Let $i,j \in \llbracket 0, n-1 \rrbracket$ such that $i < j$. Then \begin{equation*}
					\omega_c(\beta_{\affc{i,j}_0}, \beta_{\affc{j,i}_1}) = 2(\ddelta_{i \in R_c} - \ddelta_{j \in R_c}).
				\end{equation*}
			\end{proposition}

			\begin{proof}
				We have
				\begin{equation*}
					\omega_c(\beta_{\affc{i,j}_0}, \beta_{\affc{j,i}_1}) = \omega_c(\beta_{\affc{i,j}_0}, \delta - \beta_{\affc{i,j}_0}) = \omega_c(\beta_{\affc{i,j}_0}, \delta)
				\end{equation*}
				and we conclude with Lemma \ref{lem: omega_c(ij0 delta)}.
			\end{proof}

			Note that in this case, the sign can be zero (if and only if $i$ and $j$ belong to the same set $L_c$ or $R_c$).

			\begin{exemple}
				Let $n = 14$ and $c = \affc{1,4,6,7,8,13,14}_{[1]} \affc{12,11,10,9,5,3,2}_{[-1]}$ as in example \ref{ex: calcul de omega_c pour un grand n}. We can easily check that the signs of the values computed there match those from Table \ref{tbl: sign of omega_c}. Also, we have thanks to Proposition \ref{prop: sign of omega_c on A1 tilda parabolic generators}:
				\begin{equation*}
					\left\lbrace
					\begin{array}{*5{@{\;}l}}
					\omega_c(\beta_{\affc{4,14}_{0}}, \beta_{\affc{14,4}_{1}}) & = & 2(0 - 0) & = & 0 \\
					\omega_c(\beta_{\affc{4,10}_{0}}, \beta_{\affc{10,4}_{1}}) & = & 2(0 - 1) & = & -2 \\
					\omega_c(\beta_{\affc{11,14}_{0}}, \beta_{\affc{14,11}_{1}}) & = & 2(1 - 0) & = & 2 \\
					\end{array}
					\right.
				\end{equation*}
			\end{exemple}

		\subsubsection{Generalized rank two parabolic subgroups} \label{sssec: sgpgr2nc}

			Let $t$ and $t'$ be two distinct reflections in $W$. The subgroup $\langle t,t' \rangle$ is a dihedral reflection subgroup of $W$, but there may be other dihedral reflection subgroups containing both $t$ and $t'$ (consider for example $W$ of type $\widetilde{A}_1$, which already is dihedral). On the other hand, there may not exist a rank two parabolic subgroup of $W$ containing both $t$ and $t'$. For example if $W$ is of type $\widetilde{A}_2$, all rank 2 parabolic subgroups are finite but $\langle s_0, s_1s_2s_1 \rangle$ is infinite.

			In this section, following \cite{reading2010sortable}, we use \emph{generalized rank two parabolic subgroups} of $W$, which will solve both issues. It happens that some complications arise when dealing with ranks larger than 2: for generalized parabolic subgroups of arbitrary rank, one can define them as maximal rank $k$ reflection subgroups \cite{DyerMaximalRankKReflectionSubgroups}. In this paper, we only use generalized rank two parabolic subgroups.

			\begin{definition}[Generalized rank two parabolic subgroup]
				Let $t,t' \in R$ be two distinct reflections of $W$. The \emph{generalized rank two parabolic subgroup} containing $t$ and $t'$ is the subgroup of $W$ generated by the set of all reflections $r \in R$ such that $\beta_r \in \text{Span}(\beta_t, \beta_{t'})$.
			\end{definition}

			We use the notion of \emph{canonical generators} of a reflection subgroup of $W$: if $W'$ is a reflection subgroup of $W$, there exists a unique subset $T \subset R \cap W'$ such that for any reflection $r \in R \cap W'$, the positive root $\beta_r$ lies in the positive span of $\{ \beta_t \;|\; t \in T \}$ and no reflection $t \in T$ lies in the positive span of $\{ \beta_r \;|\; r \in R \cap W' \setminus \{r\} \}$ \cite{Deodhar1989, Dyer1990}. The elements of $T$ are called the \emph{canonical generators} of $W'$. If $W'$ is a standard parabolic subgroup $W' = \langle S' \rangle$ where $S' \subset S$, then $S'$ is the set of canonical generators of $W'$. In general, the number of canonical generators is not under control \cite[2.13]{reading2010sortable}, but in the case of generalized rank two parabolic subgroups defined above, they indeed have exactly two canonical generators.

			\begin{proposition}[{\cite[2.11]{reading2010sortable}}] \label{prop: rank 2 generalized parabolic subgroups are of rank 2}
				Let $W'$ be a generalized rank two parabolic subgroup of $W$. Then $W'$ has exactly two canonical generators. Moreover, if $\{r_1, r_2\}$ are the canonical generators of $W'$, then
				\begin{equation*}
					W' \cap R = \{ r_1(r_2r_1)^k \;|\; k \in \NN \} \cup \{ (r_2r_1)^kr_2 \;|\; k \in \NN \}.
				\end{equation*}
			\end{proposition}

			Let $t_1$ and $t_2$ be two distinct reflections of $W$. If there exists a rank two parabolic subgroup $W'$ of $W$ containing $t_1$ and $t_2$, then $W'$ is the generalized rank two parabolic subgroup containing $t_1$ and $t_2$. This implies that all parabolic subgroups of $W$ of rank two are generalized rank two parabolic subgroups. The converse is not true in general, but it is when $W$ is finite. It is a consequence of \cite[2.9]{reading2010sortable}, because when $W$ is finite, the \textsc{Tits} cone is the whole space.

			\begin{exemple}
				Let $W$ be of type $\widetilde{G}_2$ with the following \textsc{Coxeter} diagram:
				
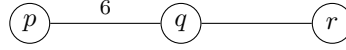
\begin{center}
					\begin{tikzpicture}[sommet/.style={draw, circle, inner sep=0, minimum size=15}, scale=2]
						\node[sommet] (1) at (0,0) {$p$};
						\node[sommet] (2) at (1,0) {$q$};
						\node[sommet] (3) at (2,0) {$r$};

						\draw (1) -- node[above] {\footnotesize$6$} (2) -- (3);
					\end{tikzpicture}
				\end{center}

				Let $t = qpqpq$ and $t' = rtr$. We have $\beta_t = \sqrt{3} \alpha_p + 2\alpha_q$ and $\beta_{t'} = \sqrt{3} \alpha_p + 2 \alpha_q + 2\alpha_r$ (recall that we choose a symmetric generalized \textsc{Cartan} matrix). Since the imaginary root $\delta$ is equal to $\sqrt{3}\alpha_p + 2 \alpha_q + \alpha_r$, we have $\beta_t = \delta - \alpha_r$ and $\beta_{t'} = \delta + \alpha_r$. The positive roots in $\text{Span}(\beta_t, \beta_{t'})$ therefore are
				\begin{equation*}
					\{\beta_t + m\delta \;|\; m \in \NN \} \cup \{ \alpha_r + m\delta \;|\; m \in \NN \}.
				\end{equation*}

				The generalized rank two parabolic subgroup $W'$ containing $t$ and $t'$ has therefore canonical generators $r$ and $t$ and satisfies
				\begin{equation*}
					W' \cap R = \{ r, rtr, rtrtr, rtrtrtr, \dots, trtrtrt, trtrt, trt, t \}.
				\end{equation*}

				Moreover, since $W'$ contains an infinite number of reflections, it is an infinite generalized rank two parabolic subgroup of $W$. Because all proper parabolic subgroups of $W$ are finite, then $W'$ is a generalized rank two parabolic subgroup of $W$ that is not a parabolic subgroup of $W$.

			\end{exemple}

			In light of Proposition \ref{prop: rank 2 generalized parabolic subgroups are of rank 2}, we will associate with each generalized rank two parabolic subgroup $W'$ of $W$ a sequence of reflections.

			\begin{definition} \label{def: sequence of reflextions associated to a generalized parabolic subgroup}
				Let $W'$ be a generalized rank two parabolic subgroup of $W$ with canonical generators $r_1$ and $r_2$. Note $m = \text{ord}(r_2r_1) \in \NN \cup \{\infty\}$ and consider two sequences $(u_k)_{k < m}$ and $(u_{m-k})_{k < m}$ defined by
				\begin{equation*}
					\forall k \in \llbracket 0, m \llbracket,\; u_{k+1} = r_1(r_2r_1)^k \quad \text{ and } \quad u_{m-k} = (r_2r_1)^kr_2.
				\end{equation*}
				When $m$ is finite, since $(r_1r_2)^m = 1$ and both $r_1$ and $r_2$ are reflections, we have for all $1 \leq k \leq m$,
				\begin{equation*}
					r_1 r_2 r_1 \cdots = r_2 r_1 r_2 \cdots
				\end{equation*}
				with $2k-1$ terms on the left hand side and $2(m-k)+1$ terms on the right hand side. This means that  $u_k = u_{m - (m-k)}$ and both these sequences coincide but are reversed. We will refer to these two sequences as one, which we denote $(u_k)_k$, implicitly indexing it by $\llbracket 1,m \rrbracket$ if $m$ is finite and by $\NN^* \sqcup (\infty - \NN)$ if $m = +\infty$.
			\end{definition}

			\begin{proposition}[{\cite[2.11]{reading2010sortable}}]
				Let $W'$ be a generalized rank two parabolic subgroup of $W$ and $m \in \NN \cup \{ \infty \}$ and $(u_k)_k$ defined in Definition \ref{def: sequence of reflextions associated to a generalized parabolic subgroup}. Then $W' \cap R = \{ u_1, \dots, u_m \}$.
			\end{proposition}

			Generalized rank two parabolic subgroups can be used to characterize inversion sets of elements of $W$.

			\begin{proposition}[{\cite[2.17]{reading2010sortable}}] \label{prop: ensemble d'inversion ssi segment initial ou final}
				Let $I$ be a finite subset of $T$. The following are equivalent:
				\begin{enumerate}
					\item There exists an element $w \in W$ such that $I = N(w)$.
					\item For every noncommutative generalized rank two parabolic subgroup $W'$ of $W$, with canonical generators $r_1$ and $r_2$, the intersection $W \cap I$ is an initial or a final segment of the sequence $(u_k)_k$.
				\end{enumerate}
			\end{proposition}

			We will focus for the rest of this subsection on the computation of all the generalized rank two parabolic subgroups of $W$ of type $\widetilde{A}$, and give their canonical generators.

			\begin{proposition} \label{prop: generalized rank two parabolic subgroup of the affine symmetric group}
				Let $W$ be a \textsc{Coxeter} group of type $\widetilde{A}_{n-1}$, with generators $s_0, \dots, s_{n-1}$, identified as the affine symmetric group $\widehat{\mathfrak{S}}_n$. Let $W'$ be a subgroup of $W$. Then $W'$ is a noncommutative generalized rank two parabolic subgroup of $W$ if and only if $W'$ is
				\begin{itemize}
					\item either finite of type $A_2$ with canonical generators $\{\affc{i,j}_p, \affc{j,k}_q\}$ with $i,j,k \in \llbracket 1,n \rrbracket$ distinct and $p \geq \ddelta_{i > j}$ and $q \geq \ddelta_{j > k}$. In this case, we have
					\begin{equation*}
						W' \cap R = \{ \affc{i,j}_p, \affc{i,k}_{p+q}, \affc{j,k}_q \}.
					\end{equation*}
					\item either infinite of type $\widetilde{A}_1$ with canonical generators $\{\affc{i,j}_0, \affc{j,i}_1\}$ with $i,j \in \llbracket 1,n \rrbracket$ such that $i < j$. In this case, we have
					\begin{equation*}
						W' \cap R = \{ \affc{i,j}_p \;|\; p \in \NN \} \sqcup \{ \affc{j,i}_p \;|\; p \in \NN^* \}.
					\end{equation*}
				\end{itemize}
			\end{proposition}

			\begin{proof}
				Let $W'$ be a noncommutative generalized rank two parabolic subgroup of $W$, and let $t_1 = \affc{a,b}_u$ and $t_2 = \affc{c,d}_v$ be its canonical generators. We have $a,b,c,d \in \llbracket 1,n \rrbracket$, $a \neq b$, $c \neq d$, $p \geq \ddelta_{a > b}$ and $q \geq \ddelta_{c > d}$. We want to show that $W'$ satisfies one of the two conditions of the proposition.

				\begin{itemize}
					\item If $W'$ is finite, then the order of $t_1 t_2$ is finite. This means that $\left| \{ a,b,c,d \} \right| \geq 3$. Indeed, if $\{a,b\} = \{c,d\}$, then $t_1t_2$ is of infinite order: using the cycle with shift notation, we have
					\begin{equation*}
						\affc{a,b}_u \affc{a,b}_v = \affc{a}_{[v-u]} \affc{b}_{[u-v]} \quad \text{ and } \quad \affc{a,b}_u \affc{b,a}_v = \affc{a}_{[-u-v]} \affc{b}_{[u+v]}.
					\end{equation*}
					We also cannot have $\left| \{ a,b,c,d \} \right| = 4$ since $t_1$ and $t_2$ do not commute. Therefore, $\left| \{ a,b,c,d \} \right| = 3$. This means that exactly one of the following cases is true: $a = c$, $a = d$, $b = c$ or $b = d$.
					Before proceeding to the case by case study, we will show the following result: Let $i,j,k$ be three distinct integers in $\llbracket 1,n \rrbracket$. We have
					\begin{equation} \label{eq: lemme utile dans la preuve des sg paraboliques}
						(-1)^{\ddelta_{i > j}}\beta_{\affc{i,j}} + (-1)^{\ddelta_{j > k}}\beta_{\affc{j,k}} = (-1)^{\ddelta_{i > k}}\beta_{\affc{i,k}}.
					\end{equation}
					If $i < j < k$, we indeed have $\beta_{\affc{i,j}} + \beta_{\affc{j,k}} = \beta_{\affc{i,k}}$. If $j < i < k$, we have $-\beta_{\affc{i,j}} + \beta_{\affc{j,k}} = \beta_{\affc{i,k}}$. If $i < k < j$, we have $\beta_{\affc{i,j}} - \beta_{\affc{j,k}} = \beta_{\affc{i,k}}$. Finally, reversing the orders of the three previous cases negates the sign in front of each term, and hence the result is true for each possible order of the integers $i,j,k$.

					\begin{itemize}
						\item If $a = c$. We will show this case cannot happen. Let
						\begin{equation*}
							t_3 = \left\lbrace
								\begin{array}{ll}
								\affc{b,d}_{v-u} & \text{ if } v - u \geq \ddelta_{b > d} \\
								\affc{d,b}_{u-v} & \text{ otherwise}.
								\end{array}
							\right.
						\end{equation*}
						If $v-u \geq \ddelta_{b > d}$, we have
						\begin{equation*}
							\beta_{t_3} = (v-u)\delta + (-1)^{\ddelta_{b > d}} \beta_{\affc{b,d}}
						\end{equation*}
						and using (\ref{eq: lemme utile dans la preuve des sg paraboliques}) we obtain
						\begin{equation*}
							\beta_{t_1} + \beta_{t_3} = \beta_{t_2}.
						\end{equation*}
						This means that $t_1$ and $t_2$ are not the canonical generators of $W'$ which is absurd. Similarly, if $u-v \geq \ddelta_{d > b}$, we have
						\begin{equation*}
							\beta_{t_2} + \beta_{t_3} = \beta_{t_1}
						\end{equation*}
						and once again a contradiction. Therefore, the case $a = c$ is not possible.

						\item If $b = d$, the same argument using
						\begin{equation*}
							t_3 = \left\lbrace
								\begin{array}{ll}
								\affc{a,c}_{u-v} & \text{ if } u-v \geq \ddelta_{a > c} \\
								\affc{c,a}_{v-u} & \text{ otherwise}
								\end{array}
							\right.
						\end{equation*} also leads to a contradiction.

						\item If $a = d$, the canonical generators of $W'$ are $\{ \affc{c,a}_v, \affc{a,b}_u \}$.

						\item If $b = c$, the canonical generators of $W'$ are $\{ \affc{a,b}_u, \affc{b,c}_v \}$.
					\end{itemize}
					Therefore, the canonical generators of $W'$ are of the form $\{ \affc{i,j}_p, \affc{j,k}_q \}$ with $i,j,k \in \llbracket 1,n \rrbracket$ pairwise distinct, $p \geq \ddelta_{i > j}$ and $q \geq \ddelta_{j > k}$. In this case, we have
					\begin{equation*}
						t_1 t_2 t_1 = \affc{i,j}_p \affc{j,k}_q \affc{i,j}_p = \affc{i,k}_{p+q} \quad\text{ and }\quad t_2 t_1 t_2 = \affc{j,k}_q \affc{i,j}_p \affc{j,k}_q = \affc{i,k}_{p+q}.
					\end{equation*}
					This means $t_1 t_2$ is of order $3$, and $R(W') = \{ t_1, t_1t_2t_1, t_2 \}$. Therefore, $W'$ is a \textsc{Coxeter} group of type $A_2$.

					\item If $W'$ is infinite, then $\left| \{ a,b,c,d \} \right| = 2$ and we have either $(a,b) = (c,d)$ or $(a,b) = (d,c)$. If $(a,b) = (c,d)$, then $v \neq u$ since $t_1$ and $t_2$ are distinct reflections. Without loss of generality, we assume that $u < v$. We can write
					\begin{equation*}
						\beta_{t_1} = u\delta + (-1)^{\ddelta_{a > b}} \beta_{\affc{a,b}} \quad \text{ and } \quad \beta_{t_2} = v\delta + (-1)^{\ddelta_{a > b}} \beta_{\affc{a,b}}.
					\end{equation*}
					Let $t_3 = \affc{a,b}$. We have $\beta_{t_3} = \beta_{\affc{a,b}}$, and thus \begin{equation*}
						\dfrac{u}{v} \beta_{t_2} + \left( 1 - \dfrac{u}{v} \right)\beta_{t_3} = \beta_{t_1}
					\end{equation*}
					which is a positive linear combination of $\beta_{t_2}$ and $\beta_{t_3}$. This means that $t_1$ and $t_2$ are not the canonical generators of $W'$, which is absurd. Therefore, we have $(a,b) = (d,c)$.

					Without loss of generality, we can assume $a < b$. In this case, we have
					\begin{equation*}
						\beta_{t_1} = \beta_{\affc{a,b}_u} = u\delta + \beta_{\affc{a,b}} \quad \text{ and } \quad \beta_{t_2} = \beta_{\affc{b,a}_v} = v\delta -\beta_{\affc{a,v}}.
					\end{equation*}
					We will show that $(u,v) = (0,1)$. Note that $v \geq \ddelta_{b > a} = 1$. If $u \neq 0$, then let $t_3 = \affc{a,b}_0$, we have $t_3 \neq t_1$, $t_3 \neq t_2$ and
					\begin{equation*}
						\beta_{t_3} = \beta_{\affc{a,b}} = \dfrac{1}{uv} \left(v \beta_{t_1} - u \beta_{t_2}\right).
					\end{equation*}
					This contradicts the fact that $t_1$ and $t_2$ are the canonical generators of $W'$, so it is impossible. Hence $p = 0$. Similarly, if we assume that $v > 1$, let $t_4 = \affc{b,a}_1$. We have $t_4 \neq t_1$, $t_4 \neq t_2$ and
					\begin{equation*}
						\beta_{t_4} = \delta - \beta_{\affc{a,b}} = \dfrac{1}{v}\left( -(v-1) \beta_{t_1} + \beta_{t_2} \right).
					\end{equation*}
					Again, this is absurd. Therefore, $v = 1$, and the canonical generators of $W'$ are $\{ \affc{a,b}_1, \affc{b,a}_{-1} \}$. Since $W'$ is an infinite \textsc{Coxeter} group of rank $2$, it is of type $\widetilde{A}_1$.
				\end{itemize}

				Now we prove the converse and show that such subgroups of $W$ are indeed generalized rank two parabolic subgroups.
				\begin{itemize}
					\item If $W'$ is an infinite reflection subgroup of type $\widetilde{A}_1$ of $W$, with canonical generators $t_1 = \affc{i,j}_0$ and $t_2 = \affc{j,i}_{1}$ where $1 \leq i < j \leq n$. We need to show that all roots in $\Phi^+(W)$ belonging to the positive linear span of $\beta_{t_1}$ and $\beta_{t_2}$ are roots in $\Phi^+(W')$.

					Let $\beta \in \Phi^+(W)$ such that $\beta = \lambda \beta_{t_1} + \mu \beta_{t_2}$ with $\lambda, \mu \in \RR_+$. We have
					\begin{equation*}
						\beta = \mu \delta + (\lambda - \mu) \beta_{\affc{i,j}}.
					\end{equation*}
					Since $\beta$ is a positive root of an affine \textsc{Coxeter} root system of type $\widetilde{A}_n$, we have $\lambda - \mu = 1$ and $\mu \in \NN$, or $\lambda - \mu = -1$ and $\mu \in \NN^*$. In the first case, we obtain that $\beta = \beta_{\affc{i,j}_\mu}$, and in the second case $\beta = \beta_{\affc{j,i}_\mu}$. These are all roots in $\Phi^+(W')$, therefore $W'$ is a generalized parabolic subgroup of rank $2$ of $W$.

					\item If $W'$ is a finite reflection subgroup of type $A_2$ with canonical generators $t_1 = \affc{i,j}_p$ and $t_2 = \affc{j,k}_q$ where $i,j,k \in \llbracket 1,n \rrbracket$ are distinct, $p \geq \ddelta_{i > j}$ and $q \geq \ddelta_{j > k}$, we will show that $W'$ is in fact a parabolic subgroup of $W$, and therefore also a generalized one. Let $w_1 \in W$ such that $w_1 \affc{i,j}_p w_1^{-1} = (a,a+1)$ for some $a \in \ZZ$. We also have $w_1 \affc{j,k}_q w_1^{-1} = \affc{a+1,b}$ for some $b \in \ZZ$, and $a$, $a+1$ and $b$ have distinct residue modulo $n$.
					\begin{itemize}
						\item If $b \not\equiv a+2 \bmod n$, then consider $w_2 = \affc{b,a+2}$. We have $w_2\affc{a,a+1}w_2^{-1} = \affc{a,a+1}$ and $w_2\affc{a+1,b}w_2^{-1} = \affc{a+1,a+2}$.
						\item Else, if $n > 3$, then we can choose $c \in \ZZ$ belonging to a different class modulo $n$ than $a$, $a+1$ and $a+2$. Let $w_2 = \affc{b}_{[v]} \affc{c}_{[-v]}$ where $v$ is such that $b + vn = a+2$. We obtain that $w_2 \affc{a,a+1} w_2^{-1} = \affc{a,a+1}$ and $w_2 \affc{a+1,b} w_2^{-1} = \affc{a+1,a+2}$.

						If $n = 3$, we need to be more careful since we cannot use a spare number $c$ like before. We write $a+2 = b + vn$ as before.

						\begin{itemize}
							\item If $v = 3k$ for some $k \in \ZZ$. Consider $w_2$ the affine permutation defined by the window
							\begin{equation*}
								\left[ \begin{array}{*3c}
								a & a+1 & a+2 \\
								a + kn & a+1+kn & a+2 - 2kn
								\end{array} \right].
							\end{equation*}
							We have $w_2 \affc{a,a+1} w_2^{-1} = \affc{a,a+1}$ and $w_2 \affc{a+1, a+2 + vn} w_2^{-1} = \affc{a+1, a+2}$.
							\item If $v = 3k+1$ for some $k \in \ZZ$. Consider $w_2$ the affine permutation defined by the window
							\begin{equation*}
								\left[ \begin{array}{*3c}
								a & a+1 & a+2 \\
								a+1 + kn & a+2+kn & a - 2kn
								\end{array} \right].
							\end{equation*}
							We have $w_2 \affc{a,a+1} w_2^{-1} = \affc{a+1,a+2}$ and $w_2 \affc{a+1, a+2 + vn} w_2^{-1} = \affc{a+2, a+n} = \affc{a+2,a+3}$.
							\item If $v = 3k-1$ for some $k \in \ZZ$. Consider $w_2$ the affine permutation defined by the window
							\begin{equation*}
								\left[ \begin{array}{*3c}
								a & a+1 & a+2 \\
								a+1 + kn & a+kn & a+2 - 2kn
								\end{array} \right].
							\end{equation*}
							We have $w_2 \affc{a,a+1} w_2^{-1} = \affc{a,a+1}$ and $w_2 \affc{a+1, a+2 + vn} w_2^{-1} = \affc{a, a+2-n} = \affc{a,a-1}$.
						\end{itemize}
					\end{itemize}
					In all cases, with $w = w_2w_1$, we obtain that $w t_1 w^{-1}$ and $w t_2 w^{-1}$ are two simple reflections, therefore $W' = \langle t_1, t_2 \rangle$ is a parabolic subgroup of $W$.
				\end{itemize}
			\end{proof}

			\begin{exemple}
				Let $W$ be of type $\widetilde{A}_2$. As in \cite{Hohlweg_2014} or \cite{dyer2016imaginaryconeslimitroots}, we represent its positive root system in Figure \ref{fig: example generalized parabolic subgroups in A2 tilda} using a projective projection, because an affine representation would be of dimension 3 and harder to illustrate. Broadly, each point on the plane of this representation corresponds to a line in the vector space $V$ and positive linear combinations of vectors become barycenters. For example, the convex hull of the three simple roots represents their positive linear span, and the root $\alpha_0 + \alpha_1$ is the midpoint between $\alpha_0$ and $\alpha_1$.

				Let $t_1$ and $t_2$ be two distinct reflections. Then the reflections of $W'$, the generalized rank two parabolic subgroup containing $t_1$ and $t_2$, are the ones associated with all the roots belonging to the intersection of the line going through $\beta_{t_1}$ and $\beta_{t_2}$, and the canonical generators of $W'$ are the end points of the convex hull of these reflections.
				\begin{enumerate}
					\item Let $t_1 = \affc{1,2}_1$ and $t_2 = \affc{3,1}_2$ and consider $W'$ the generalized rank two parabolic subgroup containing $t_1$ and $t_2$. We have $\beta_{t_1} = \alpha_1 + \delta$ and $\beta_{t_2} = \alpha_0 + \delta$. Figure \ref{fig: example generalized parabolic subgroups in A2 tilda} shows in {\color{gray}gray} all the positive roots associated with the reflections of $W'$: they are $\{ \alpha_1 + \delta, \alpha_0 + \alpha_1 + 2\delta, \alpha_0 + \delta \}$. Then $W'$ is of type $A_2$ and its canonical generators are $t_1$ and $t_2$.

					\item Let $t_1 = \affc{2,3}_0$ and $t_2 = \affc{2,3}_1$ and consider $W'$ the generalized rank two parabolic subgroup containing $t_2$ and $t_2$. We have $\beta_{t_1} = \alpha_2$ and $\beta_{t_2} = \alpha_2 + \delta$. Figure \ref{fig: example generalized parabolic subgroups in A2 tilda} shows in \tikz[baseline=(x.base)]{\node[inner sep=0pt, outer sep=0] (x) {dotted}; \draw[thick, densely dotted] ([yshift=-1pt]x.south west) -- ([yshift=-1pt]x.south east);} all the positive roots associated with the reflections of $W'$: they are $\{ \alpha_2 + k\delta \;|\; k \in \NN \} \sqcup \{ \alpha_0 + \alpha_1 + k\delta \;;\; k \in \NN \}$. Then $W'$ is of type $\widetilde{A}_1$ and its canonical generators are $t_1$ and $\affc{3,2}_1$.
				\end{enumerate}

			\end{exemple}

			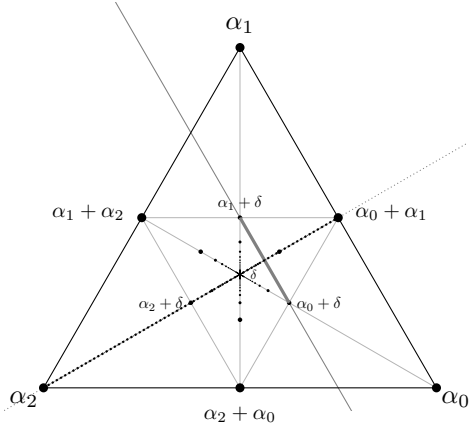
\begin{figure}[ht!]
				\centering

				\begin{tikzpicture}[scale=3]
				\foreach \x in {0,...,30}
				{
				 \pgfmathparse{1/(3*\x + 1)}\edef\y{\pgfmathresult} 
				 \pgfmathparse{1/(3*\x + 2)}\edef\z{\pgfmathresult} 
				 \pgfmathparse{1/(\x+1)}\edef\s{\pgfmathresult} 
				 \foreach \r in {0, 1, 2}
				 {
				  \coordinate (a\r\x) at (-30+\r*120:\y);
				  \coordinate (aa\r\x) at (-30+\r*120:-\z);
				  \draw[draw=none, fill=black] (a\r\x) circle (\s/50);
				  \draw[draw=none, fill=black] (aa\r\x) circle (\s/50);
				 }
				}

				\foreach \x in {0, 60, 120, 180, 240, 300}{
				\draw[fill, line width=0, rotate=\x] (0,0) -- (0.001,0.02) -- (-0.001, 0.02) -- (0,0);
				}

				\draw (90:1) -- (210:1) -- (330:1) -- cycle; 
	      		\draw[ultra thin, opacity=0.3] (-90:0.5) -- (30:0.5) -- (150:0.5) -- cycle;

				\foreach \x in {0, 120, 240}{
				 \draw[ultra thin, opacity=0.3, rotate=\x] (0,0) -- (0,1);
				 \draw[ultra thin, opacity=0.3, rotate=\x] (0,0) -- (0,-0.5);
				}

				\node at (-30:1.1) {$\alpha_0$};
				\node at ( 90:1.1) {$\alpha_1$};
				\node at (210:1.1) {$\alpha_2$};
				\node[scale=0.8, anchor=west] at ( 30: 0.55) {$\alpha_0 + \alpha_1$};
				\node[scale=0.8, anchor=east] at (150: 0.55) {$\alpha_1 + \alpha_2$};
				\node[scale=0.8, anchor=north] at (-90:0.55) {$\alpha_2 + \alpha_0$};
				\node[scale=0.6, anchor=west] at (-30: 0.26) {$\alpha_0 + \delta$};
				\node[scale=0.6, anchor=south] at (90: 0.26) {$\alpha_1 + \delta$};
				\node[scale=0.6, anchor=east] at (210: 0.26) {$\alpha_2 + \delta$};
				\node[scale=0.5] at (0.06,0) {$\delta$};

				\clip (-150:1.2) rectangle (1,1.2);
				\draw[opacity=0.5, shorten <= -10cm, shorten >= -10cm] (a01) -- (a11);
				\draw[gray, very thick] (a01) -- (a11);

				\draw[densely dotted, opacity=0.5, shorten <= -10cm, shorten >= -10cm] (a20) -- (aa20);
				\draw[densely dotted, thick] (a20) -- (aa20);
				\end{tikzpicture}
				\caption{Generalized rank two parabolic subgroups in the positive root system of type $\widetilde{A_2}$.}
				\label{fig: example generalized parabolic subgroups in A2 tilda}
			\end{figure}

	\subsection{Pattern avoidance criterion for \textsc{Coxeter} sortable elements} \label{ssec: critère motif c-triable}

		In this section, we will study the inversion sets of $c$-sortable elements of the affine symmetric group using the notion of $c$-aligned elements defined in \cite[§4]{reading2010sortable}, based on the skew-symmetric form $\omega_c$.

		\subsubsection{\texorpdfstring{$c$}{c}-aligned elements} \label{sssec: c-aligné Sn tilde}

			Let $W$ be a \textsc{Coxeter} group and fix $c$ a \textsc{Coxeter} element of $W$. We can define an orientation on the set of reflections of $W$ using the $\omega_c$ form. This orientation is well behaved with respect to generalized rank two parabolic subgroups as shown in the following result.

			\begin{proposition}[{\cite[4.1]{reading2010sortable}}]
				Let $W'$ be a generalized rank two parabolic subgroup of $W$. Let $(u_k)_k$ be the sequence defined in Definition \ref{def: sequence of reflextions associated to a generalized parabolic subgroup}.

				\begin{itemize}
					\item If $\omega_c(\beta_{u_1}, \beta_{u_m}) > 0$, then $\omega_c(\beta_{u_i}, \beta_{u_j}) > 0$ for all indices $i < j$,
					\item if $\omega_c(\beta_{u_1}, \beta_{u_m}) < 0$, then $\omega_c(\beta_{u_i}, \beta_{u_j}) < 0$ for all indices $i < j$,
					\item if $\omega_c(\beta_{u_1}, \beta_{u_m}) = 0$, then $\omega_c(\beta_{u_i}, \beta_{u_j}) = 0$ for all indices $i,j$. In this case, we say that $W'$ is \emph{not oriented}.
				\end{itemize}

				Without loss of generality, we can always assume that $\omega_c(\beta_{u_1}, \beta_{u_2}) \geq 0$ by exchanging the roles of the canonical generators.
			\end{proposition}

			\begin{remarque}
				By convention, we have $a < \infty - b$ for all $a,b \in \NN$ and $\infty - a < \infty - b$ for all $a,b \in \NN$ such that $b < a$.
			\end{remarque}

			This proposition means that generalized rank two parabolic subgroups are either fully oriented with a total order given by $\omega_c$, or there is no orientation between any two reflections of $W'$.

			\textsc{Reading} defines in \cite[§4]{reading2010sortable} the notion of $c$-aligned elements of a \textsc{Coxeter} group $W$ to be the elements such that their inversions satisfy some alignment properties with respect to the form $\omega_c$. Let $W'$ be a noncommutative generalized rank two parabolic subgroup of $W$ and denote $(u_k)_k$ the sequence from Definition \ref{def: sequence of reflextions associated to a generalized parabolic subgroup}.

			\begin{definition}[$c$-aligned elements] \label{def: biclos c-aligné}
				We say that $w \in W$ is \emph{$c$-aligned with respect to} $W'$ if one of the following conditions is true.
				\begin{enumerate}
					\item $\omega_c(\beta_{u_1}, \beta_{u_2}) > 0$ and $N(w) \cap W'$ is either the empty set, the singleton $\{u_m\}$ or an initial segment of $(u_k)_k$.
					\item $\omega_c(\beta_{u_1}, \beta_{u_2}) = 0$ and $N(w) \cap W'$ is either the empty set or a singleton (necessarily one of the canonical generators of $W'$).
				\end{enumerate}

				We say that $w$ is \emph{$c$-aligned} if it is $c$-aligned with respect to all noncommutative generalized rank two parabolic subgroups.
			\end{definition}

			The usefulness of this notion appears in the following result by \textsc{Reading}.

			\begin{theoreme}[{\cite[§4.3]{reading2010sortable}}]
				Let $w \in W$. Then $w$ is $c$-sortable if and only if $w$ is $c$-aligned.
			\end{theoreme}

			This result enlightens us on the inversions of a $c$-sortable element, which are easily encoded in the one-line notation of an affine permutation. The rest of this section uses this fact to derive a criterion on the one-line notation of an affine permutation to decide whether it is $c$-sortable or not.

		\subsubsection{The case of the affine symmetric group}

			From now on, $W$ is a \textsc{Coxeter} group of type $\widetilde{A}_{n-1}$ which will be identified with the affine symmetric group $\widehat{\mathfrak{S}}_n$, and $c$ is a \textsc{Coxeter} element of $W$. Since we would like to find a criterion on pattern avoidance, a first step is to characterize the inversions of non $c$-sortable elements of $W$.
			\begin{proposition} \label{prop: pas c triable quelles sont les inversions}
				Let $w \in W$. Then $w$ is not $c$-sortable if and only if at least one the following conditions holds.
				\begin{enumerate}
					\item There exists $i,j \in \llbracket 1,n \rrbracket$ such that $i < j$ and
					\begin{enumerate}[label*=\alph*.]
						\item if $(i,j) \in L_c^2 \sqcup R_c^2$, then $\{ \affc{i,j}_1, \affc{j,i}_2 \} \cap N(w) \neq \emptyset$. \label{enum: 1a}
						\item if $(i,j) \in L_c \times R_c$, then $\affc{i,j}_1 \in N(w)$. \label{enum: 1b}
						\item if $(i,j) \in R_c \times L_c$, then $\affc{j,i}_2 \in N(w)$. \label{enum: 1c}
					\end{enumerate}

					\item There exists $i,j,k \in \llbracket 1,n \rrbracket$ distinct, $p \geq \ddelta_{j < i}$ and $q \geq \ddelta_{k < j}$ such that
					\begin{enumerate}[label*=\alph*.]
						\item if $j \in L_c$, then $\{ \affc{i,j}_p, \affc{i,k}_{p+q}, \affc{j,k}_q \} \cap N(w) = \{ \affc{i,k}_{p+q}, \affc{j,k}_q \}$. \label{enum: 2a}
						\item if $j \in R_c$, then $\{ \affc{i,j}_p, \affc{i,k}_{p+q}, \affc{j,k}_q \} \cap N(w) = \{ \affc{i,j}_q, \affc{i,k}_{p+q} \}$. \label{enum: 2b}
					\end{enumerate}
				\end{enumerate}
			\end{proposition}

			\begin{proof}
				Suppose that $w$ is not $c$-sortable. By \textsc{Reading}'s theorem, $w$ is not $c$-aligned. By Definition \ref{def: biclos c-aligné}, there exists a generalized noncommutative rank two parabolic subgroup $W'$ of $W$ such that one of the points of Definition \ref{def: biclos c-aligné} does not hold.
				\begin{enumerate}
					\item If $\omega_c(\beta_{u_1}, \beta_{u_2}) > 0$, then by Proposition \ref{prop: ensemble d'inversion ssi segment initial ou final} $N(w) \cap W'$ is a final segment of $(u_k)_k$ of cardinality at least $2$ different from the set of all inversions of $W'$. We are going to split cases depending on the nature of $W'$ using Proposition \ref{prop: generalized rank two parabolic subgroup of the affine symmetric group}.
					\begin{itemize}
						\item If $W'$ is of type $\widetilde{A}_1$, then by Proposition \ref{prop: sign of omega_c on A1 tilda parabolic generators} we have $u_1 = \affc{i,j}_0$ and $u_{\infty} = \affc{j,i}_1$ with $i < j$ and $(i,j) \in R_c \times L_c$, or $u_1 = \affc{j,i}_1$ and $u_\infty = \affc{i,j}_0$ with $i < j$ and $(i,j) \in L_c \times R_c$. In the first case, a final segment of $(u_k)_k$ of cardinality at least 2 contains $\affc{j,i}_2$ (thus condition \ref{enum: 1c}) and in the second case, it contains $\affc{i,j}_1$ (thus condition \ref{enum: 1b}), hence the result in this case.
						\item If $W'$ is of type $A_2$, then $N(w) \cap W' = \{ u_2, u_3 \}$. By Corollary \ref{cor: sign of omega_c}, we have $u_1 = \affc{i,j}_p$, $u_2 = \affc{i,k}_{p+q}$ and $u_3 = \affc{j,k}_q$ if $i,j,k$ satisfy a «~$+$~» case of table \ref{tbl: sign of omega_c} (call it \emph{case $(+)$}), and $u_1 = \affc{j,k}_q$, $u_2 = \affc{i,k}_{p+q}$ and $u_3 = \affc{i,j}_p$ otherwise (call it \emph{case $(-)$}).

						For now, put aside the case where $(i,j,k) \in L_c^3 \sqcup R_c^3$. In all the other cases, we are in case $(+)$ exactly when $j \in L_c$ and in case $(-)$ exactly when $j \in R_c$. Thus we have condition \ref{enum: 2a} of the proposition.

						Now, if $i,j,k \in L_c$, then we are in the case $(+)$ if $i < j < k$, $j < k < i$ or $k < i < j$ and condition \ref{enum: 2a} of the proposition is held. And in the three other cases on the order of $i,j,k$, we are in the case $(-)$: if $j < i < k$ or $i < k < j$, we have $p+q \geq 1$ and $\affc{i,k}_{p+q}$ belongs to $N(w)$, hence $\affc{i,k}_1 \in N(w)$ by Proposition \ref{prop: ensemble d'inversion ssi segment initial ou final}. Therefore we obtain condition \ref{enum: 1a} of the proposition. And if $k < j < i$, then $p+q \geq 2$ and this time we have $\affc{i,k}_2 \in N(w)$, which is also condition \ref{enum: 1a} of the proposition.

						Finally, the same argument for $i,j,k \in R_c$ gives for each case on the 6 possible orders of these integers one of the conditions of the proposition.
					\end{itemize}

					\item If $\omega_c(\beta_{u_1}, \beta_{u_2}) = 0$, then $W'$ is of type $\widetilde{A}_1$ by Propositions \ref{prop: sign of omega_c on A1 tilda parabolic generators}, \ref{prop: generalized rank two parabolic subgroup of the affine symmetric group} and \ref{cor: sign of omega_c}, and its canonical generators are of the form $\affc{i,j}_0$ and $\affc{j,i}_1$ with $i,j \in \llbracket 1,n \rrbracket$ such that $i < j$ and $(i,j) \in L_c^2 \sqcup R_c^2$. Since $W' \cap N(w)$ is neither empty nor the singleton $\{\affc{i,j}_0\}$ nor the singleton $\{\affc{j,i}_1\}$, then by Proposition \ref{prop: ensemble d'inversion ssi segment initial ou final} $N(w) \cap W'$ contains $\affc{i,j}_1$ or $\affc{j,i}_2$, hence condition \ref{enum: 1a} of the proposition.
				\end{enumerate}
				We proved that if $w$ is not $c$-sortable, then one of the conditions of the proposition holds. Now we prove the converse: if one of the conditions of the proposition holds, then $w$ is not $c$-sortable.
				\begin{enumerate}
					\item If condition $1$ of the proposition holds, consider $W'$ the generalized rank two parabolic subgroup containing $\affc{i,j}_0$ and $\affc{j,i}_1$. By Proposition \ref{prop: sign of omega_c on A1 tilda parabolic generators} and Definition \ref{def: biclos c-aligné}:
					\begin{enumerate}[label*=\alph*.]
						\item if $(i,j) \in L_c^2 \sqcup R_c^2$, then $W'$ is not oriented. Since $N(w) \cap W'$ contains $\affc{i,j}_1$ or $\affc{j,i}_2$, it is of cardinality greater than $2$ by Proposition \ref{prop: ensemble d'inversion ssi segment initial ou final}. Therefore $w$ is not $c$-aligned, and thus not $c$-sortable.
						\item if $(i,j) \in L_c \times R_c$, then we have $u_1 = \affc{j,i}_1$ and $u_\infty = \affc{i,j}_0$ with the orientation $\omega_c(\beta_{u_1}, \beta_{u_\infty}) > 0$. Since $N(w)$ contains $\affc{i,j}_1 = u_{\infty - 1}$, then $N(w) \cap W'$ is a finite final segment of $(u_k)_k$ by Proposition \ref{prop: ensemble d'inversion ssi segment initial ou final}, and of cardinality at least two: $w$ is therefore not $c$-aligned and thus not $c$-sortable.
						\item if $(i,j) \in R_c \times L_c$, then we have $u_1 = \affc{i,j}_0$ and $u_\infty = \affc{j,i}_1$ with the orientation $\omega_c(\beta_{u_1}, \beta_{u_\infty}) > 0$. Since $N(w)$ contains $\affc{j,i}_2 = u_{\infty - 1}$, then $N(w) \cap W'$ is a finite final segment of $(u_k)_k$ by Proposition \ref{prop: ensemble d'inversion ssi segment initial ou final}, and of cardinality at least two: $w$ is therefore not $c$-aligned and thus not $c$-sortable.
					\end{enumerate}

					\item If condition $2$ of the proposition holds, consider $W'$ the generalized rank two parabolic subgroup containing $\affc{i,j}_p$ and $\affc{j,k}_q$. If $t_1$ and $t_2$ are two reflections, we write $t_1 \rightarrow t_2$ if $\omega_c(t_1, t_2) > 0$. By Corollary \ref{cor: sign of omega_c} and Definition \ref{def: biclos c-aligné}:
					\begin{itemize}
						\item If $j \in L_c$ and $i,j,k$ are such that the corresponding case is «~$+$~» in table \ref{tbl: sign of omega_c}, then $(u_k)_k$ is oriented as such:
						\begin{equation*}
							\affc{i,j}_p \rightarrow \affc{i,k}_{p+q} \rightarrow \affc{j,k}_q.
						\end{equation*} Since $N(w) \cap W' = \{ \affc{i,k}_{p+q}, \affc{j,k}_q \}$, $w$ is not $c$-aligned and thus not $c$-sortable.
						\item If $j \in L_c$ (resp. $j \in R_c$) and $i,j,k$ are such that the corresponding case if «~$-$~» (resp. «~$+$~») in Table \ref{tbl: sign of omega_c}, then we have either $\affc{i,k}_1 \in N(w)$ and $i < k$ (because $p + q \geq 1$), or $\affc{k,i}_2 \in N(w)$ and $k < i$ (because $p + q \geq 2$), and $(i,k) \in L_c^2$ (resp.$(i,k) \in R_c^2$). Therefore, first condition of the proposition also holds and we already shown that $w$ is not $c$-sortable in this case.
						\item If $j \in R_c$ and $i,j,k$ are such that the corresponding case is «~$-$~» in Table \ref{tbl: sign of omega_c}, then $(u_k)_k$ is oriented as such:
						\begin{equation*}
							\affc{j,k}_q \rightarrow \affc{i,k}_{p+q} \rightarrow \affc{i,j}_p.
						\end{equation*} Since $N(w) \cap W' = \{ \affc{i,j}_p, \affc{i,k}_{p+q} \}$, $w$ is not $c$-aligned and thus not $c$-sortable.
					\end{itemize}
				\end{enumerate}

				Hence the result.

			\end{proof}

			Each of the cases of Proposition \ref{prop: pas c triable quelles sont les inversions} will give us insight on the patterns that cannot appear in the one-line notation of an affine permutation $w \in W$. We will now recall and prove Theorem \ref{thm: critère par motif pour les c-triables de Sn tilde}.

			\motifaffine*

			\begin{proof}
				We will prove the result by double implication of the negation statement: $w$ is not $c$-sortable if and only if it contains a pattern $kij$ with $i < j < k$ and $j \in \Lcn$ or a pattern $jki$ with $i < j < k$ and $j \in \Rcn$.

				We begin by showing that if $w$ is not $c$-sortable, then it contains one such pattern. Assume that $w \in W$ is not $c$-sortable. This means one of the conditions of Proposition \ref{prop: pas c triable quelles sont les inversions} holds.

				\begin{enumerate}
					\item If there exists $a,b \in \llbracket 1,n \rrbracket$ such that $a < b$ and $\affc{a,b}_1 \in N(w)$, then we also have $\affc{a,b}_0 \in N(w)$ by Proposition \ref{prop: ensemble d'inversion ssi segment initial ou final}. This means the following pattern appears in the one-line notation of $w$: $b(b+n)a(a+n)$. In particular, there are the patterns $b(b+n)a$ of the form $jki$ and $(b+n)a(a+n)$ of the form $kij$. Similarly, if we have $\affc{b,a}_2 \in N(w)$, then we have a pattern $(a+n)(a+2n)b(b+n)$ and once again the two types of patterns: $(a+n)(a+2n)b$ of the form $jki$ and $(a+2n)b(b+n)$ of the form $kij$.

					Now, if $(a,b) \in L_c^2 \sqcup R_c^2$, we know by Proposition \ref{prop: pas c triable quelles sont les inversions} that one of $\affc{a,b}_1$ or $\affc{a,b}_2$ belongs to $N(w)$ and therefore there is a pattern of the form $kij$ with $j \in \Lcn$ or $jki$ with $j \in \Rcn$. Similarly, if $(a,b) \in L_c \times R_c$, we have a pattern $jki$ with $j \in \Rcn$ and if $(a,b) \in R_c \times L_c$, we have a pattern $kij$ with $j \in \Lcn$.

					\item If there exists $a,b,c \in \llbracket 1,n \rrbracket$ distinct and $p \geq \ddelta_{b < a}$ and $q \geq \ddelta_{c < b}$ such that $b \in L_c$ and
					\begin{equation*}
						\{ \affc{a,b}_p, \affc{a,c}_{p+q}, \affc{b,c}_q \} \cap N(w) = \{ \affc{a,c}_{p+q}, \affc{b,c}_q \},
					\end{equation*}
					then there is the following pattern in the one-line notation of $w$: $(c + (p+q)n)a(b+pn)$ which is of the form $kij$ with $j \in \Lcn$. Similarly, if $b \in R_c$ and
					\begin{equation*}
						\{ \affc{a,b}_p, \affc{a,c}_{p+q}, \affc{b,c}_q \} \cap N(w) = \{ \affc{a,b}_q, \affc{a,c}_{p+q} \},
					\end{equation*}
					then we have a pattern of the form $jki$ with $j \in \Rcn$, namely $(b+qn)(c+(p+q)n)a$.
				\end{enumerate}

				In all cases, $w$ contains one of the patterns.

				\paragraph*{} Now we show the other implication of the statement: If $w$ contains a forbidden pattern, then $w$ is not $c$-sortable. Suppose that $w$ contains at least one pattern $kij$ with $i < j < k$ and $j \in \Lcn$, we will show that one of the conditions of Proposition \ref{prop: pas c triable quelles sont les inversions} holds. Without loss of generality, we can assume that $j \in \llbracket 1,n \rrbracket$. Define $a,b,c,p,q$ to be integers such that $a = i + pn$, $b = j$, $c = k - qn$ and $a,b,c \in \llbracket 1,n \rrbracket$. In particular, we have $p \geq \ddelta_{a \geq b}$ and $q \geq \ddelta_{b \geq c}$.
				\begin{itemize}
					\item If $a \neq b$, then $a,b,c$ are distinct (we cannot have $a = c$ or $b = c$ because elements of the same class modulo $n$ appear in increasing order from left to right in the one-line notation of $w$). We have the pattern $(c+qn)(a-pn)b$, this means $\affc{a,c}_{p+q}$ and $\affc{b,c}_q$ are inversions of $w$ but not $\affc{a,b}_p$. Since $j \in \Lcn$, condition \ref{enum: 2a} of Proposition \ref{prop: pas c triable quelles sont les inversions} holds and $w$ is not $c$-sortable.
					\item If $a = b$, we have the pattern $(c+qn)(b-pn)b$ and in particular the pattern (after translation by $pn$) $(c+(p+q)n)b$. This means that the reflection $\affc{b,c}_{p+q}$ belongs to $N(w)$. Since $p \geq \ddelta_{a \geq b} = 1$, then we have $p+q \geq 1 + \ddelta_{b > c}$. This means that if $b < c$, we have $\affc{b,c}_1 \in N(w)$ and since $b \in L_c$, condition \ref{enum: 1a} or condition \ref{enum: 1b} of Proposition \ref{prop: pas c triable quelles sont les inversions} is fulfilled. And if $c < b$, then we have $\affc{b,c}_2 \in N(w)$ and since $b \in L_c$, condition \ref{enum: 1a} or condition \ref{enum: 1c} of Proposition \ref{prop: pas c triable quelles sont les inversions} is fulfilled.
				\end{itemize}
				In all cases, $w$ is not $c$-sortable.

				Exactly the same argument applied to a pattern $jki$ with $j \in \Rcn$ also shows that $w$ is not $c$-sortable.
			\end{proof}

			We end this section with some examples and remarks. First, even though the number of patterns to check in the one-line notation of $\sigma \in \widehat{\mathfrak{S}}_n$ is seemingly infinite, only a finite number of checks needs to be performed. Indeed, one can assume that $ki$ are consecutive in the patterns $kij$ and $jki$, because the existence of a pattern $kij$ (resp. $jki$) implies the existence of a pattern $k'i'j$ (resp. $jk'i'$), where $k'$ is defined as the furthest element on the right of $k$ and on the left of $i$ greater than $j$:
			\begin{equation*}
				k' = \sigma\left( \max\left( \left\{u \in \llbracket \sigma^{-1}(k), \sigma^{-1}(i) \rrbracket \;|\; \sigma(u) > j \right\} \right) \right)
			\end{equation*}
			and $i'$ is the integer directly on the right of $k'$, which is by definition smaller than $j$.
			Moreover, it suffices to check the consecutive pairs $ki$ for $k$ in the window of an affine permutation, because of the periodicity. Note that $i$ may be outside the window (if $k$ is the last element of the window), and $j$ may be very far away, but there are only a finite number of $j$ to check for each pair $ki$.

			\begin{exemple}
				Let $n = 10$ and $c = \affc{1,4,6,9}_{[1]}\affc{8,7,5,3,2,0}_{[-1]}$. Consider the following two affine permutations:
				\begin{equation*}
					\sigma_1 = [0, 1, 4, 3, 9, 6, 5, 8, 7, 12] ; \hspace*{0.1\textwidth} \sigma_2 = [-2, 0, 3, 4, 5, 9, 6, 11, 7, 12].
				\end{equation*}
				Let us check all the patterns $kij$ and $jki$ with $i < j < k$ and $k$ in the window of $\sigma_1$:
				\begin{itemize}
					\item $k = 4$, then $i = 3$ but no $j$ are possible, and similarly if $k = 8$ (then $i = 7$).
					\item $k = 9$, then $i = 6$. For $j = 7$, we have the pattern $kij$ and $j \in \Rcn$ (not forbidden) and for $j = 8$, we have the pattern $kij$ with $j \in \Rcn$ (not forbidden),
					\item $k = 12$. Then $i = 0 + 10 = 10$. For $j = 11$ (the only possible value for $j$), let us extend the window to see $i$ and $j$:
					\begin{equation*}
						{\color{gray}\cdots, -5, -2 ,-3, 2}[0, 1, 4, 3, 9, 6, 5, 8, 7, 12]{\color{gray}10, 11, 14, 13, \cdots}
					\end{equation*}
					Thus, we have a pattern $kij$ and $j \in \Lcn$ : therefore $\sigma_1$ is not $c$-sortable.
				\end{itemize}
				In this example, there are no forbidden patterns in the window of $\sigma_1$ and we needed to check beyond the window.

				For $\sigma_2$, we list all the patterns $kij$ and $jki$ with $k$ in the window of $\sigma_2$ and $ki$ consecutive in the one-line notation of $\sigma_2$. Since none of them are forbidden, it is a $c$-sortable element:
				\begin{equation*}
					\begin{array}{*9{!{\;}c}}
						1\;(-2)\;0; &
						9\;6\;7; &
						9\;6\;8; &
						11\;7\;8; &
						9\;7\;11; &
						11\;7\;10; &
						9\;12\;8; &
						12\;8\;10; &
						11\;12\;8.
					\end{array}
				\end{equation*}
			\end{exemple}

			\begin{remarque}
				Let $c \in \mathfrak{S}_n$ be a \textsc{Coxeter} element. The affine permutation $\widehat{c} = s_0 c \in \widehat{\mathfrak{S}}_n$ is a \textsc{Coxeter} element of $\widehat{\mathfrak{S}}_n$, and the natural injection $\mathfrak{S}_n \hookrightarrow \widehat{\mathfrak{S}}_n$ restricts to an injection $\mathrm{Sort}_c(\mathfrak{S}_n) \hookrightarrow \mathrm{Sort}_{\widehat{c}}(\widehat{\mathfrak{S}}_n)$. This should also be visible in the pattern avoidance criterion of Theorem \ref{thm: critère par motif pour les c-triables de Sn tilde}.

				Indeed, if $L_c \sqcup R_c = \rrbracket 1,n \llbracket$ is the partition of $\rrbracket 1,n \llbracket$ defined by $c$, then we have $L_{\widehat{c}} = L_c \sqcup \{1\}$ and $R_{\widehat{c}} = R_c \sqcup \{n\}$. Let $\sigma \in \mathfrak{S}_n$ and $\widehat{\sigma}$ the affine permutation whose window notation is obtained from the one-line notation of $\sigma$. If $a$ and $b$ are respectively the first and last elements of the window of $\widehat{\sigma}$, then $b < a+n$, therefore, if a pattern $kij$ or $jki$ with $k$ in the window of $\widehat{\sigma}$ appears in the one-line notation of $\widehat{\sigma}$, then both $k$ and $i$ are in the window. Thus, for each $j \in \rrbracket i,k \llbracket$, we have that $j$ also appears in the window of $\widehat{\sigma}$, and moreover is neither $1$ or $n$. Therefore, a pattern $kij$ with $j \in \overline{L_{\widehat{c}}}$ (reps. $jki$ with $j \in \overline{R_{\widehat{c}}}$) exists in the one-line notation of $\widehat{\sigma}$ if and only if a pattern $kij$ with $j \in L_c$ (resp. $jki$ with $j \in R_c$) exists in the one-line notation of $\sigma$. Therefore, Theorem \ref{thm: critère par motif pour les c-triables de Sn tilde} generalizes the criterion for the finite symmetric group.
			\end{remarque}

\section{\texorpdfstring{$c$}{c}-sortable biclosed sets and cyclic \texorpdfstring{$c$}{c}-noncrossing arc diagrams} \label{sec: biclos c-triable et généralisation de l'application de Reading}

	In this section, we generalize in type $\widetilde{A}$ the notion of a $c$-sortable element by leaving the \textsc{Coxeter} group $W$ and going in the set of biclosed sets of reflections of $W$. We will use a bijection between these biclosed sets and a subset of TITOs on $\ZZ$ \cite[4.12]{barkley2025extendedweakorderaffine} in order to define a $c$-sortable biclosed set similarly as done in Theorem \ref{thm: critère par motif pour les c-triables de Sn tilde}. The goal of this section is to define the necessary objects in order to define in Section \ref{sec: generalization of reading's bijection} a bijective map $nc_c^{\widetilde{A}}$ (Definition \ref{def: reading généralisée}) that generalizes \textsc{Reading}'s map $nc_c$:

	\readinggeneralise*

	\subsection{\textsc{Coxeter} sortable biclosed sets in affine type A} \label{ssec: biclos c-triables}

		In this section, we will present a definition of $c$-sortable biclosed sets of the set of reflections of the affine symmetric group, generalizing the notion of $c$-sortable elements of $\widehat{\mathfrak{S}}_n$. The definition is based on the representation of these biclosed sets as TITOs on $\ZZ$ from \cite[§3]{barkley2022combinatorial}.

		\subsubsection{Definition}
			Recall from Theorem \ref{thm: critère par motif pour les c-triables de Sn tilde} that the $c$-sortable elements of the affine symmetric group are the affine permutations such that their one-line notation avoids the patterns $kij$ with $j \in \Lcn$ and $jki$ with $j \in \Rcn$.
			As a candidate for a generalized definition, we need these elements to stay «~generalized $c$-sortable~», and we also want the other affine permutations to remain «~not generalized $c$-sortable~». Therefore, in our definition of a $c$-sortable TITO on $\ZZ$, a TITO of shape $\left[ \llbracket 0,n-1 \rrbracket \right]$ (\textit{i.e.} an affine permutation) is $c$-sortable if and only if it avoids the previously mentioned patterns.

			\biclosctriable


			Here is why we get such restriction for TITOs of other shapes. First, we would like to keep the following important property, true for any \textsc{Coxeter} group: the map $w \in \mathrm{Sort}_c(W) \mapsto \mathrm{Cov}(w) \in \mathcal{P}(R)$ is injective. But many TITOs on $\ZZ$ avoid the patterns $kij$ with $j \in \Lcn$ and $jki$ with $j \in \Rcn$, while having the same set of cover reflections (for example, with $c = \affc{1,2}_{[1]} \affc{3}_{[-1]}$: $[0,1,2]$ and $[1,2][0]$). We thus need to forbid two consecutive waxing blocks from appearing. This is also motivated by the fact we want to define the $c$-sortable TITOs as a subset of \emph{widely generated TITOs} (Paragraph \ref{par: cyclic noncrossing arc diagrams} of Section \ref{sec: background}).

			We also forbid two consecutive waning blocks from appearing: the only situation where the forbidden patterns are avoided is for a shape $\underline{[L]}\,\underline{[R]}$ where $L \subset \Lcn$ and $R \subset \Rcn$. Under the bijection with cyclic noncrossing arc diagrams, this leads to an arc diagram with two loops which we do not want in the future section.

			Finally, it is easy to see that if there are three waxing blocks in a TITO, then a pattern $kij$ with $j \in \Lcn$ or $jki$ with $j \in \Rcn$ always appears, and same goes for if a waning block appears before a block containing an element of $\Lcn$ or after a block containing an element of $\Rcn$.

			This leaves only the TITOs of shape $[L]\underline{[M]}[R]$ where $L \sqcup M \sqcup R = \llbracket 1,n \rrbracket$, and $L \subset L_c$ and $R \subset R_c$.
			In fact, we also force $M$ to contain at least an element of $L_c$ and an element of $R_c$, for reasons involving the $c$-noncrossing partitions we consider not having dangling annular blocks (See Paragraph \ref{par: geometric interpretation of noncrossing partitions} of Section \ref{sec: background}) once we establish the generalized bijection.

			\begin{exemple} \label{ex: c-sortable TITO}
				Let $c = \affc{1,4,5,6,9}_{[1]} \affc{10,8,7,3,2}_{[-1]}$. The following TITOs on $\ZZ$ are $c$-sortable:
				\begin{equation*}
					[9,6]\underline{[18,11,5,10,4]}[3,2,7],\;\; [14,11,9,15,16,3,12,8,7,10],\;\; [5,4,6,1]\underline{[17,10,9]}[3,2,8].
				\end{equation*}
				The following ones are not $c$-sortable. A forbidden pattern is marked in \textbf{bold} in each case.
				\begin{equation*}
					[19,\mathbf{16}]\underline{[\mathbf{5,11},18,10,4]}[3,2,7], \;\; [5,4,6,1]\underline{[\mathbf{7},0,-1]}[3,\mathbf{8,2}].
				\end{equation*}
				Note that it is not straightforward to check the presence or not of such a pattern, since one needs to check patterns in the whole TITO and not just the window notation. We will see in Proposition \ref{prop: TITO c triable motifs plus facile à distance finie} that it suffices to check a finite number of patterns for each affine permutation, for similar reasons than those explained at the end of Section \ref{sec: c triable du groupe symmetrique affine}.
			\end{exemple}

			By definition, a $c$-sortable TITO on $\ZZ$ is a widely generated TITO of the extended weak order, and we are able to restrict the map $\ncad$ on the set of $c$-sortable TITOs. Moreover, this also gives us the following result, which was also true for $c$-sortable elements of $\widehat{\mathfrak{S}}_n$.

			\begin{proposition} \label{prop: les couvertures d'un TITO c-triable le caractérisent}
				A $c$-sortable TITO on $\ZZ$ is uniquely characterized by its cover reflections: Let $\prec_1$ and $\prec_2$ be two $c$-sortable TITOs on $\ZZ$. Then $\prec_1 {=} \prec_2$ if and only if $\mathrm{Cov}(\prec_1) = \mathrm{Cov}(\prec_2)$.
			\end{proposition}

			\begin{proof}
				Under the map $\ncad$, a cover reflection $\affc{p,q}$ with $p < q$ is sent on the set of arcs
				\begin{equation*}
					\{ (p + kn,q + kn,L + kn,R + kn) \;|\; k \in \ZZ \}
				\end{equation*} such that in the TITO, the elements of $L$ precede $q$ and the elements of $R$ succeed $p$, with $L \sqcup R = \llbracket p+1, q-1 \rrbracket$. By Definition \ref{def: biclos c-triable}, we have $L \subset \Lcn$ and $R \subset \Rcn$. Therefore, these arcs are uniquely defined by their extremities, thus a $c$-sortable TITO is uniquely defined by its cover reflections.
			\end{proof}

		\subsubsection{Properties and structure}

			From this definition, we can obtain a more precise description of the blocks $[L]$, $\underline{[M]}$ and $[R]$ of a $c$-sortable TITO on $\ZZ$ that is not an affine permutation. These structures will be useful later for defining and proving a generalization of \textsc{Reading}'s bijection in Theorem \ref{thm: generalisation de l'application de Reading}. We start with the structure of $\underline{[M]}$.

			\begin{proposition} \label{prop: forme de M dans le cas où il y a plusieurs chaînes de descentes}
				Let $\prec$ be a $c$-sortable TITO on $\ZZ$ that is not an affine permutation. If there exist $x, y \in \underline{[M]}$ such that $x \prec y$ and $x < y$, then there exist $k \in \NN^*$, $r_1, \dots, r_k \in \NN^*$ and a family of pairwise distinct modulo $n$ integers $\left(a_j^{(i)}\right)_{\substack{1 \leq i \leq k \\ 1 \leq j \leq r_i}}$ such that:
				\begin{enumerate}[label*=(\roman*)]
					\item $\underline{[M]} = \underline{\left[a_1^{(1)}, \dots, a_{r_1}^{(1)}, a_1^{(2)}, \dots, a_{r_2}^{(2)}, \cdots, a_1^{(k)}, \dots, a_{r_k}^{(k)}\right]}$,
					\item For all $i \in \llbracket 1,k \rrbracket$, $a_1^{(i)} > a_2^{(i)} > \dots > a_{r_i}^{(i)}$, and if $i < k$, $a_{r_i}^{(i)} < a_1^{(i+1)}$.
					\item $a_1^{(1)} - a_{r_k}^{(k)} > n$,

					\item \label{enum: a1 dans Rc et ar dans Lc} For all $i \in \llbracket 1, k \rrbracket$, $a_1^{(i)} \in \Rcn$ and $a_{r_i}^{(i)} \in \Lcn$,
					\item For all $i \in \llbracket 1, k \rrbracket$ and $j \in \llbracket 1, r_i \rrbracket$ such that $a_j^{(i)} \in \Rcn$ (resp. $a_j^{(i)} \in \Lcn$), if $x \in \ZZ$ is such that $a_j^{(i)} \prec x$ (resp. $x \prec a_j^{(i)}$), then $a_j^{(i)} > x$ (resp. $x > a_j^{(i)}$).
				\end{enumerate}
			\end{proposition}

			\begin{proof}
				If such a pair $x, y \in \underline{[M]}$ exists, we can assume without loss of generality that $x \prec y$ is a cover relation. Consider the window of $\underline{[M]}$ starting at $y$ (and ending at $x - n$ because this block is waning). We have
				\begin{equation*}
					\underline{[M]} = \underline{\left[ y = b_1 , b_2 , \dots , b_{m-1}, b_m = x-n \right]}.
				\end{equation*}
				Now group the $b_i$'s by the maximal descending chains (Paragraph \ref{par: tito} of Section \ref{sec: background}) to which they belong in order to define the sequence $\left(a_j^{(i)}\right)_{i,j}$:
				\begin{itemize}
					\item $b_1 = a_1^{(1)}$,
					\item assume that $b_m = a_j^{(i)}$ is already defined, define $k = i$ and stop,
					\item assume that $b_l = a_j^{(i)}$ is already defined, define $b_{l+1} = a_{1}^{(i+1)}$ and $r_i = j$ if $b_{l+1} > b_l$, and define $b_{l+1} = a_{j+1}^{(i)}$ otherwise.
				\end{itemize}
				By definition, $M = \{ a_j^{(i)} \bmod n \;|\; i \in \llbracket 1, k\rrbracket,\; j \in \llbracket 1, r_i \rrbracket \}$. Moreover, by construction, point $(ii)$ is satisfied, and the lexicographical order on the indices $(i,j)$ of the family $\left(a_j^{(i)}\right)_{i,j}$ coincides with the order $\prec$ on this family, hence $(i)$. Finally, since $x < y$, we have $(iii)$ because
				\begin{equation*}
					a_1^{(1)} - a_{r_k}^{(k)} = b_1 - b_m = y - (x-n) = y - x + n > n.
				\end{equation*}

				Now, we prove the last two points. To simplify the following reasoning, we note $a_j^{(0)} = a_j^{(k)}$ and $a_j^{(1)} = a_j^{(k + 1)}$.
				\begin{enumerate}[label*=$(\roman*)$] \setcounter{enumi}{3}
					\item For all $i \in \llbracket 1, k\rrbracket$, the following pattern appears in $\prec$: $a_{r_i}^{(i)} \prec a_1^{(i+1)} \prec (a_{r_i}^{(i)} - n)$. It is of the form $jki$ with $i < j < k$. Since $\prec$ if $c$-sortable, $j$ cannot belong to $\Rcn$. Therefore $a_{r_i}^{(i)} \in \Lcn$. Similarly, we have the pattern $(a_1^{(i)} + n) \prec a_{r_{i-1}}^{(i-1)} \prec a_1^{(i)}$ which is of the form $kij$ with $i < j < k$, and since $\prec$ is $c$-sortable then $j$ cannot belong to $\Lcn$. Therefore $a_1^{(i)} \in \Rcn$.

					\item If $a_j^{(i)} \prec x$ and $a_j^{(i)} \in \Rcn$ (resp. $x \prec a_j^{(i)}$ and $a_j^{(i)} \in \Lcn$), there is a pattern $a_j^{(i)} \prec x \prec a_j^{(i)} - pn$ (resp. $(a_j^{(i)} + pn) \prec x \prec a_j^{(i)}$) where $p \in \NN^*$ is such that $x \prec a_j^{(i)} - pn$ and $a_j^{(i)} - pn < x$ (resp. $(a_j^{(i)} + pn) \prec x$ and $x < (a_j^{(i)} + pn)$). If $a_j^{(i)} < x$ (resp. $x < a_j^{(i)}$), it is a pattern of the form $jki$ with $j \in \Rcn$ (resp. $kij$ with $j \in \Lcn$). Since $\prec$ is $c$-sortable, this pattern cannot occur. Therefore  $x < a_j^{(i)}$ (resp. $a_j^{(i)} < x$).
				\end{enumerate}
			\end{proof}

			One useful consequence of this theorem is that every element of $\underline{[M]}$ appears in a cover reflection of $\prec$.

			\begin{lemme}
				Let $\prec$ be a $c$-sortable TITO on $\ZZ$. If $\prec$ is not a single waxing block, then for all $a \in \underline{[M]}$, there exists $t \in \mathrm{Cov}(\prec)$ and $b \in \underline{[M]}$ such that $t = \affc{a,b}$.
			\end{lemme}

			\begin{proof}
				If for all $x,y \in \underline{[M]}$, we have $x \prec y \Rightarrow x > y$, then the numbers appear in $\underline{[M]}$ in decreasing order for $<$. Thus, for all $a \in \underline{[M]}$, the reflection $t = \affc{a,b}$ where $b$ is the successor of $a$ for $\prec$ is a cover reflection of $\prec$.

				Otherwise, let $a \in \underline{[M]}$. Then $a$ belongs to a maximal descending chain $a_1 \prec a_2 \prec \cdots \prec a_k$. By Proposition \ref{prop: forme de M dans le cas où il y a plusieurs chaînes de descentes} point \ref{enum: a1 dans Rc et ar dans Lc}, we know there is at least two elements in this chain (since $a_1 \in \Rcn$ and $a_k \in \Lcn$, they are distinct). Let $i \in \llbracket 1,k \rrbracket$ such that $a = a_i$. If $i > 1$, then $\affc{a_{i-1}, a_i}$ is a cover reflection of $\prec$ since $a_{i-1} > a_i$. Otherwise, we have $a = a_1$ and $\affc{a_1, a_2}$ is also a cover reflection of $\prec$. In both cases, there exists a cover reflection $t$ of $\prec$ and $b \in \underline{[M]}$ such that $t = \affc{a,b}$.
			\end{proof}

			Proposition \ref{prop: forme de M dans le cas où il y a plusieurs chaînes de descentes} also has the following consequence that we only need to consider $kij$ and $jki$ patterns with $k$ and $i$ consecutive for the order $\prec$. This makes looking for patterns much easier to do by hand: for cover reflections (a pair $ki$), check all $j$ between $i$ and $k$: the ones in $\Lcn$ should be on the left and the ones in $\Rcn$ should be on the right of the pair $ki$. If this is not the case, then there is a pattern $kij$ or $jki$.

			\begin{proposition} \label{prop: TITO c triable motifs plus facile à distance finie}
				Let $\prec$ be a TITO on $\ZZ$ of one of the two shapes of Definition \ref{def: biclos c-triable}.
				Then the following are equivalent:
				\begin{enumerate}[label*=(\roman*)]
					\item $\prec$ is $c$-sortable
					\item $\prec$ avoids the patterns $kij$ with $i < j < k$, $\left|[k,i]_\prec\right| < +\infty$ and $j \in \Lcn$, and the patterns $jki$ with $i < j < k$, $\left|[k,i]_\prec\right| < +\infty$ and $j \in \Rcn$.
					\item $\prec$ avoids the patterns $kij$ with $i < j < k$, $k$ and $i$ consecutive and $j \in \Lcn$, and the patterns $jki$ with $i < j < k$, $k$ and $i$ consecutive and $j \in \Rcn$.
				\end{enumerate}
			\end{proposition}

			\begin{proof}
				It is clear that $(i)$ implies $(ii)$ which implies $(iii)$. Instead of proving $(iii) \Rightarrow (i)$, we will go upwards and show that $(iii) \Rightarrow (ii) \Rightarrow (i)$.

				\begin{itemize}
				\item[$(iii) \Rightarrow (ii)$] We prove the contrapositive: assume that $\prec$ contains a pattern $kij$ or $jki$ with $|[k,i]_\prec| < +\infty$. Define $k' = \max_\prec \left(\{ a \in [k,i]_\prec \;|\; a > j \}\right)$ and $i'$ the element following $k'$ in $\prec$. By definition, we have $k' > j$ and $i' < j$, and $k'$ and $i'$ are consecutive in this order. Hence a pattern $k'i'j$ or $jk'i'$ with $k'$ and $i'$ consecutive.

				\item[$(ii) \Rightarrow (i)$] We only need to show that $\prec$ avoids all patterns such that $k$ and $i$ are at infinite distance of each other, meaning they belong to two different blocks of $\prec$. This already covers the case where $\prec$ is an affine permutation (it only has a single block). In the case where $\prec$ is of shape $[L]\underline{[M]}[R]$, we will show that we cannot have a pattern $kij$ with $j \in \Lcn$ or $jki$ with $j \in \Rcn$ and $k$ and $i$ in two different blocks. Assume that we have such a pattern $kij$, we must have $i,j \in \underline{[M]}$ and $k \in [L]$ since $j \in \Lcn$ cannot be in $[R]$. But then, we have a pattern $(i+pn)ij$ for $p \in \NN$ large enough since $\underline{[M]}$ is waxing, which is of the form $kij$ with $k$ and $i$ at finite distance. This is impossible by hypothesis. Same for a pattern $jki$: we have $j,k \in \underline{[M]}$ and $i \in [R]$ because $j$ cannot belong to $[L]$. But then we have a pattern $jk(k-pn)$ with $p \in \NN$ large enough, which is of the form $jki$ with $j \in \Rcn$ and $i$ and $k$ at finite distance.
				\end{itemize}
				Hence the result.
			\end{proof}

			\begin{exemple}
				In the first example of a non $c$-sortable TITO on $\ZZ$ of Example \ref{ex: c-sortable TITO}, the pattern chosen was of the form $kij$ with $k$ and $i$ at infinite distance. However, we also have a pattern $kij$ with $k$ and $i$ consecutive: $14 \prec 5 \prec 11$ is one such pattern.

				In this example, we can now more easily check the ones that are said to be $c$-sortable are indeed $c$-sortable: in each block and for all cover reflections $\affc{k,i}$ of $\prec$ (two consecutive elements for $\prec$ in decreasing order for $<$), we verify that for all $j$ between $i$ and $k$, it is on the left if in $\Lcn$ and on the right if in $\Rcn$. For example, in the TITO $[9,6]\underline{[18,11,5,10,4]}[3,2,7]$, between $9$ and $6$ ($k$ and $i$), we have $7$ and $8$ in $\Rcn$ and both are on the right of $6$.
			\end{exemple}

	\subsection{Cyclic \texorpdfstring{$c$}{c}-noncrossing arc diagrams} \label{ssec: c-noncrossing arc diagrams}

		\subsubsection{Definition}

		Noncrossing arc diagrams are a useful tool for studying lattice quotients of the weak order of $\mathfrak{S}_n$ \cite{reading2015noncrossingarcdiagramscanonical}, and indeed the $c$-sortable elements of $\mathfrak{S}_n$ are easily characterized by their noncrossing arc diagrams: they correspond to the noncrossing arc diagrams such that each arc never goes on the left of an element of $L_c$ and on the right of an element of $R_c$.

		In the case of the affine symmetric group $\widehat{\mathfrak{S}}_n$, we introduce the notion of \emph{cyclic $c$-noncrossing arc diagrams}, which will in turn be in bijection with the $c$-sortable TITOs and biclosed sets of Definition \ref{def: biclos c-triable}. Like in the finite case, each arc must pass on the left of elements in $\Lcn$ and on the right of elements in $\Rcn$, but we will need more constraints. In order to define them, we need to introduce the notion of \emph{chains} and \emph{loops}.

		\begin{definition}[full chains and loops] \label{def: full chains and loops}
			Let $\mathcal{A}$ be a cyclic noncrossing arc diagrams. A \emph{partial chain} is a finite sequence of arcs $(\alpha_1, \alpha_2, \dots, \alpha_k) \in \mathcal{A}^k$, such that for all $i \in \llbracket 1, k-1 \rrbracket$, the final point $q_i$ of $\alpha_i$ and the initial point $p_{i+1}$ of $\alpha_{i+1}$ are equal: $q_i = p_{i+1}$. A \emph{full chain} is a partial chain $(\alpha_1, \dots, \alpha_k) \in \mathcal{A}^k$ such that for all arc $\beta \in \mathcal{A}$, both $(\beta, \alpha_1, \dots, \alpha_k)$ and $(\alpha_1, \dots, \alpha_k, \beta)$ are not partial chains.

			A \emph{loop} is a bi-infinite sequence $(\alpha_i)_{i \in \ZZ} \in \mathcal{A}^\ZZ$ of arcs such that for all $i \in \ZZ$, $q_i = p_{i+1}$. Let $S$ denote the set of all initial and final points of arcs in a loop. If $S \subset \Lcn$ or $S \subset \Rcn$, we say that the loop is \emph{imaginary}, otherwise, we say it is \emph{real}. The set of loops of $\mathcal{A}$ is denoted $\mathrm{L}(\mathcal{A})$.
		\end{definition}

		\begin{exemple}
			The cyclic noncrossing arc diagram in Example \ref{ex: C_a(A)} has three full chains up to translation by a multiple of $n$:
			\begin{equation*}
				((1,6,\{4\}, \{2,3,5\}), (6,8,\emptyset,\{7\})),\
				((7,15,\{8,11,14\}, \{9,10,12,13\}))\ \text{ and } \
				((9,12,\{11\}, \{10\}), (12, 13, \emptyset, \emptyset)).
			\end{equation*}
		\end{exemple}

		\begin{definition}[cyclic $c$-noncrossing arc diagram]
			Let $c$ be a \textsc{Coxeter} element of $\widehat{\mathfrak{S}}_n$, given by non trivial partition $L_c \sqcup R_c = \llbracket 1,n \rrbracket$. A \emph{cyclic $c$-noncrossing arc diagramm} is a cyclic noncrossing arc diagram satisfying the two following conditions:
			\begin{enumerate}
				\item for each arc $\alpha = (p,q,L,R)$, we have $L \subset \Lcn$ and $R \subset \Rcn$,
				\item every loop is a real loop.
			\end{enumerate}
			The set of cyclic $c$-noncrossing arc diagrams is denoted $c$-$\widetilde{\mathrm{NCAD}}$.
		\end{definition}

		The requirement for real loops is analogous to the absence of dandling annular blocks in our definition of $c$-noncrossing partitions of an annulus in Paragraph \ref{par: geometric interpretation of noncrossing partitions} of Section \ref{sec: background}. Moreover, a cyclic $c$-noncrossing arc diagram can have at most one real loop.

		Notice that an arc of a cyclic $c$-noncrossing arc diagram is entirely determined by its initial and final point. This way, we can describe any arc by a tuple of two integers, and we will often write them as $p \rightarrow q$ for the arc having initial point $p$ and final point $q$ if $p < q$. A full chain (resp. a loop) becomes a finite sequence (resp. an bi-infinite sequence) of strictly increasing integers, also written with arrows as separators. To prevent the apparition of many subcases later on, we define a \emph{chain} to be either a full chain or a singleton. The set of chains of a cyclic $c$-noncrossing arc diagram $\mathcal{A}$ is denoted $\mathrm{Ch}(\mathcal{A})$.

		\begin{exemple} \label{ex: chaines et loops de A}
			The noncrossing arc diagram of Example \ref{ex: C_a(A)} is a $c$-noncrossing arc diagram, and we have:
			\begin{equation*}
				\mathrm{Ch}(\mathcal{A}) = \{ 1 \rightarrow 6 \rightarrow 8, 4, 7 \rightarrow 15, 9 \rightarrow 12 \rightarrow 13, 10 \} + 10\ZZ \quad \text{ and } \quad \mathrm{L}(\mathcal{A}) = \emptyset.
			\end{equation*}
		\end{exemple}

		The crossing criterion between two arcs of Equation \ref{eq: critère de croisement d'arcs} has a simpler form in the case where every arc $\alpha = (p,q,L,R)$ satisfies $L \subset \Lcn$ and $R \subset \Rcn$, thus simplifying the identification of cyclic $c$-noncrossing arc diagrams. We recall that two arcs are allowed to have a nontrivial intersection as long as it consists only of the final point of one being the initial point of the other.

		\begin{lemme} \label{lem: critère de croisement d'arcs}
			Let $\mathcal{A}$ be a cyclic arc diagram such that each arc $(a,b,L,R)$ satisfies $L \subset \Lcn$ and $R \subset \Rcn$. We write $a \rightarrow b$ for such an arc. Let $\alpha = p \rightarrow q$ and $\beta = s \rightarrow t$ be two arcs of $\mathcal{A}$ such that $p < s$. Then $\alpha$ and $\beta$ intersect if and only if
			\begin{equation*}
				\left( p < s < q < t \text{ and } (s,q) \in (\Lcn)^2 \sqcup (\Rcn)^2 \right) \text{ or } \left( p < s < t < q \text{ and } (s,t) \in \Lcn \times \Rcn \sqcup \Rcn \times \Lcn \right).
			\end{equation*}
		\end{lemme}

		Considering intersections of (partial) chains, they behave in a similar way than arcs and we only need to compare their endpoints and wether they belong to $\Lcn$ or $\Rcn$.

		\begin{corollaire} \label{cor: croisement de chaîne}
			Let $\mathcal{A}$ be a cyclic $c$-noncrossing arc diagram and $\gamma, \gamma'$ two disjoint partial chains of $\mathcal{A}$ such that $\min(\gamma) < \min(\gamma') < \max(\gamma)$. Then one of the two following cases holds:
			\begin{itemize}
				\item either $\min(\gamma) < \min(\gamma') < \max(\gamma) < \max(\gamma')$ and $\left(\min(\gamma'), \max(\gamma)\right) \in \Lcn \times \Rcn \cup \Rcn \times \Lcn$,
				\item or $\min(\gamma) < \min(\gamma') \leq \max(\gamma') < \max(\gamma)$ and $\left(\min(\gamma'), \max(\gamma')\right) \in (\Lcn)^2 \cup (\Rcn)^2$.
			\end{itemize}
		\end{corollaire}

		In fact, in the second case, we even have $\mathrm{Supp}(\gamma') \subset \Lcn$ or $\mathrm{Supp}(\gamma') \subset \Rcn$ since all subchains of $\gamma'$ also satisfy the second case.

		\begin{proof}
			In all the other cases, we can find an arc of $\gamma$ and an arc of $\gamma'$ satisfying a case of Lemma \ref{lem: critère de croisement d'arcs}, which is absurd since $\mathcal{A}$ is a cyclic $c$-noncrossing arc diagram.
		\end{proof}

		As we notice in Example \ref{ex: chaines et loops de A}, the loops and chains form a partition of $\ZZ$.

		\subsubsection{Bijection with \texorpdfstring{$c$}{c}-noncrossing partitions} \label{sssec: bijection between NCAD and NC}

		In the finite symmetric group $\mathfrak{S}_n$, there is a bijection between the $c$-noncrossing partitions and the $c$-noncrossing arc diagrams where $c$ is a fixed \textsc{Coxeter} element since both are enumerated by the \textsc{Catalan} numbers \cite{kreweras1972, reading2015noncrossingarcdiagramscanonical}. Geometrically, this bijection consists of vertically «~squashing~» the $c$-noncrossing partition of a disk into a line, and each polygon becomes a chain of arcs. Figure \ref{fig: bijection entre diagrammes d'arcs et partitions noncroisées en type A fini} shows an example for $n = 6$ and $c = (1,2,4,5,6,3)$.

		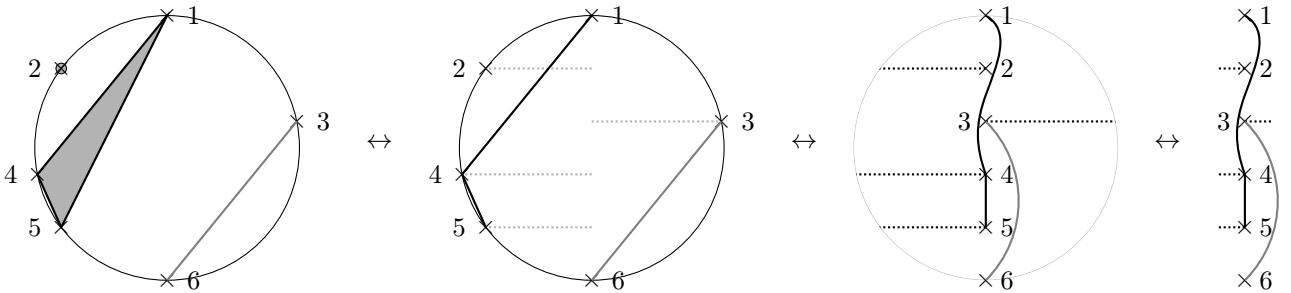
\begin{figure}[ht!]
			\centering

			\begingroup
			\def\n{6}
			\def\r{(\n-1)/2} \def\c{-(\n+1)/2}
			\def\Lc{2,4,5} \def\Rc{3}
			\def\s{0.7}

			\newcommand{\NCAD}[1]{
				\foreach \x in {1,...,\n}{
					\coordinate (\x) at (0,-\x);
					\node[draw=none,inner sep=0] at (\x) {$\times$};
				}

				\foreach \x in {1,\n}{
					\node[draw=none,inner sep=0,xshift=8] at (\x) {\x};
				}

				\ifnum #1=1 \clip (0,{\c}) circle ({\r}); \fi
				\foreach \x in \Lc{
					\draw[densely dotted, thick] (\x) -- ({-0.5 - #1*\r},-\x);
					\node[draw=none,inner sep=0,xshift=8] at (\x) {\x};
				}
				\foreach \x in \Rc{
					\draw[densely dotted, thick] (\x) -- ({0.5 + #1*\r},-\x);
					\node[draw=none,inner sep=0,xshift=-8] at (\x) {\x};
				}
				\ifnum #1=1 \draw[thin, opacity=0.3] (0,{\c}) circle ({\r}); \fi
			}

			\newcommand{\NCP}[1]{
				\foreach \x in {1,...,\n}{
					\coordinate (\x) at (0,-\x);
				}

				\foreach \x in \Lc{
					\coordinate (c\x) at (-{sqrt(\r*\r - (-\x - \c)*(-\x - \c))},-\x);
					\node[draw=none,inner sep=0] at (c\x) {$\times$};
					\node[draw=none,inner sep=0,xshift=-10] at (c\x) {\x};
					\ifnum #1=1 \draw[thick, densely dotted, opacity=0.3] (\x) -- (c\x); \fi
				}
				\foreach \x in \Rc{
					\coordinate (c\x) at ({sqrt(\r*\r - (-\x - \c)*(-\x - \c))},-\x);
					\node[draw=none,inner sep=0] at (c\x) {$\times$};
					\node[draw=none,inner sep=0,xshift=10] at (c\x) {\x};
					\ifnum #1=1 \draw[thick, densely dotted, opacity=0.3] (\x) -- (c\x); \fi
				}

				\foreach \x in {1,\n}{
					\coordinate (c\x) at (\x);
					\node[draw=none,inner sep=0] at (c\x) {$\times$};
					\node[draw=none,inner sep=0,xshift=10] at (c\x) {\x};
				}

				\draw[thin] (0,{\c}) circle ({\r});
			}
			\begin{tikzpicture}[baseline=(current bounding box.center), scale=\s]
				\NCP{0}

				\draw[thick, fill, fill opacity=0.3, rounded corners=0.01] (c1) -- (c4) -- (c5) -- cycle;
				\draw[thick, gray] (c3) -- (c6);
				\draw[fill, fill opacity=0.3] (c2) circle (0.1);
			\end{tikzpicture}\hfill$\leftrightarrow$\hfill%
			\begin{tikzpicture}[baseline=(current bounding box.center), scale=\s]
				\NCP{1}

				\draw[thick] (c1) -- (c4) -- (c5);
				\draw[thick, gray] (c3) -- (c6);
			\end{tikzpicture}\hfill$\leftrightarrow$\hfill%
			\begin{tikzpicture}[baseline=(current bounding box.center), scale=\s]
			\NCAD{1}

			\draw[thick] (1) to[out=-30,in=70] (0,-2.5) to[out=-110,in=110] (4) to (5);
			\draw[thick, gray] (3) to[out=-45,in=45] (6);
			\end{tikzpicture}\hfill$\leftrightarrow$\hfill%
			\begin{tikzpicture}[baseline=(current bounding box.center), scale=\s]
				\NCAD{0}

				\draw[thick] (1) to[out=-30,in=70] (0,-2.5) to[out=-110,in=110] (4) to (5);
				\draw[thick, gray] (3) to[out=-45,in=45] (6);
			\end{tikzpicture}%
			\endgroup
			\caption{Bijection between $c$-noncrossing partitions and $c$-noncrossing arc diagrams in type $A$}
			\label{fig: bijection entre diagrammes d'arcs et partitions noncroisées en type A fini}
		\end{figure}

		We do a similar construction for cyclic $c$-noncrossing arc diagrams and $c$-noncrossing partitions. Using the definitions from Paragraph \ref{par: geometric interpretation of noncrossing partitions} of Section \ref{sec: background}, for each polygon $P$ of a $c$-noncrossing partition of an annulus and for each $U \in \mathrm{part}(P)$, we define for each consecutive integers in $U$ (for the restricted natural order of $\ZZ$ on $U$) $p$ and $q$ the arc $\alpha = (p,q,L,R)$ where $L \subset \Lcn$ and $R \subset \Rcn$. The set of all such arcs coming from all curved polygons of a $c$-noncrossing partition of an annulus is a cyclic $c$-noncrossing arc diagram. Geometrically, Figure \ref{fig: bijection entre diagrammes d'arcs cycliques et partitions noncroisées} shows on an example for $n = 9$ and $c = \affc{1,2,4,5,7,9}_{[1]}\affc{8,6,3}_{[-1]}$ how to transform a $c$-noncrossing partition of an annulus into a cyclic $c$-noncrossing arc diagram and vice versa.

		\begin{figure}[ht!]
			\centering
			\begingroup
			\def\n{9}
			\def\Lc{1,2,5,4,7,9}
			\def\Rc{3,6,8}

			\def\r{0.8}
			\def\R{2}

			\def\s{0.8}

			\newcommand{\ang}[1]{{-(#1-1)*360/\n + 90}}
			\newcommand{\dis}[1]{{#1*\R + (1-#1)*\r}}

			\newcommand{\anneau}[1]{
				\draw[opacity={\ifnum #1=0 1 \else 0.1 \fi}] (0,0) circle (\dis{0});
				\draw[opacity={\ifnum #1=0 1 \else 0.1 \fi}] (0,0) circle (\dis{1});

				\foreach \i in \Lc{
					\coordinate (\i) at (\ang{\i}:\dis{1});
					\ifnum #1=0
					\draw[draw=none] (0,0) -- node[pos=1.2] {\i} (\i);
					\else
					\draw (\ang{\i}:{\dis{1.3}}) -- node[pos=1,fill=white, circle, inner sep=0pt] {\i} ({\ang{\i}}:{\dis{0.5}});
					\fi
				}
				\foreach \i in \Rc{
					\coordinate (\i) at (\ang{\i}:\dis{0});
					\ifnum #1=0
					\draw[draw=none] (0,0) -- node[pos=0.7] {\i} (\i);
					\else
					\draw (0,0) -- node [pos=1,fill=white, circle, inner sep=0pt] {\i} ({\ang{\i}}:{\dis{0.5}});
					\fi
				}

				\foreach \i in {1,...,\n}{
					\node[draw, circle, fill, inner sep=1pt] at (\i) {};
				}

			}
			\begin{tikzpicture}[bezier bounding box,baseline=(current bounding box.center), scale=\s]
				\anneau{0}

				\draw[relative, fill, fill opacity=0.2, rounded corners=0.05]
				(4)	to[out=100,in=80, looseness=2.8]	(1)
				to[out=80, in=40, looseness=1.5, out looseness=3]	(6)
				to[out=-80, in=-100, looseness=1.2]	(3)
				to[out=-80, in=-100, looseness=1.4]	({\ang{8}}:{\dis{0.3}})
				to[out=-77, in=-118, looseness=1.4]	(4);

				\draw[thick, relative] (2) to[out=95, in=45, out looseness=3.4, in looseness=2.8] (8);

				\draw[thick, relative, densely dotted] (7) to[out=30, in=150, looseness=1] (9);
			\end{tikzpicture}\hfill$\leftrightarrow$\hfill%
			\begin{tikzpicture}[bezier bounding box,baseline=(current bounding box.center), scale=\s]
				\anneau{0}

				\draw[relative, fill, fill opacity=0.08, rounded corners=0.05, draw opacity=0.1]
				(4)	to[out=100,in=80, looseness=2.8]	(1)
				to[out=80, in=40, looseness=1.5, out looseness=3]	(6)
				to[out=-80, in=-100, looseness=1.2]	(3)
				to[out=-80, in=-100, looseness=1.4]	({\ang{8}}:{\dis{0.3}})
				to[out=-77, in=-118, looseness=1.4]	(4);

				\draw[relative, rounded corners=0.05, densely dashed, thick]
				(4)	to[out=100,in=80, looseness=2.8]	(1)
				to[out=20,in=160, looseness=1.2]	(3)
				to[out=80, in=100, looseness=1.2]	(6);

				\draw[thick, relative] (2) to[out=95, in=45, out looseness=3.4, in looseness=2.8] (8);

				\draw[thick, relative, densely dotted] (7) to[out=30, in=150, looseness=1] (9);
			\end{tikzpicture}\hfill$\leftrightarrow$\hfill%
			\begin{tikzpicture}[bezier bounding box,baseline=(current bounding box.center), scale=\s]
				\anneau{1}

				\draw[relative, rounded corners=0.05, densely dashed, thick]
				(4)	to[out=100,in=80, looseness=2.8]	(1)
				to[out=20,in=160, looseness=1.2]	(3)
				to[out=80, in=100, looseness=1.2]	(6);

				\draw[thick, relative] (2) to[out=95, in=45, out looseness=3.4, in looseness=2.8] (8);

				\draw[thick, relative, densely dotted] (7) to[out=30, in=150, looseness=1] (9);
			\end{tikzpicture}\hfill$\leftrightarrow$\hfill%
			\begin{tikzpicture}[baseline=(current bounding box.center), scale=\s]
				\foreach \i in \Lc{
					\coordinate (\i) at (\ang{\i}:\dis{0.5});
				}
				\foreach \i in \Rc{
					\coordinate (\i) at (\ang{\i}:\dis{0.5});
				}

				\draw[densely dashed, thick] plot [ smooth, tension=0.8 ] coordinates {(4) (\ang{5}:\dis{0.3}) (\ang{6}:\dis{0.9}) (\ang{7}:\dis{0.3}) (\ang{8}:\dis{0.7}) (\ang{9}:\dis{0.3}) (1) (\ang{2}:\dis{0.3}) (3) (\ang{4}:\dis{0}) (\ang{5}:\dis{0}) (6)};

				\draw[thick] plot [ smooth, tension=0.8 ] coordinates {(2) (\ang{3}:\dis{0.7}) (\ang{4}:\dis{0.15}) (\ang{5}:\dis{0.15}) (\ang{6}:\dis{0.7}) (\ang{7}:\dis{0.2}) (8)};

				\draw[thick, densely dotted] plot [ smooth, tension=0.8 ] coordinates {(7) (\ang{8}:\dis{0.9}) (9)};

				\foreach \i in \Lc{
					\draw (\ang{\i}:{\dis{1.3}}) -- node[pos=1,fill=white, circle, inner sep=0pt] {\i} ({\ang{\i}}:{\dis{0.5}});
				}
				\foreach \i in \Rc{
					\draw (0,0) -- node [pos=1,fill=white, circle, inner sep=0pt] {\i} ({\ang{\i}}:{\dis{0.5}});
				}
			\end{tikzpicture}
			\endgroup
			\caption{Bijection between $c$-noncrossing partitions of an annulus and cyclic $c$-noncrossing arc diagrams}
			\label{fig: bijection entre diagrammes d'arcs cycliques et partitions noncroisées}
		\end{figure}
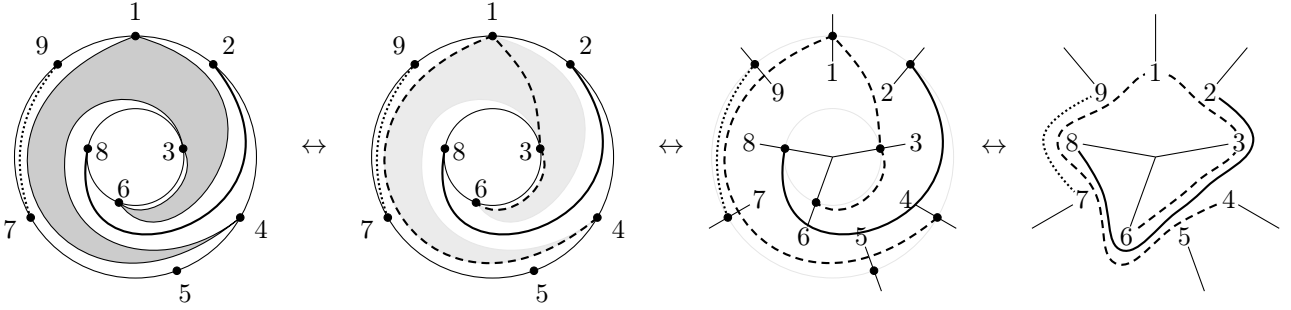

		The noncrossing property of the arc diagram obtained from this procedure comes from the fact that two curved polygons cross if and only if there exists a diagonal of each one such that these diagonals cross, where a diagonal is any segment between two vertices of a polygon.

		\begin{proposition} \label{prop: bijection partitions noncroisées et diagrammes d'arcs cycliques}
			The previous construction defines a bijection between $c$-noncrossing partitions of an annulus and cyclic $c$-noncrossing arc diagrams.
		\end{proposition}

		The inverse map of this bijection will be denoted $\mathrm{NC}_c$, such that if $\mathcal{A}$ is a cyclic $c$-noncrossing arc diagram, $\mathrm{NC}_c(\mathcal{A})$ is a $c$-noncrossing partition. This inverse map is also explicit: associate to each chain $\gamma$ the cycle $\affc{a_1, \dots, a_k, b_l, \dots, b_1}$ where $\{ a_1, \dots, a_k \} = \mathrm{Supp}(\gamma) \cap \Lcn$, $\{ b_1, \dots, b_l \} = \mathrm{Supp}(\gamma) \cap \Rcn$ and $a_1 < \dots < a_k$ and $b_1 < \dots < b_l$. Similarly, associate to each loop $\gamma$ the product of $\affc{a_1, \dots, a_k}_{[1]} \affc{b_l, \dots, b_1}_{[-1]}$ where $\{ a_1, \dots, a_k \} = \mathrm{Supp}(\gamma) \cap L_c$, $\{ b_1, \dots, b_l \} = \mathrm{Supp}(\gamma) \cap R_c$ and $a_1 < \dots < a_k$ and $b_1 < \dots < b_l$. Taking the product of these pseudo-cycles defined from all the loops and chains of a cyclic $c$-noncrossing arc diagram $\mathcal{A}$ gives a $c$-noncrossing partition $\mathrm{NC}_c(\mathcal{A})$ such that $\mathcal{A}$ is the cyclic $c$-noncrossing arc diagram obtained by the map of Proposition \ref{prop: bijection partitions noncroisées et diagrammes d'arcs cycliques}. The permutation obtained is indeed $c$-noncrossing due to the fact that the crossing property between two arcs of a cyclic $c$-noncrossing arc diagram (Lemma \ref{lem: critère de croisement d'arcs}) and the cyclic factors of a $c$-noncrossing partition (Paragraph \ref{par: partitions non croisées Sn tilde} of Section \ref{sec: background}) is the same.

		\subsubsection{Structure of cyclic \texorpdfstring{$c$}{c}-noncrossing arc diagrams} \label{sssec: structure des c-NCAD + suite C}

		Let $\mathcal{A}$ be a cyclic $c$-noncrossing arc diagram, and fix $a \in \ZZ$. There may be arcs passing on the left or the right of $a$ (but not both at the same time). We call all these arcs \emph{covering arcs of $a$}. It is also possible that many copies modulo $n$ of the same arc covers $a$. We would like to find, if $a$ has covering arcs, the closest one. Geometrically, draw a horizontal line passing through $a$ (or the radial half line in a cyclic representation) and mark all intersections points of arcs and this line. The point that minimizes the distance to $a$ belongs to the closest arc. We call this arc a \emph{neighbor arc of $a$}.

		\begin{exemple}
			In the $c$-noncrossing arc diagram of Example \ref{ex: C_a(A)}, the neighbor arc of $8$ is $7 \rightarrow 15$. It looks like it is also the neighbor arc of $2$, but it is in fact its previous copy modulo $n$, $-3 \rightarrow 5$. The point $4$ has two covering arcs, $1 \rightarrow 8$ and $-3 \rightarrow 5$, but the former is closer to it, thus it is its neighbor arc.
		\end{exemple}

		We now give an equivalent, more abstract definition of covering arcs and neighbor arcs, which will be more useful in proofs.

		\begin{definition}[Covering arcs, neighbor arc] \label{def: arc couvrant, arc voisin}
			Let $\mathcal{A}$ be a cyclic $c$-noncrossing arc diagram, and $a \in \ZZ$.
			The set of \emph{covering arcs of $a$} is
			\begin{equation*}
				\Gamma_a = \{ p \rightarrow q \in \mathcal{A} \;|\; p < a < q \}.
			\end{equation*}

			If $\Gamma_a = \emptyset$, we say that $a$ \emph{does not have any neighbor arc}. Otherwise, define for all $X \subset \ZZ$ the set of all covering arcs with their initial point in $X$:
			\begin{equation*}
				S_a^X = \{ p \rightarrow q \in \Gamma_a \;|\; p \in X \}.
			\end{equation*}
			In particular, $S_a^\ZZ = \Gamma_a$. We order the arcs $p \rightarrow q$ of $S_a^X$ by considering the order induced on $p$ by the natural order of $\ZZ$.

			Let $X \in \{ \Lcn, \Rcn \}$ such that $a \in X$. The \emph{neighbor arc of $a$} is the arc of $\Gamma_a$ defined by
			\begin{equation*}
				\alpha_a = \left\lbrace \begin{array}{ll}
					\max(S_a^X)		& \text{ if } S_a^X \neq \emptyset \\
					\min(S_a^\ZZ)	& \text{ if } S_a^X = \emptyset.
				\end{array} \right.
			\end{equation*}
		\end{definition}

		This definition indeed coincides with the geometric notion previously described:

		\begin{lemme} \label{lem: arc voisin = le plus proche}
			Let $\mathcal{A}$ be a cyclic $c$-noncrossing arc diagram and $a \in \ZZ$. If $a$ has a neighbor arc $\alpha_a$, then it is the covering arc of $a$ that is geometrically the closest to $a$.
		\end{lemme}

		\begin{remarque}
			Consider the symmetric diagram $\mathcal{A}' = \{ -q \rightarrow -p \;|\; p \rightarrow q \in \mathcal{A} \}$. We have for any $a \in \ZZ$, $\alpha_a = u \rightarrow v$ in $\mathcal{A}$ if and only if $\alpha_{-a} = -v \rightarrow -u$ in $\mathcal{A}'$. Therefore, one can also define the neighbor arc of $a \in \ZZ$ as
			\begin{equation*}
				\alpha_a = \left\lbrace \begin{array}{ll}
					\min({S'}_a^X)		& \text{ if } {S'}_a^X \neq \emptyset \\
					\max({S'}_a^\ZZ)	& \text{ if } {S'}_a^X = \emptyset
				\end{array} \right.
			\end{equation*}
			where $X \in \{\Lcn, \Rcn\}$ is such that $a \in X$ and ${S'}_a^X = \{ p \rightarrow q \in \Gamma_a \;|\; q \in X \}$ is ordered by the natural order on $q$.
		\end{remarque}

		We now introduce a sequence that encodes useful information about a cyclic $c$-noncrossing arc diagram.

		\begin{definition} \label{def: suite C_a(x)}
			Let $\mathcal{A}$ be a cyclic $c$-noncrossing arc diagram. For each $a \in \ZZ$, we construct a finite or infinite sequence $\mathcal{C}_a(\mathcal{A}) $ that takes values in $\ZZ \cup \{\rightarrow, \uparrow, \downarrow\}$ as follows: start with the empty sequence. Let $b = a$ and consider, $\gamma$ the chain or loop containing $b$, which will be denoted by
			\begin{equation*}
				\gamma = \cdots \rightarrow b_{-1} \rightarrow b_0 = b \rightarrow b_1 \rightarrow \cdots .
			\end{equation*}
			Add to $\mathcal{C}_a(\mathcal{A})$ all the terms $b_0, \rightarrow, b_1, \rightarrow, \cdots$. If an infinite number of terms were added, or if $\max(\gamma)$ does not have a neighbor arc, we stop here. Otherwise, we add to $\mathcal{C}_a(\mathcal{A})$ the arrow $\downarrow$ if $\max(\gamma) \in \Lcn$ and $\uparrow$ otherwise. We modify the value of $b$ to be the final point of $\alpha_{\max(\gamma)}$ and we iterate the same procedure.

			From the sequence $\mathcal{C}_a(\mathcal{A})$, we also build the sequence $\mathcal{C}_a^n(\mathcal{A})$ obtained by replacing all integers by their residue modulo $n$.
		\end{definition}

		\begin{exemple} \label{ex: C_a(A)} Let $n = 10$ and take $c = \affc{1,4,6,8}_{[1]} \affc{10,9,7,5,3,2}_{[-1]}$.

			\noindent\begin{minipage}{0.7\textwidth}
			Consider this cyclic $c$-noncrossing arc diagram $\mathcal{A}$. We have the following sequences:
			\begin{align*}
				\mathcal{C}_2(\mathcal{A}) & = {\color{gray!70!black}2 \rightarrow 3} \uparrow \mathbf{5} \uparrow {6 \rightarrow 8} \downarrow \mathbf{15} \uparrow {16 \rightarrow 18} \downarrow \mathbf{25} \uparrow {26} \cdots \\
				\mathcal{C}_2^n(\mathcal{A}) & = {\color{gray!70!black}2 \rightarrow 3} \uparrow \mathbf{5} \uparrow {6 \rightarrow 8} \downarrow \mathbf{5} \uparrow {6 \rightarrow 8} \downarrow \mathbf{5} \uparrow {6} \cdots \\
				\mathcal{C}_7(\mathcal{A}) & = \mathbf{7 \rightarrow 15} \uparrow {16 \rightarrow 18} \downarrow \mathbf{25} \uparrow {26 \rightarrow 28} \downarrow \mathbf{35} \uparrow {36} \cdots \\
				\mathcal{C}_7^n(\mathcal{A}) & = \mathbf{7 \rightarrow 5} \uparrow {6 \rightarrow 8} \downarrow \mathbf{5} \uparrow {6 \rightarrow 8} \downarrow \mathbf{5} \uparrow {6} \cdots
			\end{align*}
			Notice that the sequences $\mathcal{C}_a^n(\mathcal{A})$ are all infinite and ultimately periodic for $a \in \{2,7\}$. Moreover, their periods are the same. This is in fact true for all $a \in \ZZ$ and the rest of this section is mainly dedicated to proving this fact.
			\end{minipage} \hfill
			\begin{tikzpicture}[baseline, scale=1.3]
				\def\n{10}

				\def\Rc{2, 3, 5, 7, 9, 10}
				\def\Lc{1, 4, 6, 8}

				\newcommand{\ang}[1]{{-(#1-1)*360/\n + 90}}

				\foreach \i in {1,...,\n}{
					\coordinate (\i) at (\ang{\i}:1);
				}

				\foreach \i in \Rc{
					\draw (0,0) -- (\i);
				}
				\foreach \i in \Lc{
					\draw (\ang{\i}:1.7) -- (\i);
				}

				\draw[gray!70!black] plot [ smooth, tension=0.8 ] coordinates {(9) (\ang{10}:1.3) (\ang{1}:0.7) (2) (3)};
				\draw[densely dashed] plot [ smooth, tension=0.8 ] coordinates {(1) (\ang{2}:1.3) (\ang{3}:1.4) (\ang{4}:0.7) (\ang{5}:1.3) (6) (\ang{7}:1.3) (8)};
				\draw[thick] plot [ smooth, tension=0.8 ] coordinates {(7) (\ang{8}:0.7) (\ang{9}:1.3) (\ang{10}:1.4) (\ang{1}:0.8) (\ang{2}:1.2) (\ang{3}:1.2) (\ang{4}:0.5) (5)};

				\foreach \i in {1,...,\n}{
					\node[circle, fill=white, inner sep=0pt] at (\i) {\i};
				}
			\end{tikzpicture}
		\end{exemple}

		In the rest of this paper, we may write $\updownarrow$ for either $\uparrow$ or $\downarrow$. For example, there is a factor $5 \updownarrow 6 \rightarrow 8 \updownarrow 15$ in $\mathcal{C}_2(\mathcal{A})$ of Example \ref{ex: C_a(A)}.

		We now state preliminary results about the sequence $\mathcal{C}_a(\mathcal{A})$. By construction, if $a \in \ZZ$, then for all $p \in \ZZ$, $\mathcal{C}_{a+pn}(\mathcal{A}) = \mathcal{C}_a(\mathcal{A}) + pn$. This means that the finite or infinite nature of the sequence $\mathcal{C}_a(\mathcal{A})$ depends only on the class modulo $n$ of $a$. In fact, it does not depend on $a$ at all:

		\begin{proposition} \label{prop: C_a(A) est fini ou infini}
			Let $\mathcal{A}$ be a cyclic $c$-noncrossing arc diagram. We have the following alternative:
			\begin{itemize}
				\item for all $a \in \ZZ$, $\mathcal{C}_a(\mathcal{A})$ is finite,
				\item for all $a \in \ZZ$, $\mathcal{C}_a(\mathcal{A})$ is infinite.
			\end{itemize}
			Moreover, in the second case, the sequence $\mathcal{C}_a^n(\mathcal{A})$ is ultimately periodic.
		\end{proposition}

		\begin{proof}
			Let $a_0 \in \ZZ$ such that $\mathcal{C}_{a_0}(\mathcal{A})$ is infinite. We show that $\mathcal{C}_a(A)$ is also infinite for all $a \in \ZZ$. Since $\mathcal{C}_{a_0}(\mathcal{A})$ is infinite, the maximal element of every chain of $\mathcal{A}$ has a neighbor arc. If the construction of $\mathcal{C}_a(\mathcal{A})$ stops, then a loop has been added, and therefore $\mathcal{C}_a(\mathcal{A})$ is infinite. Otherwise, the construction never ends, and $\mathcal{C}_a(\mathcal{A})$ is also infinite.

			Now we show that in the case where $\mathcal{C}_{a}(\mathcal{A})$ is infinite, the sequence $\mathcal{C}_a^n(\mathcal{A})$ is ultimately periodic. Since it is infinite, there exists an integer $c \in \llbracket 0,n-1 \rrbracket$ such that $c$ appears more than once in $\mathcal{C}_a^n(\mathcal{A})$. This means there exists $p,q \in \ZZ$ with $p < q$ such that $c + pn$ and $c + qn$ appear in $\mathcal{C}_a(\mathcal{A})$. By construction, the subsequences of $\mathcal{C}_a(\mathcal{A})$ obtained by considering only the terms starting from $c + pn$ and $c + qn$ are equal up to a shift of $(q-p)n$. This means that modulo $n$, they are equal, and the sequence $\mathcal{C}_a^n(\mathcal{A})$ is periodic after the first occurrence of $c$.
		\end{proof}

		When $\mathcal{C}_a^n(\mathcal{A})$ is infinite (thus ultimately periodic), we call a \emph{periodic factor} of $\mathcal{C}_a^n(\mathcal{A})$ any subword of minimal length such that from a certain point, the sequence $\mathcal{C}_a^n(\mathcal{A})$ consists of repetition of this subword. Two periodic factors differ only by cyclic permutation: in Example \ref{ex: C_a(A)}, a periodic factor of $\mathcal{C}_2^n(\mathcal{A})$ is $5 \uparrow 6 \rightarrow 8 \downarrow$, and another one is $\rightarrow 8 \downarrow 5 \uparrow 6$. The set of periodic factors will be called the \emph{period} of $\mathcal{C}_a^n(\mathcal{A})$.

		In fact, the period of $\mathcal{C}_a^n(\mathcal{A})$, when this sequence is infinite, is independent of the choice of $a \in \ZZ$. In order to prove this fact, we need more knowledge about the structure of this sequence.

		\begin{lemme} \label{lem: si facteur c|d dans tau_a alors enchevetrement des chaînes}
			Let $\mathcal{A}$ be a cyclic $c$-noncrossing arc diagram such that $\mathcal{C}_a(\mathcal{A})$ is infinite for all $a \in \ZZ$. Let $\tau^n_a$ be a periodic factor of $\mathcal{C}_a^n(\mathcal{A})$ for $a \in \ZZ$, and $\tau_a$ a factor of $\mathcal{C}_a(\mathcal{A})$ such that $\tau_a^n = \tau_a \bmod n$. By construction, if there is a factor $c \updownarrow d$ in $\tau_a$, then $c$ and $d$ appear in chains of $\mathcal{A}$ (and not loops). In this case, if $\gamma$ is the chain of $\mathcal{A}$ containing $c$ and $\eta$ the chain of $\mathcal{A}$ containing $d$, we have
			\begin{equation*}
				\min(\gamma) < \min(\eta) < \max(\gamma) = c < \max(\eta).
			\end{equation*}
		\end{lemme}

		\begin{remarque}
			Note that $\gamma$ and $\eta$ can be equal modulo $n$.
		\end{remarque}

		\begin{proof}
			By definition, $c = \max(\gamma)$. Since $c$ has a neighbor arc in the chain $\eta$, we also have $\min(\eta) < c < d \leq \max(\eta)$. Moreover, by construction of $\mathcal{C}_a(\mathcal{A})$, we have either $\min(\gamma) < \min(\eta) < \max(\gamma) < \max(\eta)$ or $\min(\eta) < \min(\gamma) \leq \max(\gamma) < \max(\eta)$. We show by contradiction that the second case cannot happen, using an argument of infinite descent.

			We thus assume that we are in the second case: by construction of the sequence $\mathcal{C}_a(\mathcal{A})$, there exists an arc in $\gamma$ that is the neighbor arc of some integer $z \in \ZZ$ that belongs to $\tau_a$. Thus, $\min(\gamma) < z < \max(\gamma)$ hence $\min(\eta) < z < \max(\eta)$.
			By Corollary \ref{cor: croisement de chaîne}, we have either $\mathrm{Supp}(\gamma) \subset \Lcn$ or $\mathrm{Supp}(\gamma) \subset \Rcn$. If $\mathrm{Supp}(\gamma) \subset \Lcn$ (resp. $\mathrm{Supp}(\gamma) \subset \Rcn$),
			since $\gamma$ is geometrically closer to $z$ than $\eta$ by Lemma \ref{lem: arc voisin = le plus proche}, then $z \in \Lcn$ (resp. $z \in \Rcn$). By Corollary \ref{cor: croisement de chaîne}, this means that the chain containing $z$ is also contained between $\min(\gamma)$ and $\max(\gamma)$. Repeating this argument gives an infinite number of chains appearing between $\min(\gamma)$ and $\max(\gamma)$, which is absurd. Therefore, only the case $\min(\gamma) < \min(\eta) < \max(\gamma) < \max(\eta)$ happens.
		\end{proof}

		A useful consequence of this lemma is the fact that the minimal element of the union of the supports of the chains that appear in a periodic factor of the sequence $\mathcal{C}_a(\mathcal{A})$ belongs to the first chain that appears in this factor, even though it may not appear in the period itself (in fact, it never does).

		\begin{corollaire} \label{cor: m appartient à la première chaine}
			Let $\mathcal{A}$ be a cyclic $c$-noncrossing arc diagram and $a \in \ZZ$ such that $\mathcal{C}_a(\mathcal{A})$ is infinite and $\tau^n_a$, a periodic factor of $\mathcal{C}_a^n(\mathcal{A})$, contains an arrow $\uparrow$ or $\downarrow$. Let $\tau_a$ be a factor of $\mathcal{C}_a(\mathcal{A})$ such that $\tau_a^n = \tau_a \bmod n$. Let $\Gamma_{\tau_a} = \{ \gamma \in \mathrm{Ch}(\mathcal{A}) \;|\; \mathrm{Supp}(\gamma) \cap \tau_a \neq \emptyset \}$ be the set of all chains that have an element in $\tau_a$, and denote $\gamma_0$ the chain that contains the first integer element of $\tau_a$. Then $\min\left( \bigcup_{\gamma \in \Gamma_{\tau_a}} \mathrm{Supp}(\gamma) \right)$ belongs to $\gamma_0$.
		\end{corollaire}

		\begin{lemme} \label{lem: deux arcs imbriqués avec la fin dans le même Lc Rc alors ils sont dans la période}
			Let $\mathcal{A}$ be a cyclic $c$-noncrossing arc diagram such that $\mathcal{C}_a(\mathcal{A})$ is infinite. Let $q \rightarrow q'$ be an arc of $\mathcal{A}$ belonging to a chain $\gamma'$, and $\gamma$ be a chain of $\mathcal{A}$ such that $\gamma \bmod n$ intersects the period of $\mathcal{C}_a^n(\mathcal{A})$. If there exists $p \in \gamma$ such that $q < p < q'$ and $(p,q') \in (\Lcn)^2 \sqcup (\Rcn)^2$, then $\gamma' \bmod n$ also intersects the period of $\mathcal{C}_a^n(\mathcal{A})$.
		\end{lemme}

		\begin{proof}
			Suppose that $p$ and $q'$ both belong in $\Lcn$ (resp. $\Rcn$). In particular, by Corollary \ref{cor: croisement de chaîne}, we have that $p' = \max(\gamma) < q'$ and $p' \in \Lcn$ (resp. $p' \in \Rcn$). Thus, we have ${S'}_{p'}^{\Lcn} \neq \emptyset$ (resp. ${S'}_{p'}^{\Rcn} \neq \emptyset$) since it contains the arc $q \rightarrow q'$. Therefore, the neighbor arc of $p'$, $\alpha_{p'} = u_0 \rightarrow u_0'$, has its final point $u_0'$ in $\Lcn$ (resp. $\Rcn$) and $u_0' \leq q'$. If $u_0' = q'$, then by construction the chain containing $q'$ intersects the period of $\mathcal{C}_a^n(\mathcal{A})$. Otherwise, we have $q < u_0' < q'$. We apply the same argument again with $u_0'$ and $q'$ to obtain an arc $u_1 \rightarrow u_1'$, then $u_2 \rightarrow u_2'$, etc. Each time, either $q' = u_i'$ for some $i \in \NN$ and thus the chain containing $q'$ intersects the period of $\mathcal{C}_a^n(\mathcal{A})$, or we consider the neighbor arc of $u_i'$. Since the sequence $u_i'$ is strictly increasing and takes values between $p'$ and $q'$, it takes a finite number of iterations to conclude.
		\end{proof}

		The next result states that if $\mathcal{C}_a(\mathcal{A})$ is infinite, one can remove all chains that appear in the period of $\mathcal{C}_a^n(\mathcal{A})$ an replace them with a single loop.

		\begin{proposition} \label{prop: on peut remplacer une boucle de chaine par une seule boucle}
			Let $\mathcal{A}$ be a cyclic $c$-noncrossing arc diagram and $a \in \ZZ$ such that $\mathcal{C}_a(\mathcal{A})$ is infinite. Note $\tau_a^n$ a periodic factor of $\mathcal{C}_a^n(\mathcal{A})$. Let
			\begin{equation*}
				\Gamma_{\tau_a^n} = \{ \gamma \in \mathrm{Ch}(\mathcal{A}) \cup \mathrm{L}(\mathcal{A}) \;|\; \exists p \in \mathrm{Supp}(\gamma),\, p \bmod n \in \tau_a^n \}.
			\end{equation*}

			Split all the numbers of the chains of $\Gamma_{\tau_a^n}$ in two sets $L$ and $R$ such that:
			\begin{equation*}
				L = \Lcn \cap \bigcup_{\gamma \in \Gamma_{\tau_a^n}} \mathrm{Supp}(\gamma) \quad\text{ and }\quad R = \Rcn \cap \bigcup_{\gamma \in \Gamma_{\tau_a^n}} \mathrm{Supp}(\gamma).
			\end{equation*}

			We have the two following results:
			\begin{enumerate}
				\item The sets $L$ and $R$ are both nonempty and thus infinite and translationally invariant by multiples of $n$.
				\item If $\widetilde{\gamma}$ is the loop of support $L \cup R$, then the set $\mathcal{B} = (\mathcal{A} \setminus \bigcup_{\gamma \in \Gamma_{\tau_a^n}} \gamma) \cup \widetilde{\gamma}$ is a cyclic $c$-noncrossing arc diagram.
			\end{enumerate}
		\end{proposition}

		\begin{proof}
			If $|\Gamma_{\tau_a^n}| = 1$, then its only element is a (real) loop. By definition, it contains in its support an infinite amount of elements in $\Lcn$ and $\Rcn$ (there is at least one in each set, as well as all the translations by a multiple of $n$). This proves the first point. The second point is trivial in this case since $\widetilde{\gamma}$ is the only element of $\Gamma_{\tau_a^n}$.

			Now suppose that $|\Gamma_{\tau_a^n}| > 1$. We show the two points separately.
			\begin{enumerate}
				\item Let $\tau_a$ be a factor of $\mathcal{C}_a(\mathcal{A})$ such that $\tau_a \bmod n = \tau_a^n$. Since $\Gamma_{\tau_a^n}$ contains at least two cycles, there exists a factor $x \updownarrow y$, with $x,y \in \ZZ$, in $\tau_a$. Denote $z$ the minimal element of the chain containing $y$. Using Lemma \ref{lem: si facteur c|d dans tau_a alors enchevetrement des chaînes} and Corollary \ref{cor: croisement de chaîne}, we have that $(x,z) \in \Lcn \times \Rcn \sqcup \Rcn \times \Lcn$. This means the sets $L$ and $R$ are nonempty. Since they are stable by translation by a multiple of $n$, they are infinite.

				\item We know that $(\mathcal{A} \setminus \bigcup_{\gamma \in \Gamma_{\tau_a^n}} \gamma) \cup \widetilde{\gamma}$ is a cyclic arc diagram since the loop added has support exactly the union of the support of the chains and loops removed. We need to show that it is noncrossing. Two arcs in $(\mathcal{A} \setminus \bigcup_{\gamma \in \Gamma_{\tau_a^n}} \gamma)$ are naturally noncrossing since $\mathcal{A}$ is a cyclic $c$-noncrossing arc diagram. We only need to check that no arc of $\mathcal{A}$ crosses an arc of $\widetilde{\gamma}$. Let $p \rightarrow q$ be an arc of $\widetilde{\gamma}$ and $x \rightarrow y$ be an arc of $(\mathcal{A} \setminus \bigcup_{\gamma \in \Gamma_{\tau_a^n}} \gamma)$, where $p < q$ and $x < y$. Note $\gamma_1$ the chain or loop of $\Gamma_{\tau_a^n}$ that contains $p$, and $\gamma_2$ the chain or loop of $\Gamma_{\tau_a^n}$ that contains $q$. If $\gamma_1 = \gamma_2$, then $p \rightarrow q$ is an arc of $\gamma_1$, and therefore an arc of $\mathcal{A}$. Hence, it does not cross $x \rightarrow y$. We now assume that $\gamma_1$ and $\gamma_2$ are distinct.

				Let us consider all six possible orders of the integers $p, q, x, y$ and we show using Lemma \ref{lem: critère de croisement d'arcs} that these arcs do not cross.

				\begin{itemize}
					\item If $p < q < x < y$ or $x < y < p < q$. In this case, the arcs $p \rightarrow q$ and $x \rightarrow y$ do not cross.

					\item If $p < x < y < q$, we need to show that $(x,y) \in (\Lcn)^2 \sqcup (\Rcn)^2$. If $p = \max(\gamma_1)$, consider $\alpha_p = u \rightarrow v$ the neighbor arc of $p$. By construction, $u < p < v$ and $u,v \in L \cup R$, since $p$ and $q$ are consecutive in $L \cup R$, then $u < p < q \leq v$. Therefore, $u < x < y < v$. Since the arc $u \rightarrow v$ and $x \rightarrow y$ belong to $\mathcal{A}$, they do not cross, hence $(x,y) \in (\Lcn)^2 \sqcup (\Rcn)^2$. If $p < \max(\gamma_1)$, then we have $p < x < y < \max(\gamma_1)$ since $q < \max(\gamma_1)$ by the same logic as previously for $v$. Once again, this proves that $(x,y) \in (\Lcn)^2 \sqcup (\Rcn)^2$.

					\item If $x < p < y < q$, we need to show that $(p,y) \in (\Lcn \times \Rcn) \sqcup (\Rcn \times \Lcn)$. By Lemma \ref{lem: deux arcs imbriqués avec la fin dans le même Lc Rc alors ils sont dans la période}, we have $(p,y) \in (\Lcn \times \Rcn) \sqcup (\Rcn \times \Lcn)$ since the chain containing $x \rightarrow y$ does not belong to $\Gamma_{\tau_a^n}$.

						\item If $p < x < q < y$, we need to show that $(x,q) \in (\Lcn \times \Rcn) \sqcup (\Rcn \times \Lcn)$. Suppose that $q \in \Lcn$, the case where $q \in \Rcn$ is analogous, and we show that $x \in \Rcn$. By Lemma \ref{lem: deux arcs imbriqués avec la fin dans le même Lc Rc alors ils sont dans la période}, $y \in \Rcn$ since $q \in \Lcn$ and the chain containing $x \rightarrow y$ does not belong to $\Gamma_{\tau_a^n}$.

					\begin{minipage}{0.72\linewidth}
						If $p \neq \max(\gamma_1)$, let $u = p$ and $v \in \ZZ$ such that $u \rightarrow v \in \mathcal{A}$. Otherwise, let $u \rightarrow v = \alpha_p$, the neighbor arc of $p$. In both cases, we have $u \leq p < q \leq v$, and $u \rightarrow v$ is an arc of some chain of  $\Gamma_{\tau_a^n}$.

						If $u < x < y < v$, then $(x,y) \in (\Lcn)^2 \sqcup (\Rcn)^2$ therefore $x \in \Rcn$. Otherwise, we have $u < x < v < y$. By Lemma \ref{lem: deux arcs imbriqués avec la fin dans le même Lc Rc alors ils sont dans la période} applied to $x < v < y$, we have that $v \in \Lcn$ because $x \rightarrow y$ does not belong to a chain of $\Gamma_{\tau_a^n}$. Therefore, by Lemma \ref{lem: critère de croisement d'arcs}, $x \in \Rcn$. The following diagram gives a visual representation of this last case.
					\end{minipage}\hfill%
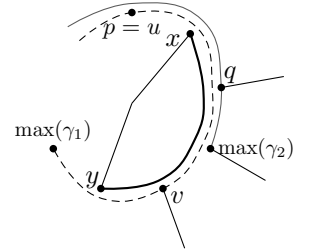
					\begin{tikzpicture}[baseline={(current bounding box.center)}, scale=1.2]
						\def\n{9}

						\newcommand{\ang}[1]{{-(#1-1)*360/\n + 90}}

						\foreach \i in {0,...,7} {
							\coordinate (\i) at (\ang{\i}:1);
						}

						\foreach \i in {2,6}{
							\draw (0,0) -- (\i);
						}
						\foreach \i in {3,4,5}{
							\draw (\ang{\i}:1.7) -- (\i);
						}

						\draw[gray!70!black] plot [ smooth, tension=0.8 ] coordinates {(\ang{0}:1.1) (\ang{1}:1.1) (\ang{2}:1.2) (3) (4)};

						\draw[densely dashed] plot [ smooth, tension=0.8 ] coordinates {(\ang{0}:0.9) (1) (\ang{2}:1.1) (\ang{3}:0.9) (\ang{4}:0.9) (5) (\ang{6}:1.1) (7)};

						\draw[thick] plot [ smooth, tension=0.8 ] coordinates {(2) (\ang{3}:0.8) (\ang{4}:0.8) (\ang{5}:0.9) (6)};

						\draw[draw=none] (0,0) -- node[pos=1, circle, fill=black, inner sep=1pt] {} node[pos=0.8, scale=0.9] {$p=u$} (1);
						\draw[draw=none] (0,0) -- node[pos=1, circle, fill=black, inner sep=1pt] {} node[pos=0.9, xshift=-1ex] {$x$} (2);
						\draw[draw=none] (0,0) -- node[pos=1, circle, fill=black, inner sep=1pt] {} node[pos=0.9, xshift=-1ex] {$y$} (6);
						\draw[draw=none] (0,0) -- node[pos=1, circle, fill=black, inner sep=1pt] {} node[pos=1, anchor=south, scale=0.8] {$\max(\gamma_1)$} (7);

						\draw[draw=none] (0,0) -- node[pos=1, circle, fill=black, inner sep=1pt] {} node[pos=1.1, yshift=1ex] {$q$} (3);
						\draw[draw=none] (0,0) -- node[pos=1, circle, fill=black, inner sep=1pt] {} node[pos=1, anchor=west, scale=0.8] {$\max(\gamma_2)$} (4);
						\draw[draw=none] (0,0) -- node[pos=1, circle, fill=black, inner sep=1pt] {} node[pos=1.1, xshift=1ex] {$v$} (5);
					\end{tikzpicture}

					\item If $x < p < q < y$, we need to show that $(p,q) \in (\Lcn)^2 \times (\Rcn)^2$. By Lemma \ref{lem: deux arcs imbriqués avec la fin dans le même Lc Rc alors ils sont dans la période} applied on $x < p < y$ and $x < q < y$, we obtain that $(p,y), (q,y) \in (\Lcn \times \Rcn) \sqcup (\Rcn \times \Lcn)$. Therefore, $(p,q) \in (\Lcn)^2 \sqcup (\Rcn)^2$.
				\end{itemize}

				Hence the result: the arcs $p \rightarrow q$ and $x \rightarrow y$ do not cross.
			\end{enumerate}
		\end{proof}

		\begin{exemple} \label{ex: C(A) remplace chaine par boucle} Consider the cyclic $c$-noncrossing arc diagram $\mathcal{A}$ from Example \ref{ex: C_a(A)}.

			\noindent\begin{minipage}{0.7\textwidth}
				We have ($n = 10$):
				\begin{align*}
					\mathcal{C}_2(\mathcal{A}) & = {\color{gray!70!black}2 \rightarrow 3} \uparrow \mathbf{5} \uparrow {6 \rightarrow 8} \downarrow \mathbf{15} \uparrow {16 \rightarrow 18} \downarrow \mathbf{25} \uparrow {26} \cdots \\
					\mathcal{C}_2^n(\mathcal{A}) & = {\color{gray!70!black}2 \rightarrow 3} \uparrow \mathbf{5} \uparrow {6 \rightarrow 8} \downarrow \mathbf{5} \uparrow {6 \rightarrow 8} \downarrow \mathbf{5} \uparrow {6} \cdots
				\end{align*}
				Thus, $\gamma_{\tau_2^n} = \{ -3 \rightarrow 5,\; 1 \rightarrow 6 \rightarrow 8 \} + n\ZZ$. If we remove all these chains from $\mathcal{A}$ and replace them with the loop of support $\{ -3, 5, 1, 6, 8 \} + n\ZZ$, the cyclic arc diagram obtained is indeed still $c$-noncrossing.
			\end{minipage} \hfill
			\begin{tikzpicture}[baseline, scale=1.1]
				\def\n{10}

				\def\Rc{2, 3, 5, 7, 9, 10}
				\def\Lc{1, 4, 6, 8}

				\newcommand{\ang}[1]{{-(#1-1)*360/\n + 90}}

				\foreach \i in {1,...,\n}{
					\coordinate (\i) at (\ang{\i}:1);
				}

				\foreach \i in \Rc{
					\draw (0,0) -- (\i);
				}
				\foreach \i in \Lc{
					\draw (\ang{\i}:1.5) -- (\i);
				}

				\draw[gray!70!black] plot [ smooth, tension=0.8 ] coordinates {(9) (\ang{10}:1.3) (\ang{1}:0.7) (2) (3)};
				\draw[thick] plot [ smooth, tension=0.8 ] coordinates {(7) (8) (\ang{9}:1.2) (\ang{10}:1.4) (1) (\ang{2}:1.2) (\ang{3}:1.2) (\ang{4}:0.7) (5) (6) (7)};

				\foreach \i in {1,...,\n}{
					\node[circle, fill=white, inner sep=0.2pt] at (\i) {\i};
				}
			\end{tikzpicture}\hfill\phantom{}
		\end{exemple}

		We can finally prove the fact we noticed in Example \ref{ex: C_a(A)}.

		\begin{proposition} \label{prop: les périodes sont les mêmes dans C_a(A)}
			Let $\mathcal{A}$ be a cyclic $c$-noncrossing arc diagram and $a,b \in \ZZ$. If $\mathcal{C}_a(\mathcal{A})$ is infinite, then $\mathcal{C}_a^n(\mathcal{A})$ and $\mathcal{C}_b^n(\mathcal{A})$ are both ultimately periodic, and they have the same period.
		\end{proposition}

		\begin{proof}
			We already know from Proposition \ref{prop: C_a(A) est fini ou infini} that $\mathcal{C}_a^n(\mathcal{A})$ and $\mathcal{C}_b^n(\mathcal{A})$ are infinite and periodic. Let $\tau_a^n$ (resp. $\tau_b^n$) be a periodic factor of $\mathcal{C}_a^n(\mathcal{A})$ (resp. $\mathcal{C}_b^n(\mathcal{A})$). If $\tau_a^n \cap \tau_b^n \neq \emptyset$, then there exists $c \in \ZZ$ appearing in both sequences $\mathcal{C}_a(\mathcal{A})$ and $\mathcal{C}_b(\mathcal{A})$. This means these two sequences coincide after the term $c$, therefore the sequences $\mathcal{C}_a^n(\mathcal{A})$ and $\mathcal{C}_b^n(\mathcal{A})$ are equal from a certain point. Thus, they have the same period.

			We now show by contradiction that the case $\tau_a^n \cap \tau_b^n = \emptyset$ cannot happen. Suppose that it does. Then, let $\mathcal{B}_a = \left( \mathcal{A} \setminus \Gamma_{\tau_a^n} \right) \cup \widetilde{\gamma}_a$, where $\Gamma_{\tau_a^n}$ is the set of chains and loops defined in Proposition \ref{prop: on peut remplacer une boucle de chaine par une seule boucle} and $\widetilde{\gamma}_a$ is the loop of support $\bigcup_{\gamma \in \Gamma_{\tau_a^n}} \mathrm{Supp}(\gamma)$. By Proposition \ref{prop: on peut remplacer une boucle de chaine par une seule boucle}, $\mathcal{B}_a$ is a cyclic $c$-noncrossing arc diagram. Define similarly the cyclic $c$-noncrossing arc diagram $\mathcal{B}_b = \left( \mathcal{A} \setminus \Gamma_{\tau_b^n} \right) \cup \widetilde{\gamma}_b$. We have that $\mathrm{Supp}(\widetilde{\gamma}_a) \cap \mathrm{Supp}(\widetilde{\gamma}_b) = \emptyset$, because $\Gamma_{\tau_a^n} \cap \Gamma_{\tau_b^n} = \emptyset$. Since $\widetilde{\gamma}_b$ is a (real) loop, there exists $x \rightarrow y \in \widetilde{\gamma}_b$ such that $(x,y) \in (\Lcn \times \Rcn) \sqcup (\Rcn \times \Lcn)$. We know that neither $x$ nor $y$ belong to a periodic factor of $\mathcal{C}_a(\mathcal{B})$, as such a periodic factor is a segment of $\mathrm{Supp}(\widetilde{\gamma}_a)$.

			Let $p \rightarrow p',q \rightarrow q'$ be two arcs of $\widetilde{\gamma}_a$ such that $p < x < p'$ and $q < y < q'$, which is possible to consider since $\widetilde{\gamma}_a$ is a loop and has infinitely many points. Therefore, by Lemma \ref{lem: deux arcs imbriqués avec la fin dans le même Lc Rc alors ils sont dans la période} applied to $p < x < p'$ and $q < y < q'$, we obtain that $(x,p'), (y,q') \in (\Lcn \times \Rcn) \sqcup (\Rcn \times \Lcn)$. If $p' = q'$, then $(x,y) \in (\Lcn)^2 \sqcup (\Rcn)^2$. If $p' \neq q'$, then $x < p' < y < q'$, and by Corollary \ref{cor: croisement de chaîne} we have $(p',y) \in (\Lcn \times \Rcn) \sqcup (\Rcn \times \Lcn)$, hence $(x,y) \in (\Lcn)^2 \sqcup (\Rcn)^2$. This is a contradiction in both cases, therefore $\tau_a^n = \tau_b^n$.
		\end{proof}

		Proposition \ref{prop: les périodes sont les mêmes dans C_a(A)} allows us to not specify the index $a$ in the period of $\mathcal{C}_a^n(\mathcal{A})$. We may also omit this index for the sequence itself, in this case the notation $\mathcal{C}^n(\mathcal{A})$ refers to any periodic final segment of a sequence $\mathcal{C}_a^n(\mathcal{A})$ for any $a \in \ZZ$.

		One final result that will be useful in the next section is a criterion for a chain or a loop to have some of its elements in the period of the sequence $\mathcal{C}^n(\mathcal{A})$.

		\begin{lemme} \label{lem: chaine dans Lc et Rc alors dans période}
			Let $\mathcal{A}$ be a cyclic $c$-noncrossing arc diagram such that $\mathcal{C}^n(\mathcal{A})$ is infinite, and let $\tau^n$ be a periodic factor of this sequence. For all chain or loop $\gamma \in \mathrm{Ch}(\mathcal{A}) \cup \mathrm{L}(\mathcal{A})$, if $\mathrm{Supp}(\gamma) \cap \Lcn \neq \emptyset$ and $\mathrm{Supp}(\gamma) \cap \Rcn \neq \emptyset$, then $\mathrm{Supp}(\gamma) \cap (\tau^n + n\ZZ) \neq \emptyset$.

			In other words, if a chain or a loop contains an element of $\Lcn$ and an element of $\Rcn$, then some of its elements belong to the period $\tau^n$.
		\end{lemme}

		\begin{proof}
			If $\tau^n$ does not contain an arrow $\uparrow$ or $\downarrow$, then $\mathcal{A}$ has a loop $\widetilde{\gamma}$ and $\tau^n$ contains only elements of its support. Moreover, any arc $p \rightarrow q$ in $\mathcal{A} \setminus \widetilde{\gamma}$ has both its extremities in the same set $\Lcn$ or $\Rcn$, otherwise it would intersect the loop. Therefore, the result is true.

			Otherwise, we know that $\mathcal{A}$ does not have any loop. Consider $p \rightarrow q$ an arc of a chain $\gamma$ such that $(p,q) \in (\Lcn \times \Rcn) \cup (\Rcn \times \Lcn)$. If $\gamma \notin \Gamma_{\tau^n}$ (using the notation from  Proposition \ref{prop: on peut remplacer une boucle de chaine par une seule boucle}), then by this Proposition there exists a diagram with a loop and the arc $p \rightarrow q$, which is absurd since such a loop will cross this arc. Therefore, $\gamma \in \Gamma_{\tau^n}$ and by definition, it has elements in $\tau^n$.
		\end{proof}

\section{Generalization of \textsc{Reading}'s bijection} \label{sec: generalization of reading's bijection}

	In this section, we define a generalization of \textsc{Reading}'s map $nc_c$, which we denote $nc_c^{\widetilde{A}}$, and prove Theorem \ref{thm: generalisation de l'application de Reading}.

	\readinggeneralise*

	\subsection{A map from \texorpdfstring{$c$}{c}-sortable biclosed sets to \texorpdfstring{$c$}{c}-noncrossing partitions}

		We first define a bijective map between $c$-sortable TITOs on $\ZZ$ and cyclic $c$-noncrossing arc diagrams by restriction of \textsc{Barkley}'s bijection defined in Paragraph \ref{par: cyclic noncrossing arc diagrams} of Section \ref{sec: background}.

		If $\prec$ is a $c$-sortable TITO on $\ZZ$, then by definition it does not have two consecutive waxing blocks. This means the map $\ncad$ can indeed be restricted to the set of $c$-sortable TITOs on $\ZZ$. We note this restriction $\ncad_c$. At this point, $\ncad_c$ is an injective map from the set of $c$-sortable TITOs on $\ZZ$ to the set $\widetilde{\mathrm{NCAD}}$ of cyclic noncrossing arc diagrams.

		\begin{lemme}
			The map $\ncad_c$ takes values in $c$-$\widetilde{\mathrm{NCAD}}$, the set of cyclic $c$-noncrossing arc diagrams.
		\end{lemme}

		\begin{proof}
			Let $\prec$ be a $c$-sortable TITO on $\ZZ$. We need to show that for all $\alpha = (i,k,L,R) \in \ncad_c(\prec)$, $L \subset \Lcn$ and $R \subset \Rcn$, as well as $\ncad_c(\prec)$ has no imaginary loop (Definition \ref{def: full chains and loops}). Let $j \in L$, by definition, we have $j \prec k \prec i$. Since $\prec$ is $c$-sortable, then $j \in \Lcn$. The same argument shows that $R \subset \Rcn$. Moreover, a loop of $\ncad_c(\prec)$ comes from a maximal descending chain of $\prec$ such that its first and last element are equal modulo $n$. This can only happen in a waning block in which any window consists of decreasing elements. By Definition \ref{def: biclos c-triable}, there are elements from both $\Lcn$ and $\Rcn$ in the waning block of $\prec$ if it has one, therefore $\ncad_c(\prec)$ does not have imaginary loops.

			Hence, $\ncad_c(\prec)$ is a cyclic $c$-noncrossing arc diagram.
		\end{proof}

		Note that, under the bijection $\mathrm{NC}_c$ between cyclic $c$-noncrossing arc diagrams and $c$-noncrossing partitions defined in Section \ref{sssec: bijection between NCAD and NC}, we can already define what will be the \emph{generalized \textsc{Reading} bijection}, denoted
		\begin{equation} \label{eq: nc_c exprimé comme composée}
			nc_c^{\widetilde{A}} = \mathrm{NC}_c \circ \ncad_c \circ \pprec.
		\end{equation} It is an injective map from the set of $c$-sortable biclosed sets of the affine root system to the set of $c$-noncrossing partitions that extends the right part of the factorization $nc_c : \mathrm{Sort}_c(W) \hookrightarrow c\text{-}\mathrm{Bic}(W) \xrightarrow{\sim} \mathrm{NC}(W,c)$. Despite this composition of bijections, the map $nc_c^{\widetilde{A}}$ is explicit (see Example \ref{ex: direct map} for a computation).

		\begin{definition}[Generalized \textsc{Reading} bijection] \label{def: reading généralisée}
			Let $B$ be a $c$-sortable biclosed set of $W$. If $\pprec(B)$ is of shape $[L]\underline{[M]}[R]$:
			\begin{itemize}
				\item let $L_1, \dots, L_l \subset \ZZ$ ($l \geq 1$) be orbits of the maximal descending chains of the block $[L]$,
				\item let $M_1, \dots, M_m \subset \ZZ$ ($m \geq 1$) be orbits of the maximal descending chains of the block $\underline{[M]}$,
				\item let $R_1, \dots, R_l \subset \ZZ$ ($r \geq 1$) be orbits of the maximal descending chains of the block $[R]$.
			\end{itemize}
			Otherwise, $\pprec(B)$ is of shape $[\llbracket 1,n \rrbracket]$: let $m = r = 0$ and $L_1, \dots, L_l \subset \ZZ$ ($l \geq 1$) be orbits of the maximal descending chains of $\pprec(B)$. Despite the use of the letter $L$, the sets $L_1, \dots, L_l$ may contain elements of $\Rcn$: this notation is there to use the same notations for both types of shape.

			Using the results of Paragraph \ref{par: partitions non croisées Sn tilde} in Section \ref{sec: background}, for each $i \in \llbracket 1,l \rrbracket$ (resp. $i \in \llbracket 1,r \rrbracket$, resp. $i \in \llbracket 1,m \rrbracket$ if $m \geq 2$) we define $\gamma_i^L$ (resp. $\gamma_i^R$, resp. $\gamma_i^M$) to be the unique cyclic elementary divisor of $c$ of orbit $L_i$ (resp. $R_i$, resp. $M_i$). If $m = 1$, we define $\gamma_1^M$ to be the unique pseudo-cyclic elementary divisor of $c$ of orbit $M_1 + n\ZZ$.

			We define the image of $B$ by the generalized \textsc{Reading} map in type $\widetilde{A}$ as follow:
			\begin{equation*}
				nc_c^{\widetilde{A}}(B) = \prod_{i = 1}^l \gamma_i^L \prod_{i = 1}^m \gamma_i^M \prod_{i = 1}^r \gamma_i^R.
			\end{equation*}
		\end{definition}

		We now need to show that the map $\ncad_c$ is surjective onto the set of cyclic $c$-noncrossing arc diagrams in order to obtain a bijective map. This will immediately prove that $nc_c^{\widetilde{A}}$ is also a bijection. We achieve this by explicitly computing its inverse map. \textsc{Barkley} computed the preimages of all cyclic noncrossing arc diagrams consisting of a single cyclic arc by the map $\ncad$ \cite[§4]{barkley2025extendedweakorderaffine}, and the preimages for any cyclic noncrossing arc diagram can be obtained as the join (in the extended weak order) of these preimages for all cyclic arcs in the diagram. However, it does not seem easy to see that our definition of $c$-sortable TITOs on $\ZZ$ are indeed the joins of such elements. We will instead explicitly compute the inverse map.

	\subsection{Renumbering a cyclic \texorpdfstring{$c$}{c}-noncrossing arc diagram} \label{ssec: renumérotation}

		The preimage by $\ncad_c$ of a cyclic $c$-noncrossing arc diagram will be obtained by «~reading~» the diagram in a specific way, similar to the description of the inverse map of \textsc{Reading}'s bijection in type $A$ \cite[6.3]{gobet2018dual}. However, some complexity arises from the fact that the labeled points of a cyclic noncrossing arc diagram are considered modulo $n$. The sequence $\mathcal{C}(\mathcal{A})$, defined in Section \ref{sssec: structure des c-NCAD + suite C}, will be of great help in order to renumber (with labels in $\ZZ$ instead of $\ZZ/n\ZZ$) the diagram.

		\begin{definition}[Numbering of a cyclic $c$-noncrossing arc diagram] \label{def: renumérotation}
			Let $\mathcal{A}$ be a cyclic $c$-noncrossing arc diagram. A \emph{numbering} of $\mathcal{A}$ is a transversal of the quotient set $\left(\mathrm{Ch}(\mathcal{A}) \cup \mathrm{L}(\mathcal{A}) \right) / \sim$ where $\gamma \sim \gamma'$ if and only if there exists $p \in \ZZ$ such that $\gamma = \gamma' + pn$.
		\end{definition}

		In other words, it is a subset $\mathcal{N}$ of $\mathrm{Ch}(\mathcal{A}) \cup \mathrm{L}(\mathcal{A})$ in which exactly one representative of each class of chains and loops modulo $n$ belongs to $\mathcal{N}$. Note that loops are alone in their class modulo $n$, hence there is no choice for numbering them.

		The term «~renumbering~» refers to the fact that, geometrically on the diagram, a choice of such a transversal corresponds to a renumbering of the labels from $1$ to $n$ (or $0$ to $n-1$) in a way that specializes each drawn chain of arcs (representing a class of chain modulo $n$) into one specific chain. Examples of renumberings are given in Examples \ref{ex: a-renumérotation arc} and \ref{ex: renumérotation arc}.

		Among all the possible numberings of a cyclic $c$-noncrossing arc diagram, we are interested in specific ones we call \emph{$a$-numberings} for $a \in \ZZ$. The definition is quite involved and it is not trivial to see that they indeed define numberings, see Proposition \ref{prop: les a numérotations sont des numérotations}.

		\begin{definition}[$a$-numbering of a cyclic $c$-noncrossing arc diagram] \label{def: a-renumérotation}
			Let $a \in \ZZ$ and consider the sequence $\mathcal{C}_a(\mathcal{A})$. If it is finite, with $b$ as its last element, define
			\begin{equation*}
				\mathcal{N}_a(\mathcal{A}) = \{ \gamma \in \mathrm{Ch}(\mathcal{A}) \;|\; \mathrm{Supp}(\gamma) \subset \rrbracket b-n, b \rrbracket \}.
			\end{equation*}

			Otherwise, consider the first factor $\tau_a$ of $\mathcal{C}_a(\mathcal{A})$ such that $\tau_a \bmod n$ is a periodic factor of $\mathcal{C}_a^n(\mathcal{A})$ and $\tau_a$ starts with $\downarrow$ if possible, otherwise with $\uparrow$ if possible, and otherwise with $\rightarrow$.

			Define $\mathcal{M}_a(\mathcal{A})$ to be the set of all chains or loops containing a number that is in $\tau_a$. By Lemma \ref{lem: chaine dans Lc et Rc alors dans période}, any chain such that no representative modulo $n$ is in $\mathcal{M}_a(\mathcal{A})$ has all its elements either in $\Lcn$ or $\Rcn$.
			\begin{equation*}
				\mathcal{M}_a(\mathcal{A}) = \{ \gamma \in \mathrm{Ch}(\mathcal{A}) \cup \mathrm{L}(\mathcal{A}) \;|\; \mathrm{Supp}(\gamma) \cap \tau_a \neq \emptyset \}.
			\end{equation*}

			If $\tau_a$ starts with $\downarrow$, consider $M$ (resp. $m$) the maximal (resp. minimal) number that appears in a chain in $\mathcal{M}_a(\mathcal{A})$, which is possible since in this case $\mathcal{M}_a(\mathcal{A})$ does not contain a loop. We select the remaining chains in the interval $\rrbracket M-n, M \llbracket$ if they are in $\Lcn$ and $\rrbracket m, m+n \llbracket$ otherwise:
			\begin{equation*}
				\left\lbrace
				\begin{array}{>{\displaystyle}l}
					\mathcal{L}_a(\mathcal{A}) = \{ \gamma \in \mathrm{Ch}(\mathcal{A}) \;|\; \mathrm{Supp}(\gamma) \subset \Lcn \cap \rrbracket M-n, M \llbracket \setminus (\tau_a + n\ZZ) \} \\
					\mathcal{R}_a(\mathcal{A}) = \{ \gamma \in \mathrm{Ch}(\mathcal{A}) \;|\; \mathrm{Supp}(\gamma) \subset \Rcn \cap \rrbracket m, m+n \llbracket \setminus (\tau_a + n\ZZ) \} .
				\end{array}
				\right.
			\end{equation*}

			Otherwise, consider the two cyclic $c$-noncrossing arc diagrams $\mathcal{A}_L$ and $\mathcal{A}_R$ obtained by only keeping the chains that only have elements of $\Lcn$ and $\Rcn$ respectively, and then removing all the chains or loops that have a representative modulo $n$ in $\mathcal{M}_a(\mathcal{A})$. Their sequences $\mathcal{C}_a(\mathcal{A}_L)$ and $\mathcal{C}_a(\mathcal{A}_R)$ are finite, therefore we can define:
			\begin{equation*}
				\left\lbrace
				\begin{array}{>{\displaystyle}l}
					\mathcal{L}_a(\mathcal{A}) = \mathrm{Ch}(\mathcal{A}) \cap \mathcal{N}_a(\mathcal{A}_L) \\
					\mathcal{R}_a(\mathcal{A}) = \mathrm{Ch}(\mathcal{A}) \cap \mathcal{N}_a(\mathcal{A}_R).
				\end{array}
				\right.
			\end{equation*}
			In both cases, let
			\begin{equation*}
				\mathcal{N}_a(\mathcal{A}) = \mathcal{L}_a(\mathcal{A}) \cup \mathcal{M}_a(\mathcal{A}) \cup \mathcal{R}_a(\mathcal{A}).
			\end{equation*}

			The set of chains or loops $\mathcal{N}_a(\mathcal{A})$ is called the $a$-numbering of $\mathcal{A}$.
		\end{definition}

		\begin{exemple} \label{ex: a-renumérotation arc}
			A numbering of a cyclic $c$-noncrossing arc diagram is represented on the diagram by specifying each class of chains modulo $n$ into a representative. To distinguish labels in $\ZZ/n\ZZ$ from labels in $\ZZ$, we draw a box around the labels in $\ZZ$. For example, consider the cyclic $c$-noncrossing arc diagram in Figure \ref{subfig: initial diagram}, for $c = \affc{1,4,6,8}_{[1]} \affc{10,9,7,5,3,2}_{[-1]}$.
			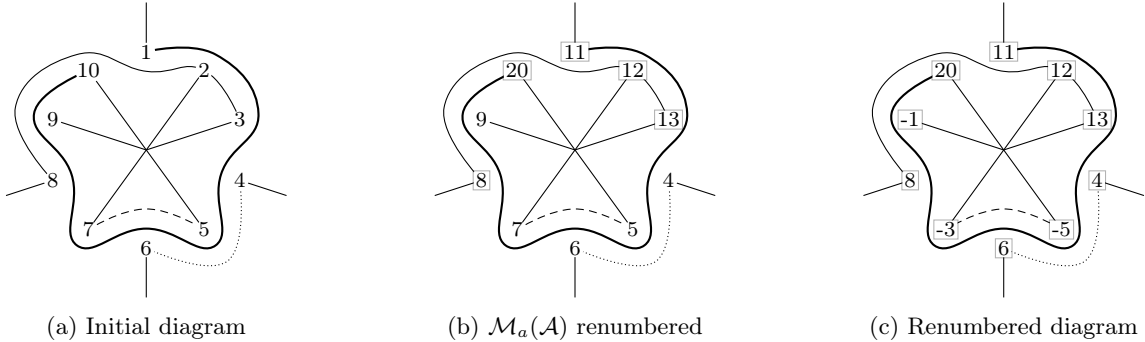
\begin{figure}[ht!]
				\centering
				\begin{subfigure}{0.333\textwidth}
					\centering
					\begin{tikzpicture}[baseline=(current bounding box.center), scale=1.3]
						\def\n{10}

						\def\Rc{2, 3, 5, 7, 9, 10}
						\def\Lc{1, 4, 6, 8}

						\newcommand{\ang}[1]{{-(#1-1)*360/\n + 90}}

						\foreach \i in {1,...,\n}{
							\coordinate (\i) at (\ang{\i}:1);
						}

						\foreach \i in \Rc{
							\draw (0,0) -- (\i);
						}
						\foreach \i in \Lc{
							\draw (\ang{\i}:1.5) -- (\i);
						}

						\draw[black, thick] plot [ smooth, tension=0.8 ] coordinates {(1) (\ang{2}:1.2) (\ang{3}:1.2) (\ang{4}:0.8) (\ang{5}:1.2) (\ang{6}:0.8) (\ang{7}:1.2) (\ang{8}:0.8) (\ang{9}:1.2) (10)};

						\draw[black] plot [ smooth, tension=0.8 ] coordinates {(8) (\ang{9}:1.4) (\ang{10}:1.2) (\ang{1}:0.8) (2) (3)};

						\draw[black, densely dotted] plot [ smooth, tension=0.8 ] coordinates {(4) (\ang{5}:1.4) (6)};

						\draw[black, densely dashed] plot [ smooth, tension=0.8 ] coordinates {(5) (\ang{6}:0.6) (7)};

						\foreach \i in {1,...,\n}{
							\node[fill=white, circle, inner sep=0pt, font=\footnotesize] at (\i) {\i};
						}
					\end{tikzpicture}
					\caption{Initial diagram}
					\label{subfig: initial diagram}
				\end{subfigure}%
				\begin{subfigure}{0.334\textwidth}
					\centering
					\begin{tikzpicture}[baseline=(current bounding box.center), scale=1.3]
						\def\n{10}

						\def\Rc{2, 3, 5, 7, 9, 10}
						\def\Lc{1, 4, 6, 8}
						\def\NUM{1/11, 2/12, 3/13, 8/8, 10/20}

						\newcommand{\ang}[1]{{-(#1-1)*360/\n + 90}}

						\foreach \i in {1,...,\n}{
							\coordinate (\i) at (\ang{\i}:1);
						}

						\foreach \i in \Rc{
							\draw (0,0) -- (\i);
						}
						\foreach \i in \Lc{
							\draw (\ang{\i}:1.5) -- (\i);
						}

						\draw[black, thick] plot [ smooth, tension=0.8 ] coordinates {(1) (\ang{2}:1.2) (\ang{3}:1.2) (\ang{4}:0.8) (\ang{5}:1.2) (\ang{6}:0.8) (\ang{7}:1.2) (\ang{8}:0.8) (\ang{9}:1.2) (10)};

						\draw[black] plot [ smooth, tension=0.8 ] coordinates {(8) (\ang{9}:1.4) (\ang{10}:1.2) (\ang{1}:0.8) (2) (3)};

						\draw[black, densely dotted] plot [ smooth, tension=0.8 ] coordinates {(4) (\ang{5}:1.4) (6)};

						\draw[black, densely dashed] plot [ smooth, tension=0.8 ] coordinates {(5) (\ang{6}:0.6) (7)};

						\foreach \i/\j in \NUM{
							\node[rectangle, fill=white, inner sep=1pt, draw, font=\footnotesize, draw opacity=0.3] at (\i) {\j};
						}
						\foreach \i in {4,5,6,7,9}{
							\node[fill=white, circle, inner sep=0pt, font=\footnotesize] at (\i) {\i};
						}
					\end{tikzpicture}
					\caption{$\mathcal{M}_a(\mathcal{A})$ renumbered}
					\label{subfig: renumérotation partielle}
				\end{subfigure}%
				\begin{subfigure}{0.333\textwidth}
					\centering
					\begin{tikzpicture}[baseline=(current bounding box.center), scale=1.3]
						\def\n{10}

						\def\Rc{2, 3, 5, 7, 9, 10}
						\def\Lc{1, 4, 6, 8}
						\def\NUM{1/11, 2/12, 3/13, 4/4, 5/-5, 6/6, 7/-3, 8/8, 9/-1, 10/20}

						\newcommand{\ang}[1]{{-(#1-1)*360/\n + 90}}

						\foreach \i in {1,...,\n}{
							\coordinate (\i) at (\ang{\i}:1);
						}

						\foreach \i in \Rc{
							\draw (0,0) -- (\i);
						}
						\foreach \i in \Lc{
							\draw (\ang{\i}:1.5) -- (\i);
						}

						\draw[black, thick] plot [ smooth, tension=0.8 ] coordinates {(1) (\ang{2}:1.2) (\ang{3}:1.2) (\ang{4}:0.8) (\ang{5}:1.2) (\ang{6}:0.8) (\ang{7}:1.2) (\ang{8}:0.8) (\ang{9}:1.2) (10)};

						\draw[black] plot [ smooth, tension=0.8 ] coordinates {(8) (\ang{9}:1.4) (\ang{10}:1.2) (\ang{1}:0.8) (2) (3)};

						\draw[black, densely dotted] plot [ smooth, tension=0.8 ] coordinates {(4) (\ang{5}:1.4) (6)};

						\draw[black, densely dashed] plot [ smooth, tension=0.8 ] coordinates {(5) (\ang{6}:0.6) (7)};

						\foreach \i/\j in \NUM{
							\node[rectangle, fill=white, inner sep=1pt, draw, font=\footnotesize, draw opacity=0.3] at (\i) {\j};
						}
					\end{tikzpicture}
					\caption{Renumbered diagram}
					\label{subfig: renumérotation finie}
				\end{subfigure}
				\caption{Example of the steps for renumbering a cyclic $c$-noncrossing arc diagram.}
				\label{fig: renumérotation avec que des ↑}
			\end{figure}%

			We have $\mathcal{C}_4(\mathcal{A}) = {\color{black!70}4 \rightarrow 6} \downarrow \mathbf{10} \uparrow 12 \rightarrow 13 \uparrow \mathbf{20} \uparrow 22 \rightarrow 23 \uparrow \cdots$. Therefore, $\tau_4 = \uparrow 12 \rightarrow 13 \uparrow \mathbf{20}$. This means that $\mathbf{11 \rightarrow 20}$ and $8 \rightarrow 12 \rightarrow 13$ are the representative chains of their class modulo $n$, we specialize their points on the diagram by putting them in a square box in Figure \ref{subfig: renumérotation partielle}.

			Since $\tau_4$ starts with $\uparrow$, we consider the two cyclic $c$-noncrossing arc diagrams $\mathcal{A}_L$ and $\mathcal{A}_R$ obtained by keeping only the chains that have not been boxed and that are in $\Lcn$ and $\Rcn$ respectively (Figures \ref{subfig: AL} and \ref{subfig: AR}).

			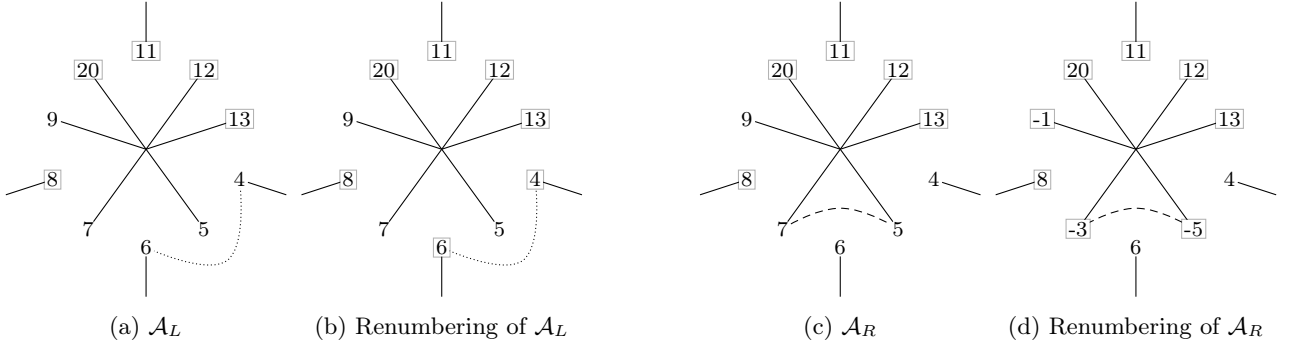
\begin{figure}[ht!]
				\centering
				\begin{subfigure}{0.23\textwidth}
					\centering
					\begin{tikzpicture}[baseline=(current bounding box.center), scale=1.3]
						\def\n{10}

						\def\Rc{2, 3, 5, 7, 9, 10}
						\def\Lc{1, 4, 6, 8}
						\def\NUM{1/11, 2/12, 3/13, 8/8, 10/20}

						\newcommand{\ang}[1]{{-(#1-1)*360/\n + 90}}

						\foreach \i in {1,...,\n}{
							\coordinate (\i) at (\ang{\i}:1);
						}

						\foreach \i in \Rc{
							\draw (0,0) -- (\i);
						}
						\foreach \i in \Lc{
							\draw (\ang{\i}:1.5) -- (\i);
						}

						\draw[black, densely dotted] plot [ smooth, tension=0.8 ] coordinates {(4) (\ang{5}:1.4) (6)};

						\foreach \i/\j in \NUM{
							\node[rectangle, fill=white, inner sep=1pt, draw, font=\footnotesize, draw opacity=0.3] at (\i) {\j};
						}
						\foreach \i in {4,5,6,7,9}{
							\node[fill=white, circle, inner sep=0pt, font=\footnotesize] at (\i) {\i};
						}
					\end{tikzpicture}
					\caption{$\mathcal{A}_L$}
					\label{subfig: AL}
				\end{subfigure}%
				\begin{subfigure}{0.23\textwidth}
					\centering
					\begin{tikzpicture}[baseline=(current bounding box.center), scale=1.3]
						\def\n{10}

						\def\Rc{2, 3, 5, 7, 9, 10}
						\def\Lc{1, 4, 6, 8}
						\def\NUM{1/11, 2/12, 3/13, 4/4, 6/6, 8/8, 10/20}

						\newcommand{\ang}[1]{{-(#1-1)*360/\n + 90}}

						\foreach \i in {1,...,\n}{
							\coordinate (\i) at (\ang{\i}:1);
						}

						\foreach \i in \Rc{
							\draw (0,0) -- (\i);
						}
						\foreach \i in \Lc{
							\draw (\ang{\i}:1.5) -- (\i);
						}

						\draw[black, densely dotted] plot [ smooth, tension=0.8 ] coordinates {(4) (\ang{5}:1.4) (6)};

						\foreach \i/\j in \NUM{
							\node[rectangle, fill=white, inner sep=1pt, draw, font=\footnotesize, draw opacity=0.3] at (\i) {\j};
						}
						\foreach \i in {5,7,9}{
							\node[fill=white, circle, inner sep=0pt, font=\footnotesize] at (\i) {\i};
						}
					\end{tikzpicture}
					\caption{Renumbering of $\mathcal{A}_L$}
					\label{subfig: renumbering AL}
				\end{subfigure}\hfill
				\begin{subfigure}{0.23\textwidth}
					\centering
					\begin{tikzpicture}[baseline=(current bounding box.center), scale=1.3]
						\def\n{10}

						\def\Rc{2, 3, 5, 7, 9, 10}
						\def\Lc{1, 4, 6, 8}
						\def\NUM{1/11, 2/12, 3/13, 8/8, 10/20}

						\newcommand{\ang}[1]{{-(#1-1)*360/\n + 90}}

						\foreach \i in {1,...,\n}{
							\coordinate (\i) at (\ang{\i}:1);
						}

						\foreach \i in \Rc{
							\draw (0,0) -- (\i);
						}
						\foreach \i in \Lc{
							\draw (\ang{\i}:1.5) -- (\i);
						}

						\draw[black, densely dashed] plot [ smooth, tension=0.8 ] coordinates {(5) (\ang{6}:0.6) (7)};

						\foreach \i/\j in \NUM{
							\node[rectangle, fill=white, inner sep=1pt, draw, font=\footnotesize, draw opacity=0.3] at (\i) {\j};
						}
						\foreach \i in {4,5,6,7,9}{
							\node[fill=white, circle, inner sep=0pt, font=\footnotesize] at (\i) {\i};
						}
					\end{tikzpicture}
					\caption{$\mathcal{A}_R$}
					\label{subfig: AR}
				\end{subfigure}%
				\begin{subfigure}{0.23\textwidth}
					\centering
					\begin{tikzpicture}[baseline=(current bounding box.center), scale=1.3]
						\def\n{10}

						\def\Rc{2, 3, 5, 7, 9, 10}
						\def\Lc{1, 4, 6, 8}
						\def\NUM{1/11, 2/12, 3/13, 5/-5, 7/-3, 8/8, 9/-1, 10/20}

						\newcommand{\ang}[1]{{-(#1-1)*360/\n + 90}}

						\foreach \i in {1,...,\n}{
							\coordinate (\i) at (\ang{\i}:1);
						}

						\foreach \i in \Rc{
							\draw (0,0) -- (\i);
						}
						\foreach \i in \Lc{
							\draw (\ang{\i}:1.5) -- (\i);
						}

						\draw[black, densely dashed] plot [ smooth, tension=0.8 ] coordinates {(5) (\ang{6}:0.6) (7)};

						\foreach \i/\j in \NUM{
							\node[rectangle, fill=white, inner sep=1pt, draw, font=\footnotesize, draw opacity=0.3] at (\i) {\j};
						}
						\foreach \i in {4,6}{
							\node[fill=white, circle, inner sep=0pt, font=\footnotesize] at (\i) {\i};
						}
					\end{tikzpicture}
					\caption{Renumbering of $\mathcal{A}_R$}
					\label{subfig: renumbering AR}
				\end{subfigure}
				\caption{Renumbering $\mathcal{A}_L$ and $\mathcal{A}_R$}
				\label{fig: renumbering AL AR}
			\end{figure}

			Moreover, $\mathcal{C}_4(\mathcal{A}_L) = 4 \rightarrow 6$ and $\mathcal{C}_4(\mathcal{A}_R) = 4$, therefore we number the chains of $\mathcal{A}_L$ that are in $\Lcn$ between $\rrbracket -4, 6 \rrbracket$ and the chains of $\mathcal{A}_R$ that are in $\Rcn$ between $\rrbracket -6, 4 \rrbracket$, which we have done in Figures \ref{subfig: renumbering AL} and \ref{subfig: renumbering AR}.

			Finally, the $4$-numbering of $\mathcal{A}$ is the set of chains $\mathcal{N}_4(\mathcal{A})$ given by the following choices of representatives is the one represented in Figure \ref{subfig: renumérotation finie}.
		\end{exemple}

		\begin{exemple} \label{ex: renumérotation arc}
			The $2$-numbering of the cyclic $c$-noncrossing diagram of Example \ref{ex: C_a(A)} is depicted in Figure \ref{fig: renumérotation de l'exemple de C_a(A)}, and is obtained from the case where $\tau_a$ starts with $\downarrow$.
			\begin{figure}[ht!]
				\centering
				\begin{tikzpicture}[baseline=(current bounding box.center), scale=1.3]
					\def\n{10}

					\def\Rc{2, 3, 5, 7, 9, 10}
					\def\Lc{1, 4, 6, 8}
					\def\NUM{1/11, 2/12, 3/13, 4/14, 5/15, 6/16, 7/7, 8/18, 9/9, 10/10}

					\newcommand{\ang}[1]{{-(#1-1)*360/\n + 90}}

					\foreach \i in {1,...,\n}{
						\coordinate (\i) at (\ang{\i}:1);
					}

					\foreach \i in \Rc{
						\draw (0,0) -- (\i);
					}
					\foreach \i in \Lc{
						\draw (\ang{\i}:1.7) -- (\i);
					}

					\draw[gray!70!black] plot [ smooth, tension=0.8 ] coordinates {(9) (\ang{10}:1.3) (\ang{1}:0.7) (2) (3)};
					\draw[densely dashed] plot [ smooth, tension=0.8 ] coordinates {(1) (\ang{2}:1.3) (\ang{3}:1.4) (\ang{4}:0.8) (\ang{5}:1.3) (6) (\ang{7}:1.3) (8)};
					\draw[thick] plot [ smooth, tension=0.8 ] coordinates {(7) (\ang{8}:0.7) (\ang{9}:1.3) (\ang{10}:1.4) (\ang{1}:0.8) (\ang{2}:1.2) (\ang{3}:1.2) (\ang{4}:0.7) (5)};

					\foreach \i/\j in \NUM{
						\node[rectangle, fill=white, inner sep=1pt, draw, font=\footnotesize, draw opacity=0.3] at (\i) {\j};
					}
				\end{tikzpicture}
				\caption{The $2$-numbering of the cyclic $c$-noncrossing arc diagram of Example \ref{ex: C_a(A)}}
				\label{fig: renumérotation de l'exemple de C_a(A)}
			\end{figure}
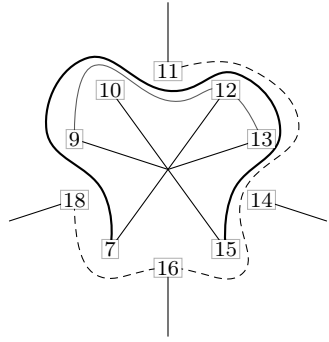%
		\end{exemple}

		\begin{proposition} \label{prop: les a numérotations sont des numérotations}
			Let $\mathcal{A}$ be a cyclic $c$-noncrossing arc diagram and $a \in \ZZ$. Then the $a$-numbering of $\mathcal{A}$ is well defined and is a numbering of $\mathcal{A}$.
		\end{proposition}

		\begin{proof}
			We need to show that for every class of chains or loops modulo $n$, there is exactly one representative in $\mathcal{N}_a(\mathcal{A})$.

			In the case where $\mathcal{C}_a(\mathcal{A})$ is finite, of last element $b \in \ZZ$, then all the integers $b + pn$ for $p \in \ZZ$ have no neighbor arc and they are the final points of the chains they belong to. This means for each arc $u \rightarrow v$ in $\mathcal{A}$, if either $u$ or $v$ is in $\rrbracket b-n, b \rrbracket$, then the other is also in this interval. Therefore, $\mathcal{N}_a(\mathcal{A})$ indeed contains exactly one chain from each class of chains modulo $n$, and since $\mathcal{A}$ contains no loop, then it is a numbering of $\mathcal{A}$.

			In the other case, if $\mathcal{A}$ contains a loop, then it is the only element of $\mathcal{M}_a(\mathcal{A})$ because any periodic factor of $\mathcal{C}_a(\mathcal{A})$ consists only of numbers present in the support of this loop. Since a loop is the only representative of its class modulo $n$, there is exactly one representative modulo $n$ of the class of this loop in $\mathcal{N}_a(\mathcal{A})$. Otherwise, $\mathcal{A}$ contains no loop and therefore $\tau_a$ starts with either $\uparrow$ or $\downarrow$. Let $\Gamma = \{ \gamma_0 + pn \;|\; n \in \ZZ \}$ be the class modulo $n$ of some chain $\gamma_0 \in \mathrm{Ch}(\mathcal{A})$. We prove that there is exactly one chain of $\Gamma$ that belongs to $\mathcal{N}_a(\mathcal{A})$.

			If there is an integer in $\tau_a$ that is also present in a chain of $\Gamma$, then a representative of $\Gamma$ is present in $\mathcal{M}_a(\mathcal{A})$ and therefore in $\mathcal{N}_a(\mathcal{A})$. Moreover, it is the only representative possible: indeed, since $\tau_a$ starts with $\updownarrow$, all the integers between two $\updownarrow$s belong to the same chain and two integers separated by $\updownarrow$s belong to different chains. Moreover, the maximal element of every chain that appears in $\tau_a$ is present in $\tau_a$, thus two distinct chains congruent modulo $n$ cannot both appear in $\tau_a$. Hence, there is exactly one representative of $\Gamma$ in $\mathcal{N}_a(\mathcal{A})$.

			Otherwise, there is no integer in $\tau_a$ that is also present in a chain of $\Gamma$. Then we know by Lemma \ref{lem: chaine dans Lc et Rc alors dans période} that all the chains in $\Gamma$ have their supports either in $\Lcn$ or $\Rcn$. Therefore, for all $\gamma \in \Gamma$, we have $\max(\gamma) < \min(\gamma) + n$, otherwise $\gamma$ would intersect $\gamma + n$. If $\tau_a$ starts with $\uparrow$, then $\Gamma \subset \mathcal{A}_L$ or $\Gamma \subset \mathcal{A}_R$ therefore there is a unique chain $\gamma \in \Gamma$ that is in $\mathcal{L}_a(\mathcal{A}) \cup \mathcal{R}_a(\mathcal{A})$ by construction. Else, $\tau_a$ starts with $\downarrow$. This means that $M$, the maximal element that appears in the chains of $\mathcal{M}_a(\mathcal{A})$, is in $\Lcn$. For $m$, the minimal element that appears in these chains, we know by Corollary \ref{cor: m appartient à la première chaine} that it belongs to the first chain appearing in $\tau_a$. Thus, by Lemma \ref{lem: si facteur c|d dans tau_a alors enchevetrement des chaînes} and Corollary \ref{cor: croisement de chaîne} applied to the chain containing $M$ and the chain containing $m+n$, we have that $m$ is in $\Rcn$. Consider $\gamma_0 \in \Gamma$ to be the unique chain such that $M-n < \min(\gamma) < M$ if $\mathrm{Supp}(\gamma) \subset \Lcn$ or $m < \min(\gamma) < m+n$ if $\mathrm{Supp}(\gamma) \subset \Rcn$. By Lemma \ref{lem: deux arcs imbriqués avec la fin dans le même Lc Rc alors ils sont dans la période}, we cannot have $\min(\gamma) < M < \max(\gamma)$ if $\mathrm{Supp}(\gamma) \subset \Lcn$ or $\min(\gamma) < m+n < \max(\gamma)$ if $\mathrm{Supp}(\gamma) \subset \Rcn$. Therefore, $\mathrm{Supp}(\gamma_0) \subset \Lcn \cap \rrbracket M-n, M \llbracket$ or $\mathrm{Supp}(\gamma_0) \subset \Rcn \cap \rrbracket m, m+n \llbracket$.

			In conclusion, we showed that for any class modulo $n$ of chains $\Gamma$, there exists a unique $\gamma \in \Gamma$ such that $\gamma \in \mathcal{N}_a(\mathcal{A})$. Therefore, $\mathcal{N}_a(\mathcal{A})$ is a numbering of $\mathcal{A}$.
		\end{proof}

		\begin{lemme} \label{lem: les cycles dans Ma(A) ont L_c et R_c dans un intervalle de longueur n}
			Let $\mathcal{A}$ be a cyclic $c$-noncrossing arc diagram with no loop. Let $\mathcal{N}_a(\mathcal{A})$ be the $a$-numbering of $\mathcal{A}$ for some $a \in \ZZ$. Denote $\mathcal{M}_a(\mathcal{A})$ the set of chains from Definition \ref{def: a-renumérotation} and, if $\mathcal{C}_a(\mathcal{A})$ is finite, let $\mathcal{M}_a(\mathcal{A})$ be the set $\mathcal{N}_a(\mathcal{A})$. Define $m = \min\left(\{\min(\gamma) \;|\: \gamma \in \mathcal{M}_a(\mathcal{A})\}\right)$ and $M = \max\left(\{\max(\gamma) \;|\: \gamma \in \mathcal{M}_a(\mathcal{A})\}\right)$.

			\begin{enumerate}[label*=(\roman*)]
				\item If $\mathcal{C}_a(\mathcal{A})$ is finite or if the period of $\mathcal{C}_a^n(\mathcal{A})$ contains an arrow $\downarrow$, then for all $\gamma \in \mathcal{N}_a(\mathcal{A})$,
				\begin{equation*}
					\mathrm{Supp}(\gamma) \cap \Lcn \subset \rrbracket M - n, M \rrbracket \quad\text{ and } \quad \mathrm{Supp}(\gamma) \cap \Rcn \subset \llbracket m, m+n \llbracket.
				\end{equation*}
				\item Otherwise, for all $\gamma \in \mathcal{M}_a(\mathcal{A})$,
				\begin{equation*}
					\mathrm{Supp}(\gamma) \cap \Lcn \subset \llbracket m, m+n \llbracket \quad\text{ and } \quad \mathrm{Supp}(\gamma) \cap \Rcn \subset \rrbracket M - n, M \rrbracket.
				\end{equation*}
			\end{enumerate}
		\end{lemme}

		\begin{proof}
			We first prove $(i)$ when $\mathcal{C}_a(\mathcal{A})$ is finite: we know that the support of all the chains of $\mathcal{N}_a(\mathcal{A})$ belong to the same interval of length $n$, hence $M - m \leq n$ and the support of any chain of $\mathcal{N}_a(\mathcal{A})$ belongs to both $\rrbracket M-n, M \rrbracket$ and $\llbracket m, m+n \llbracket$, which proves $(i)$ in this case. If $\mathcal{C}_a(\mathcal{A})$ is infinite with the period containing an arrow $\downarrow$, then the chains $\gamma \in \mathcal{L}_a(\mathcal{A}) \cup \mathcal{R}_a(\mathcal{A})$ satisfy $(i)$ by definition.

			Now we prove $(i)$ for when $\mathcal{C}_a(\mathcal{A})$ is infinite (resp. we prove $(ii)$). Let $\tau_a$ be the factor of $\mathcal{C}_a(\mathcal{A})$ defined in Definition \ref{def: a-renumérotation} and suppose it starts with an arrow $\downarrow$ (resp. it contains only arrows $\uparrow$). Note $\gamma_m$ the chain containing $m$ and $\gamma_M$ the chain containing $M$, and $p$ the first integer element of $\tau_a$. We have $p \in \mathrm{Supp}(\gamma_m)$ by Corollary \ref{cor: m appartient à la première chaine} and $M-n \downarrow p$ (resp. $M-n \uparrow p$) is a factor of $\mathcal{C}_a(\mathcal{A})$. By Definition of $\mathcal{C}_a(\mathcal{A})$, we have that $M \in \Lcn$ (resp. $M \in \Rcn$). Moreover, $m \in \Rcn$ (resp. $m \in \Lcn$) since by Lemma \ref{lem: si facteur c|d dans tau_a alors enchevetrement des chaînes} and Corollary \ref{cor: croisement de chaîne} we have
			\begin{equation} \label{eq: preuve lemme les cycles dans Lc et Rc sont dans un intervalle de longueur n}
				\min(\gamma_M) - n < m < M-n < p.
			\end{equation}

			Let $j \in \bigcup_{\gamma \in \mathcal{M}_a(\mathcal{A})} \mathrm{Supp}(\gamma)$. If $j < M-n$, we need to show that $j \in \Rcn$ (resp. $j \in \Lcn$). In this case, either $j = m$ or $m < j < M-n$. In the first case, we already have $j \in \Rcn$ (resp. $j \in \Lcn$). In the second case, we have $\min(\gamma_M - n) < j < M-n$ by (\ref{eq: preuve lemme les cycles dans Lc et Rc sont dans un intervalle de longueur n}). Let $\gamma_j$ be the chain containing $j$. Since $p$ is the first element of $\tau_a$, we have $M-n < p \leq \max(\gamma_j)$. Therefore, we have
			\begin{equation*}
				\min(\gamma_M - n) < j < M-n < \max(\gamma_j)
			\end{equation*} and by Corollary \ref{cor: croisement de chaîne}, we obtain that $j \in \Rcn$ (resp. $j \in \Lcn$). Therefore, $\mathrm{Supp}(\gamma) \cap \Lcn \subset \rrbracket M - n, M \rrbracket$ (resp. $\mathrm{Supp}(\gamma) \cap \Rcn \subset \rrbracket M - n, M \rrbracket$).

			Similarly, if $j > m+n$, we have either $j = M$ or $m+n < j < M$, and by the same argument we obtain that $j \in \Lcn$ (resp. $j \in \Rcn$), and thus $\mathrm{Supp}(\gamma) \cap \Rcn \subset \llbracket m, m+n \llbracket$ (resp. $\mathrm{Supp}(\gamma) \cap \Lcn \subset \llbracket m, m+n \llbracket$). Hence, the other case of $(i)$ (resp. $(ii)$).
		\end{proof}

	\subsection{Generalizing \textsc{Reading}'s bijection} \label{ssec: generalizing reading's bijection}

		In this section, we prove Theorem \ref{thm: generalisation de l'application de Reading} and explicitly provide the inverse map. Most of the work involves the map $\ncad_c$, of which we will also exhibit in Section \ref{sssec: défintion de TITOc} its inverse: it is called $\mathrm{tito}_c$ and associates to each cyclic $c$-noncrossing arc diagram a $c$-sortable TITO on $\ZZ$.

		\subsubsection{Ordering chains and loops} \label{sssec: ordering chains and loops}

			Being able to renumber a cyclic $c$-noncrossing arc diagram allows us to define the inverse map of $\ncad_c$. The construction to obtain the $c$-sortable TITO from a cyclic $c$-noncrossing arc diagram is based on removing arcs one by one, in a specific order, similarly to the computation of the inverse map of \textsc{Reading}'s bijection in type $A$ \cite[§6.1]{gobet2018dual}. This order is defined using the notion of arcs «~hiding~» other arcs.

			\begin{definition} \label{def: chain cachée}
				Let $\mathcal{A}$ be a cyclic $c$-noncrossing arc diagram and $\gamma, \gamma'$ be two distinct partial chains or loops of $\mathcal{A}$. We say that \emph{$\gamma$ hides $\gamma'$} if one of the following cases holds:
				\begin{itemize}
					\item $\gamma$ is a loop and $\mathrm{Supp}(\gamma') \subset \Rcn$,
					\item $\gamma'$ is a loop and $\mathrm{Supp}(\gamma) \subset \Lcn$,
					\item $\gamma$ and $\gamma'$ are partial chains and, by marking an integer with $\underline{\phantom{a}}$ (resp. $\overline{\phantom{a}}$) if it belongs to $\Rcn$ (resp. to $\Lcn$): $\min(\gamma) < \underline{\min(\gamma')} < \max(\gamma')$ or $\min(\gamma') < \overline{\min(\gamma)} < \max(\gamma)$.

				\end{itemize}
			\end{definition}

			Geometrically, a partial chain or a loop hides another one if part of it is located on the left side of the other when traveling from its minimal to its maximal element, while taking into account the numbering. 

			\begin{exemple}
				In the renumbered cyclic $c$-noncrossing arc diagram of Example \ref{ex: a-renumérotation arc}, the chain $11 \rightarrow 20$ hides the chain $8 \rightarrow 12 \rightarrow 13$. Although it seems that the point $8$ of this chain also hides the chain $11 \rightarrow 20$, it is not the case since $8 < 11$ (in fact, $8$ hides the previous copy of $11 \rightarrow 20$ which is $1 \rightarrow 10$). Similarly, the chain $4 \rightarrow 6$ does not hide any chain of the $4$-renumbering of this example.
			\end{exemple}

			Similarly to the construction of \cite[§6.1]{gobet2018dual}, we will select a chain or a loop among all the ones that are not hidden by any other chain or loop. However, it is not trivial to see that this is always possible, especially on the geometric representation of a cyclic $c$-noncrossing arc diagram.

			\begin{lemme} \label{lem: il y a une chaine non cachée}
				Let $\mathcal{N}$ be a nonempty finite subset of $\mathrm{Ch}(\mathcal{A}) \cup \mathrm{L}(\mathcal{A})$. There exists $\gamma_0 \in \mathcal{N}$ such that for all $\gamma \in \mathcal{N} \setminus \{\gamma_0\}$, $\gamma$ does not hide $\gamma_0$. Moreover, if $\gamma_0$ is a loop, it is the only element of $\mathcal{N}$ that is not hidden.
			\end{lemme}

			\begin{proof}
				We explicitly find this $\gamma_0$.
				\begin{itemize}
					\item If $\{ \min(\gamma) \;|\; \gamma \in \mathcal{N} \setminus \mathrm{L}(\mathcal{A}) \} \cap \Lcn \neq \emptyset$, we consider $\gamma_0$ the chain such that its minimal element is the maximal element of this set. In particular, since there is a point of $\gamma_0$ in $\Lcn$, no loop hides $\gamma_0$. Moreover, we have for all $\gamma \in \mathcal{N} \setminus (\mathrm{L}(\mathcal{A}) \cup \{\gamma_0\})$ either $\overline{\min(\gamma_0)} < \underline{\min(\gamma)}$ or $\min(\gamma) < \overline{\min(\gamma_0)}$. By Definition \ref{def: chain cachée}, $\gamma$ does not hide $\gamma_0$.
					\item Otherwise, this means that all the chains have an element in $\Rcn$ (their smallest element). If there is a loop $\widetilde{\gamma}$ in $\mathcal{N}$, then all the chains in $\mathcal{N}$ have all their elements in $\Rcn$. Therefore, $\widetilde{\gamma}$ is not hidden and our candidate is $\gamma_0 = \widetilde{\gamma}$. If there is no loop in $\mathcal{N}$, then we consider the chain $\gamma_0$ such that its minimal element is minimal. We have, for all $\gamma \in \mathcal{N} \setminus \{\gamma_0\}$, $\underline{\min(\gamma_0)} < \underline{\min(\gamma)}$. By Definition \ref{def: chain cachée}, $\gamma_0$ is not hidden by $\gamma$.
				\end{itemize}
				The last point of the Lemma comes from the fact that, by Definition \ref{def: chain cachée}, a loop always hides or is hidden by a chain.
			\end{proof}

			From Definition \ref{def: chain cachée} and Lemma \ref{lem: il y a une chaine non cachée}, we can define a total order on a finite subset of $\mathcal{A}$ iteratively.

			\begin{definition}
				Let $\mathcal{A}$ be a cyclic $c$-noncrossing arc diagram and $\mathcal{N} \subset \mathrm{Ch}(\mathcal{A}) \cup \mathrm{L}(\mathcal{A})$ be a finite subset of chains and loops of $\mathcal{A}$. We define a map $\texttt{select}_\mathcal{N} : \mathcal{P}(\mathcal{N}) \rightarrow \mathcal{N}$, where $\mathcal{P}(\mathcal{N})$ is the power set of $\mathcal{N}$, that selects the chain or loop from $\mathcal{N}$ whose maximal element is minimal among all the chains or loops of $\Gamma$ that are not hidden by any other chain or loop of $\Gamma$ (we state that for a loop $\gamma$, $\max(\gamma) = +\infty$):
				\begin{equation*}
					\forall \Gamma \in \mathcal{P}(\mathcal{N}),\; \texttt{select}_\mathcal{N}(\Gamma) = \mathrm{argmin} \left(\{ \max(\gamma) \;|\; \gamma \in \Gamma \text{ s.t. no chains or loops of } \Gamma \text{ hide } \gamma \}\right).
				\end{equation*}

				This defines a total order on $\mathcal{N}$, named $\lhd_\mathcal{N}$, by iteratively selecting and removing chains or loops:
				\begin{equation*}
				(\mathcal{N}, \lhd_\mathcal{N}) = \{ \gamma_1 \lhd_\mathcal{N} \gamma_2 \lhd_\mathcal{N} \dots \lhd_\mathcal{N} \gamma_k \}
				\end{equation*} if and only if for all $1 \leq i \leq k$, $\gamma_i = \texttt{select}_\mathcal{N}(\mathcal{N} \setminus \{ \gamma_1, \dots, \gamma_{i-1} \})$.
			\end{definition}

			\begin{remarque}
				Note that one could select on the minimal element of $\gamma$, or even any arbitrary element of $\gamma$. Indeed, two chains $\gamma, \gamma'$ that are both not hidden by any other chain or loop satisfy $\max(\gamma) < \min(\gamma')$ or $\max(\gamma') < \min(\gamma)$. We chose to select on the maximal element for similarities with \cite[§6.1]{gobet2018dual}.
			\end{remarque}

			\begin{exemple} \label{ex: ordre de select}
				Consider $\mathcal{N} = \mathcal{N}_2(\mathcal{A})$ the $2$-numbering of $\mathcal{A}$ from Example \ref{ex: renumérotation arc}. The only chain of $\mathcal{N}$ not hidden by any other chain of $\mathcal{N}$ is $14$. Thus $\texttt{select}_{\mathcal{N}}(\mathcal{N}) = 14$. Then $11 \rightarrow 16 \rightarrow 18$ is the only chain not hidden by any other chain of $\mathcal{N} \setminus \{14\}$, thus $\texttt{select}_{\mathcal{N}}(\mathcal{N} \setminus \{14\}) = 11 \rightarrow 16 \rightarrow 18$. Continuing this process, we obtain the following total order on $\mathcal{N}$:
				\begin{equation*}
					14 \lhd_\mathcal{N} 11 \rightarrow 16 \rightarrow 18 \lhd_\mathcal{N} 7 \rightarrow 15 \lhd_\mathcal{N} 9 \rightarrow 12 \rightarrow 13 \lhd_\mathcal{N} 10.
				\end{equation*}
			\end{exemple}

			This order is the key ingredient to define the inverse map of $\ncad_c$.

		\subsubsection{A map from cyclic \texorpdfstring{$c$}{c}-noncrossing arc diagrams to \texorpdfstring{$c$}{c}-sortable TITOs} \label{sssec: défintion de TITOc}

			\begin{definition}[Inverse map of $\ncad_c$] \label{def: TITO associé à un diagramme d'arcs}
				Let $\mathcal{A}$ be a cyclic $c$-noncrossing arc diagram and $a \in \ZZ$. Denote $\mathcal{C}_a(\mathcal{A})$ the sequence defined in Definition \ref{def: suite C_a(x)} and $\mathcal{N}_a(\mathcal{A})$ the $a$-numbering of $\mathcal{A}$ defined in Section \ref{ssec: renumérotation}. We define a TITO on $\ZZ$ from $\mathcal{A}$, named $\mathrm{tito}_{c,a}(\mathcal{A})$, in different ways depending on the nature of $\mathcal{A}$.

				If $\mathcal{C}_a(\mathcal{A})$ is finite or if there exists an arrow $\downarrow$ in the period of $\mathcal{C}_a^n(\mathcal{A})$, consider the totally ordered set $(\mathcal{N}_a(\mathcal{A}), \lhd_{\mathcal{N}_a(\mathcal{A})})$. Let $M$ be the ordered set composed of the supports of each chain of $\mathcal{N}_a(\mathcal{A})$ taken from first to last under the order $\lhd_{\mathcal{N}_a(\mathcal{A})}$, and where each support is ordered from largest to smallest integer. Define $\mathrm{tito}_{c,a}(\mathcal{A})$ to be the TITO of window $[M]$.

				Otherwise, consider the three totally ordered sets $(\mathcal{L}_a(\mathcal{A}), \lhd_{\mathcal{L}_a(\mathcal{A})})$, $(\mathcal{M}_a(\mathcal{A}), \lhd_{\mathcal{M}_a(\mathcal{A})})$ and $(\mathcal{R}_a(\mathcal{A}), \lhd_{\mathcal{R}_a(\mathcal{A})})$ from Definition \ref{def: a-renumérotation}. Let $L$ (resp. $M$, resp. $R$) be the ordered set composed of the supports of each chain of $\mathcal{L}_a(\mathcal{A})$ (resp. $\mathcal{M}_a(\mathcal{A})$, resp. $\mathcal{R}_a(\mathcal{A})$) taken from first to last under the order $\lhd_{\mathcal{L}_a(\mathcal{A})}$ (resp. $\lhd_{\mathcal{M}_a(\mathcal{A})}$, resp. $\lhd_{\mathcal{R}_a(\mathcal{A})}$), and order each support from largest to smallest integer. Define $\mathrm{tito}_{c,a}(\mathcal{A})$ to be the TITO of window $[L]\underline{[M]}[R]$.
			\end{definition}

			As defined, this map $\mathrm{tito}_{c,a}$ depends \textit{a priori} on the choice of $a \in \ZZ$. In fact, it is not the case, but it is not trivial to see that $\mathrm{tito}_{c,a} = \mathrm{tito}_{c,b}$ for $a,b \in \ZZ$ from the definition. We will prove in Propositions \ref{prop: TITO_c est bien définie} and \ref{prop: TITO_c a les bonnes couvertures} that for any $a \in \ZZ$, the map $\mathrm{tito}_{c,a}$ sends a cyclic $c$-noncrossing arc diagram $\mathcal{A}$ to a $c$-sortable TITO on $\ZZ$ that has as covers exactly the affine transpositions $\affc{p,q}$ for $p \rightarrow q \in \mathcal{A}$. By Proposition \ref{prop: les couvertures d'un TITO c-triable le caractérisent}, such a TITO is unique.

			\begin{exemple}
				Let $\mathcal{A}$ be the cyclic $c$-noncrossing arc diagram of Example \ref{ex: C_a(A)} and its $2$-numbering from Example \ref{ex: renumérotation arc}. We have an arrow $\downarrow$ in $\tau_2$, therefore we have $\mathrm{tito}_{c,2}(\mathcal{A})$ of shape $[M]$ where $M$ is the ordered set composed of the supports of each chain in $\mathcal{N}$ computed in Example \ref{ex: ordre de select}. Thus,
				\begin{equation*}
					\mathrm{tito}_{c,2}(\mathcal{A}) = [ 14,\ 18, 16, 11,\ 15, 7,\ 13, 12, 9,\ 10 ].
				\end{equation*}
			\end{exemple}
			\begin{exemple} \label{ex: un exemple de TITO_c}
				If $\mathcal{B}$ is the cyclic $c$-noncrossing arc diagram of Example \ref{ex: a-renumérotation arc}, then no $\downarrow$ appears in a periodic factor of $\mathcal{C}_4(\mathcal{B})$. Therefore, $\mathrm{tito}_{c,4}(\mathcal{B})$ is of shape $[L]\underline{[M]}[R]$. We have $L = (6,4)$, $M = (20, 11,\ 13, 12, 8)$ and $R = (-3, -5,\ -1)$ using the $\texttt{select}$ function on $\mathcal{L}_4(\mathcal{B})$, $\mathcal{M}_4(\mathcal{B})$ and $\mathcal{R}_4(\mathcal{B})$ respectively. Therefore:
				\begin{equation*}
					\mathrm{tito}_{c,4}(\mathcal{B}) = [6,4] \underline{[20, 11,\ 13, 12, 8]} [-3,-5,\ -1].
				\end{equation*}
			\end{exemple}

			We now have all the tools to prove Theorem \ref{thm: generalisation de l'application de Reading}. The following two propositions contain the major part of the proof as they state that the map $\mathrm{tito}_{c,a}$, for any $a \in \ZZ$, is well defined and is the inverse map of $\ncad_c$.

			Let $\mathcal{A}$ be a cyclic $c$-noncrossing arc diagram and $a \in \ZZ$.
			\begin{proposition}\label{prop: TITO_c est bien définie}
				The TITO $\mathrm{tito}_{c,a}(\mathcal{A})$ is well defined and is a $c$-sortable TITO on $\ZZ$.
			\end{proposition}

			\begin{proof}
				Since $\mathcal{N}_a(\mathcal{A})$ is a numbering of $\mathcal{A}$, the sets $L$, $M$ and $R$ define a subset of $\ZZ$ in which exactly one element of each class of integers modulo $n$ is present, thus the windows $[M]$ and $[L]\underline{[M]}[R]$ do well define TITOs.

				We now check that $\mathrm{tito}_{c,a}(\mathcal{A})$ is a $c$-sortable TITO on $\ZZ$. By definition, it has the right shape: either $[M]$ or $[L]\underline{[M]}[R]$ with $L$, $M$ and $R$ like in Definition \ref{def: biclos c-triable}. Using Proposition \ref{prop: TITO c triable motifs plus facile à distance finie}, it only remains to check that the patterns $ki\dots j$ with $j \in \Lcn$ and $j \dots ki$ with $j \in \Rcn$ such that $i < j < k$ and $k$ and $i$ are consecutive in the TITO, are avoided.

				Let $i,j,k \in \ZZ$ such that $\affc{i,k} \in \mathrm{Cov}(\mathrm{tito}_{c,a}(\mathcal{A}))$ and $i < j < k$. We show that $j \prec k$ if $j \in \Lcn$ and $i \prec j$ if $j \in \Rcn$. Note $\gamma$ the chain or loop containing the arc $i \rightarrow k$ in $\mathcal{A}$ and $\gamma'$ the chain or loop containing $j$. Since $k$ and $i$ are consecutive in the TITO, we have $\gamma \neq \gamma'$. Without loss of generality, we can assume that $\gamma \in \mathcal{N}_a(\mathcal{A})$ by shifting $i$, $j$ and $k$ by the same multiple of $n$.
				Moreover, let $j'$ be the representative modulo $n$ of $j$ that appears in the $a$-numbering $\mathcal{N}_a(\mathcal{A})$.

				First, we treat the cases where $\{i,k\}$ and $\{j\}$ belong to different blocks of $\mathrm{tito}_{c,a}(\mathcal{A})$. If $j \in [L]$ and $i,k$ lie in another block, then $j \in L_c$ and $j \prec k$. Similarly, if $j \in [R]$ and $i,k$ lie in another block, then $j \in R_c$ and $i \prec j$. Finally, if $j \in \underline{[M]}$ and $i,k$ are in $[L]$ (resp. $[R]$), then $j$ must belong in $\Rcn$ (resp. $\Lcn$): indeed, assume for contradiction that $j \in \Lcn$ (resp. $j \in \Rcn$), then $\gamma'$ would be confined between $i$ and $k$ by Corollary \ref{cor: croisement de chaîne}, and take values only in $\Lcn$ (resp. $\Rcn$). If $\gamma'$ is a loop, this is already a contradiction. Otherwise, since $j \in \underline{[M]}$, then $\gamma'$ is a chain that intersects a periodic factor of the sequence $\mathcal{C}_a(\mathcal{A})$. In particular, $\gamma'$ contains the neighbor arc of the maximal element of a chain $\eta$ that also appears in a periodic factor. Since all the elements of $\gamma'$ are in $\Lcn$ (resp. $\Rcn$), $\gamma'$ contains the neighbor arc of $\max(\eta)$ and $i \rightarrow k \in \gamma$ is a covering arc of $\max(\eta)$ such that $\max(\gamma') < k$, then $\max(\eta) \in \Lcn$ (resp. $\max(\eta) \in \Rcn$). Therefore $\eta$ is confined between $\min(\gamma')$ and $\max(\gamma')$ by Corollary \ref{cor: croisement de chaîne}, but this contradicts Lemma \ref{lem: si facteur c|d dans tau_a alors enchevetrement des chaînes}. Hence, $j \in \Rcn$ and $i \prec j$ (resp. $j \in \Lcn$ and $j \prec k$).

				Secondly, we treat the case where $i,j$ and $k$ belong to the same block.
				In this case, neither $\gamma$ nor $\gamma'$ are loops as otherwise they would be equal.
				If $j = j'$, meaning that $\gamma' \in \mathcal{N}_a(\mathcal{A})$, then $\gamma'$ hides $\gamma$ if $j \in \Lcn$ and $\gamma$ hides $\gamma'$ if $j \in \Rcn$ by Definition \ref{def: chain cachée}, because $i < j < k$. Therefore, if $\lhd$ is the total order defined by the \texttt{select} function (Definition \ref{def: TITO associé à un diagramme d'arcs}), we have $\gamma' \lhd \gamma$ if $j \in \Lcn$ and $\gamma \lhd \gamma'$ if $j \in \Rcn$. In particular, $j$ appears before $k$ if $j \in \Lcn$ and $j$ appears after $i$ if $j \in \Rcn$.
				This case ($j = j'$) already fully covers the cases where $i$, $j$ and $k$ belong to a block $[L]$ or a block $[R]$. Indeed, the union of the supports of all the chains in $\mathcal{L}_a(\mathcal{A})$ or all the chains in $\mathcal{R}_a(\mathcal{A})$ is contained within an interval of length $n$, and since $i < j < k$, then $j = j'$.

				From now on, we suppose that $j \neq j'$ and we finish our proof by exploring the cases where they belong to either a block $[M]$ or a block $\underline{[M]}$.
				We have by Lemma \ref{lem: les cycles dans Ma(A) ont L_c et R_c dans un intervalle de longueur n} that there exist two integers $m < M$ such that $m \leq i < k \leq M$ and one of the two following cases is true:
				\begin{enumerate}[label*=(\alph*)]
					\item if $j \in \Lcn$ and the block is $\underline{[M]}$, or if $j \in \Rcn$ and the block is $[M]$, then $j' \in \llbracket m, m+n \llbracket$,
					\item if $j \in \Rcn$ and the block is $\underline{[M]}$, or if $j \in \Lcn$ and the block is $[M]$, then $j' \in \rrbracket M-n, M \rrbracket$.
				\end{enumerate}
				Since $i < j < k$, we have $m < j < M$ and therefore $j' < j$ in case (a) and $j < j'$ in case (b).
				If the block is $\underline{[M]}$, this means that $j$ appears before this window (\textit{i.e.} $j \prec x$ for all $x \in \underline{[M]}$) if $j \in \Lcn$, and after this window if $j \in \Rcn$; and if the block is $[M]$, this means that $j$ appears after this window if $j \in \Rcn$, and before this window if $j \in \Lcn$.
				In both cases, we have that $j \prec k$ if $j \in \Lcn$ and $i \prec j$ if $j \in \Rcn$.

				This concludes the proof that $\mathrm{tito}_{c,a}(\mathcal{A})$ is a $c$-sortable TITO on $\ZZ$.
			\end{proof}

			\begin{proposition} \label{prop: TITO_c a les bonnes couvertures}
				For all $p,q \in \ZZ$ such that $p < q$, the TITO $\mathrm{tito}_{c,a}(\mathcal{A})$ satisfies:
				\begin{equation*}
					p \rightarrow q \in \mathcal{A} \Leftrightarrow \affc{p,q} \in \mathrm{Cov}(\mathrm{tito}_{c,a}(\mathcal{A})).
				\end{equation*}
			\end{proposition}

			\begin{proof}
				The forward implication $\Rightarrow$ is clear by construction, since an arrow belongs to a chain or a loop, and their support appears as a consecutive sequence of decreasing order in the TITO.
				The backwards implication $\Leftarrow$ is more complicated to prove, and we prove it by contradiction: suppose that there exists $p,q \in \ZZ$ with $p < q$ such that $\affc{p,q} \in \mathrm{Cov}(\mathrm{tito}_{c,a}(\mathcal{A}))$ and $p \rightarrow q \notin \mathcal{A}$. Up to replacing $p \rightarrow q$ by $p + kn \rightarrow q + kn$ for some integer $k \in \ZZ$, we can assume that $q$ appears in the window notation of $\mathrm{tito}_{c,a}(\mathcal{A})$ obtained by definition of the map $\mathrm{tito}_{c,a}$.
				Since $q$ and $p$ are consecutive for the order ${\prec} = \mathrm{tito}_{c,a}(\mathcal{A})$, by construction they either are consecutive elements of the same chain or loop, or they belong to different chains or loops. Only the second case is possible, because we suppose that $p \rightarrow q \notin \mathcal{A}$. Thus, denoting $\gamma_p$ (resp. $\gamma_q$) the chain or loop containing $p$ (resp. $q$), we have $\gamma_p \neq \gamma_q$. Moreover, neither $\gamma_p$ nor $\gamma_q$ is a loop because they belong to the same block, and a loop is always alone in its block. Also, we have that $p = \max(\gamma_p)$ and $q = \min(\gamma_q)$.
				\begin{equation*}
					\mathrm{tito}_{c,a}(\mathcal{A}) = \cdots \prec \underbrace{\max(\gamma_q) \prec \dots \prec q}_{\gamma_q} \prec \underbrace{p \prec \dots \prec \min(\gamma_p)}_{\gamma_p} \prec \cdots
				\end{equation*}%

				We begin by showing that $q$ is necessarily the last element of the window of the block it belongs to. Suppose otherwise by contradiction, then $p$ also appears in this window. Then $\gamma_q \lhd \gamma_p$ are consecutive for the order $\lhd$ obtained from the $\texttt{select}$ function. This means that either $\gamma_p$ is hidden by $\gamma_q$, or that $\max(\gamma_q) < \max(\gamma_p)$. Since $\max(\gamma_p) = p < q = \min(\gamma_q) < \max(\gamma_q)$, both cases are impossible (see Definition \ref{def: chain cachée}).

				Thus $q$ is the last element of the window of its block. Therefore $p - n$ (resp. $p + n$) is the first element of the same block if it is waxing (resp. waning). If this block is $[L]$ or $[R]$ (from the second case of Definition \ref{def: TITO associé à un diagramme d'arcs}), it is a waxing block and we know that all the elements of $L$ or $R$ belong to the same interval of length $n$. But $q - (p - n) = q-p+n > n$, so this is absurd. Therefore, $p$ and $q$ belong either to a block $[M]$ (first case of Definition \ref{def: TITO associé à un diagramme d'arcs}) or a block $\underline{[M]}$ (second case of the definition). We treat each case separately.
				\begin{equation*}
					[M] = [\overbrace{p-n, \dots, \min(\gamma_p) - n}^{\gamma_p - n}, \dots, \overbrace{\max(\gamma_q), \dots, q}^{\gamma_q}]; \quad \underline{[M]} = \underline{[\overbrace{p+n, \dots, \min(\gamma_p) + n}^{\gamma_p + n}, \dots, \overbrace{\max(\gamma_q), \dots, q}^{\gamma_q}]}
				\end{equation*}

				\begin{itemize}
					\item If this block is $[M]$, it is a waxing block, and this means that either $\mathcal{C}_a(\mathcal{A})$ is finite, or the periodic factor $\tau_a$ of $\mathcal{C}_a(\mathcal{A})$ defined in Definition \ref{def: a-renumérotation} contains an arrow $\downarrow$. In the first case, we know that all elements of $M$ belong to the same interval of length $n$, and we conclude similarly as in the previous cases of the blocks $[L]$ and $[R]$.

					In the second case, recall from Definition \ref{def: a-renumérotation} that there exists $m,M \in \ZZ$ such that the chains $\gamma$ in $\mathcal{L}_a(\mathcal{A})$ (resp. $\mathcal{M}_a(\mathcal{A})$, resp. $\mathcal{R}_a(\mathcal{A})$) have their support in $\rrbracket M-n, M \rrbracket$ (resp. $\llbracket m,M \rrbracket$, resp. $\llbracket m,m+n \llbracket$). Let $y = \min(\{ z \in \tau_a \cap \ZZ \;|\; z > p-n \})$, which exists since $p-n < M$ and $M \in \tau_a$, and denote $\gamma_y$ the chain containing $y$. Since $p-n$ is the maximal element of $\gamma_p - n$, then $\gamma_p - n \neq \gamma_y$. Therefore, either $p-n < \min(\gamma_y) = y$ or $\min(\gamma_y) < p-n < y$ by definition of $y$. But since $y$ is in $\tau_a$, it is not the minimal element of its chain, therefore only the second case holds.

					Since the elements of $\gamma_p - n$ appear first in the window $[M]$, then $\gamma_p - n$ is not hidden by any other chain of $\mathcal{N}_a(\mathcal{A})$. In particular, $\gamma_y$ doesn't hide $\gamma_p - n$, and by Definition \ref{def: chain cachée} we obtain that $p - n \in \Lcn$. In particular, $\gamma_p - n \notin \mathcal{R}_a(\mathcal{A})$.

					If $\gamma_p - n \in \mathcal{L}_a(\mathcal{A})$, then we know that $p-n \in \rrbracket M-n, M \rrbracket$, hence $M < p$. Recall that $p < q$, therefore $M < q$: this is absurd since $q \in \gamma_q \in \mathcal{N}_a(\mathcal{A})$, thus $q \in \llbracket m,M \rrbracket$.

					Finally, if $\gamma_p - n \in \mathcal{M}_a(\mathcal{A})$, then $p-n$ appears in $\tau_a$ since it is the maximal element of $\gamma_p - n$. Recall that $M$ also appears in $\tau_a$ (in fact, it is its last element). In the sequence $\mathcal{C}_a(\mathcal{A})$, which is a concatenation of the factors $\tau_a + kn$ for $k \in \NN$ from a certain point, the integer $p-n$ appears before $M$ (from the factor $\tau_a$), which in turn appears before $p$ (from the factor $\tau_a + n$). Since the integers in $\mathcal{C}_a(\mathcal{A})$ appear in increasing order by construction, we have that $M < p$ which is once again absurd.

					Therefore, the block containing $p$ and $q$ cannot be $[M]$.

					\item The only option left is that the block containing $p$ and $q$ is $\underline{[M]}$, hence there are no arrows $\downarrow$ in $\tau_a$.
					If $\gamma_q = \gamma_p + n$, because the window $\underline{[M]}$ starts with $p+n$ and ends with $q$, we would have $\mathcal{M}_a(\mathcal{A}) = \{\gamma_q\}$. And since $q > p$, $\gamma_q$ would be a loop and so $p \rightarrow q \in \mathcal{A}$ which is absurd.

					Hence $\gamma_p + n \neq \gamma_q$ and both belong to $\mathcal{M}_a(\mathcal{A})$ so neither of them are loops, and so $\gamma_p+n$ is hidden by no other chains from $\mathcal{M}_a(\mathcal{A})$ and no chain of $\mathcal{M}_a(\mathcal{A})$ is hidden by $\gamma_q$. By Lemma \ref{lem: si facteur c|d dans tau_a alors enchevetrement des chaînes} and Definition \ref{def: chain cachée}, for each factor $c \uparrow d$ in $\tau_a$ with $c,d \in \ZZ$, the chain containing $d$ hides the chain containing $c$. Thus, $\tau_a$ starts with elements of $\gamma_q$ and ends with elements of $\gamma_p+n$. In particular, $\gamma_q$ contains the neighbor arc of $p$. Thus, $q = \min(\gamma_q) < p$, which is absurd since $p < q$.
					\begin{equation*}
						\mathcal{C}_a(\mathcal{A}) = \cdots \uparrow \dots \rightarrow p \overbrace{{\uparrow} \underbrace{\dots \rightarrow \max(\gamma_q)}_{\in \gamma_q} \uparrow \cdots \uparrow \underbrace{\dots \rightarrow p+n}_{\in \gamma_p + n}}^{\tau_a} \uparrow \dots \rightarrow \max(\gamma_q) + n \uparrow \cdots
					\end{equation*}
				\end{itemize}
				In conclusion, $q$ does not belong to any block of $\mathrm{tito}_{c,a}(\mathcal{A})$, which is absurd. Therefore, the arc $p \rightarrow q$ exists in $\mathcal{A}$, hence the proposition is proven.
			\end{proof}

			In particular, Propositions \ref{prop: TITO_c est bien définie} and \ref{prop: TITO_c a les bonnes couvertures}, combined with Proposition \ref{prop: les couvertures d'un TITO c-triable le caractérisent}, show that $\mathrm{tito}_{c,a} = \mathrm{tito}_{c,b}$ for all $a,b \in \ZZ$. We denote this common map $\mathrm{tito}_c$. The window notation of the obtained TITO on $\ZZ$ may (and in most cases, will) be different depending on the choice of $a \in \ZZ$, but the underlying TITO is the same. See Example \ref{ex: TITOca = TITOcb} for illustrations.

			\begin{corollaire} \label{cor: ncadc est bijective d'inverse TITOc}
				The map $\ncad_c$ is a bijection between the set of $c$-sortable TITOs on $\ZZ$ and the set of cyclic $c$-noncrossing arc diagrams. Moreover, $\ncad_c^{-1} = \mathrm{tito}_c$.
			\end{corollaire}

			\begin{proof}
				We already know that $\ncad_c$ is injective. Let $\mathcal{A}$ be a cyclic $c$-noncrossing arc diagram and $\prec = \mathrm{tito}_c(\mathcal{A})$. By Proposition \ref{prop: TITO_c est bien définie} it is a $c$-sortable TITO on $\ZZ$, and by Proposition \ref{prop: TITO_c a les bonnes couvertures} we have
				\begin{equation*}
					p \rightarrow q \in \mathcal{A} \Leftrightarrow \affc{p,q} \in \mathrm{Cov}(\prec).
				\end{equation*}
				By construction, $\ncad_c(\prec)$ is the cyclic $c$-noncrossinc arc diagram whose arcs are exactly the arcs $p \rightarrow q$ for $\affc{p,q} \in \mathrm{Cov}(\prec)$. Therefore, $\ncad_c(\prec) = \mathcal{A}$ and $\ncad_c$ is bijective of inverse $\mathrm{tito}_c$.
			\end{proof}

		\subsubsection{Generalized \textsc{Reading}'s map} \label{sssec: generalized reading's map}

			As a consequence of Corollary \ref{cor: ncadc est bijective d'inverse TITOc}, the map $nc_c^{\widetilde{A}}$ defined in Definition \ref{def: reading généralisée} is also a bijection by composition of bijections (Equation \ref{eq: nc_c exprimé comme composée}). What is left to prove Theorem \ref{thm: generalisation de l'application de Reading} is to show that this map generalizes \textsc{Reading}'s map $nc_c$.

			\readinggeneralise*

			\begin{proof}
				Corollary \ref{cor: ncadc est bijective d'inverse TITOc} and Equation \ref{eq: nc_c exprimé comme composée} already show that $nc_c^{\widetilde{A}}$ is a bijective map from the set of $c$-sortable biclosed sets of $R$ to the set of $c$-noncrossing partitions of $W$. We need to show that it coincides with the map $nc_c$ on inversions sets of $c$-sortable elements of $W$.

				Let $w \in W$ be a $c$-sortable element and define two $c$-noncrossing partitions $x = nc_c(w)$ and $y = nc_c^{\widetilde{A}}(N(w))$. By definition of $nc_c$ \cite[§8]{reading2010sortable}, $x$ is a $R$-reduced product of the reflections of $\mathrm{Cov}(w)$ in a specific order. By Proposition \ref{prop: TITO_c a les bonnes couvertures} and Definition \ref{def: reading généralisée}, we have that $y$ is also a product of the elements of $\mathrm{Cov}(\pprec(N(w))) = \mathrm{Cov}(w)$, possibly in another order. Recall from Paragraph \ref{par: partitions non croisées Sn tilde} of Section \ref{sec: background} that no $c$-noncrossing partition distinct from $x$ can be obtained by permuting the letters of an $R$-reduced expression of $x$, therefore $y = x$.
			\end{proof}

			\begin{exemple}[Direct map] \label{ex: direct map}
				Consider, for $n = 4$ and $c = \affc{1,2,3}_{[1]}\affc{0}_{[-1]}$, the biclosed set
				\begin{equation*}
					B = \{ \affc{0,1}_p, \affc{0,2}_p, \affc{3,1}_{p+1}, \affc{3,2}_{p+1} \;|\; p \in \NN \} \sqcup \{ \affc{0,3}_p, \affc{3,0}_{p+1} \;|\; p \in \NN \}.
				\end{equation*}
				Its corresponding TITO on $\ZZ$ is $\pprec(B) = [1,2]\underline{[3,0]}$. Applying Definition \ref{def: reading généralisée}, we compute the maximal descending chains of $\pprec(B)$. They are $(1)$, $(2)$ (in the block $[L]$) and $(3,0,-1)$ (in the block $\underline{[M]}$). Hence, the image of $B$ is:
				\begin{equation*}
					x = nc_c^{\widetilde{A}}(B) = \affc{1} \affc{2} \affc{3}_{[1]} \affc{0}_{[-1]} = s_0s_1s_2s_1s_0 s_3.
				\end{equation*}
			\end{exemple}

			\begin{exemple}[Inverse map] \label{ex: TITOca = TITOcb}
				In Example \ref{ex: un exemple de TITO_c}, we computed ${\prec} = \mathrm{tito}_{c,4}(\mathcal{B})$ where $\mathcal{B}$ is the cyclic $c$-noncrossing arc diagram of Example \ref{ex: a-renumérotation arc}. This TITO is indeed $c$-sortable, and moreover $\ncad_c(\prec) = \mathcal{B}$. If we use another initial point for the renumbering of $\mathcal{B}$, say the $-1$-numbering, we obtain that
				\begin{equation*}
					\mathcal{C}_{-1}(\mathcal{B}) = -1 \uparrow 0 \uparrow 2 \rightarrow 3 \uparrow 10 \uparrow 12 \rightarrow 13 \uparrow 20 \uparrow \cdots.
				\end{equation*}
				Therefore, $\tau_{-1} = \uparrow 0 \uparrow 2 \rightarrow 3$, and since it does not contain $\downarrow$s, the $-1$-numbering of $\mathcal{B}$ is depicted in Figure \ref{fig: renumérotation avec autre point} and hence we have
				\begin{equation*}
					\mathrm{tito}_{c,-1}(\mathcal{B}) = [-4, -6] \underline{[3,2,-2,\ 0,-9]} [-3,-5,\ -1].
				\end{equation*}
				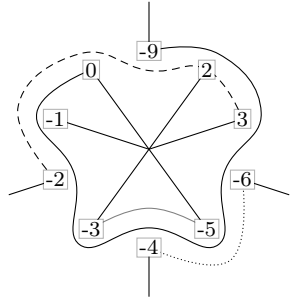
\begin{figure}[ht!]
					\centering
					\begin{tikzpicture}[baseline=(current bounding box.center), scale=1.3]
						\def\n{10}

						\def\Rc{2, 3, 5, 7, 9, 10}
						\def\Lc{1, 4, 6, 8}
						\def\NUM{1/-9, 2/2, 3/3, 4/-6, 5/-5, 6/-4, 7/-3, 8/-2, 9/-1, 10/0}

						\newcommand{\ang}[1]{{-(#1-1)*360/\n + 90}}

						\foreach \i in {1,...,\n}{
							\coordinate (\i) at (\ang{\i}:1);
						}

						\foreach \i in \Rc{
							\draw (0,0) -- (\i);
						}
						\foreach \i in \Lc{
							\draw (\ang{\i}:1.5) -- (\i);
						}

						\draw[black] plot [ smooth, tension=0.8 ] coordinates {(1) (\ang{2}:1.2) (\ang{3}:1.2) (\ang{4}:0.8) (\ang{5}:1.2) (\ang{6}:0.8) (\ang{7}:1.2) (\ang{8}:0.8) (\ang{9}:1.2) (10)};

						\draw[black, densely dashed] plot [ smooth, tension=0.8 ] coordinates {(8) (\ang{9}:1.4) (\ang{10}:1.2) (\ang{1}:0.8) (2) (3)};

						\draw[black, densely dotted] plot [ smooth, tension=0.8 ] coordinates {(4) (\ang{5}:1.4) (6)};

						\draw[gray] plot [ smooth, tension=0.8 ] coordinates {(5) (\ang{6}:0.6) (7)};

						\foreach \i/\j in \NUM{
							\node[rectangle, fill=white, inner sep=1pt, draw, font=\footnotesize, draw opacity=0.3] at (\i) {\j};
						}
					\end{tikzpicture}
					\caption{The $-1$-numbering of $\mathcal{B}$}
					\label{fig: renumérotation avec autre point}
				\end{figure}%

				The window is different from $\mathrm{tito}_{c,4}(\mathcal{B})$ computed in Example \ref{ex: un exemple de TITO_c}, but they define the same TITO since we can obtain the window of $\mathrm{tito}_{c,4}(\mathcal{B})$ from the window of $\mathrm{tito}_{c,-1}(\mathcal{B})$ by «~sliding~» the first window two steps to the right and the second window seven steps to the left.

			\end{exemple}

	\subsection*{Remarks and questions}
		\subparagraph*{Type $\widetilde{C}$} Under the classical folding of type $\widetilde{A}_{2n-3}$ into type $\widetilde{C}_{n-1}$, the objects used in this paper are well behaved (see \cite{barkley2022combinatorial, reading2024noncrossingpartitionsannulus}) and our results naturally restrict to this type. The argument is beyond the scope of this paper and will appear in the author's PhD thesis.

		\subparagraph*{Other classical affine types} Most combinatorial objects involved in this paper have generalizations in other classical affine types. See \cite{barkley2022combinatorial} for the combinatorial description of biclosed sets as subsets of TITOs on $\ZZ$ in types $\widetilde{A}$, $\widetilde{B}$, $\widetilde{C}$ and $\widetilde{D}$. See  \cite{reading2024noncrossingpartitionsannulus} for a combinatorial model of noncrossing partitions in types $\widetilde{A}$ and $\widetilde{C}$ and \cite{reading2025symmetricnoncrossingpartitionsannulus} for types $\widetilde{B}$ and $\widetilde{D}$. We think the methodology used in this paper could be adapted on a type by type basis in order to obtain similar results in the other classical affine types, but we have little faith that this method could lead to a uniform approach.

		\subparagraph*{Generalization of aligned elements} It would be interesting to generalize \textsc{Reading}'s result stating that $c$-sortable elements and $c$-aligned elements are the same in all \textsc{Coxeter} groups \cite[4.3]{reading2010sortable}. A generalized notion of $c$-aligned elements could lead to a uniform generalization of $c$-sortable elements and \textsc{Reading}'s bijection in all affine types (and potentially all \textsc{Coxeter} groups). 



\clearpage

\phantomsection
\addcontentsline{toc}{section}{References}
\bibliographystyle{alpha}
\bibliography{biblio.bib}

\end{document}